%% file: yanezetalfin.tex
\documentclass{elsarticle}
\usepackage{lipsum}
\makeatletter
\def\ps@pprintTitle{%
 \let\@oddhead\@empty
 \let\@evenhead\@empty
 \def\@oddfoot{}%
 \let\@evenfoot\@oddfoot}
\makeatother

\def\vv#1#2#3{#1^{(#3)}_{#2}}
\usepackage{graphicx}
\usepackage{hyperref}
\usepackage{amsfonts}
\usepackage{pgf,tikz,pgfplots}
\usepackage{epstopdf}
\usepackage{amssymb,amsmath}
\usepackage{float}
\usepackage[displaymath, mathlines,running]{lineno}
\usepackage{mathtools} % Bonus
\usepackage{color}
\usepackage{blkarray}
\usepackage{enumitem}
\usepackage{fmtcount}
\usepackage{multirow}

\def\dioni#1#2{\textcolor{red}{#1}{#2}}

\newenvironment{proof}{\smallskip\noindent{{\it Proof.}}\hskip \labelsep}%
            {\hfill\penalty10000\raisebox{-.09em}{\large\bf\rm $\blacksquare$}\par\medskip}

\newtheorem{theorem}{Theorem}[section]
\newtheorem{lemma}[theorem]{Lemma}
\newtheorem{proposition}[theorem]{Proposition}
\newtheorem{corollary}[theorem]{Corollary}
\newtheorem{definition}{Definition}

\newtheorem{remark}{Remark}[section]

\def\xx{\mathbf{x}}
\def\xe{\mathbf{x}^*}
\def\x12e{\mathbf{x}_{\frac12}^*}
\begin{document}
%%-----------------------------
%%      the top matter
%%-----------------------------
\begin{frontmatter}

\title{A non-separable progressive multivariate WENO-$2r$ point value}\tnotetext[label1]{The second author has been supported through project 20928/PI/18 (Proyecto financiado por la Comunidad Aut\'onoma de la Regi\'on de Murcia a trav\'es de la convocatoria de Ayudas a proyectos para el desarrollo de investigaci\'on cient\'ifica y t\'ecnica por grupos competitivos, incluida en el Programa Regional de Fomento de la Investigaci\'on Cient\'ifica y T\'ecnica (Plan de Actuaci\'on 2018) de la Fundaci\'on S\'eneca-Agencia de Ciencia y Tecnolog\'ia de la Regi\'on de Murcia) and by the national research project PID2019-108336GB-I00. The first and fourth authors have been supported through project CIAICO/2021/227 and by grant PID2020-117211GB-I00 funded by MCIN/AEI/10.13039/501100011033. The third author has been supported by NSF grant DMS-2309249.}

%\author[UPCT]{Sergio Amat}
%\ead{sergio.amat@upct.es}
\author[UV]{Pep Mulet}
\ead{pep.mulet@uv.es}
\author[UPCT]{Juan Ruiz-\'Alvarez}
\ead{juan.ruiz@upct.es}
\author[UB]{Chi-Wang Shu}
\ead{chi-wang\_shu@brown.edu}
\author[UV]{Dionisio F. Y\'a\~nez}
\ead{dionisio.yanez@uv.es}
\date{Received: date / Accepted: date}

\address[UV]{Departamento de Matem\'aticas, Facultad de Matemáticas, Universidad de Valencia, Valencia, Spain}
\address[UPCT]{Departamento de Matemática Aplicada y Estadística, Universidad  Polit\'ecnica de Cartagena, Cartagena, Spain}
\address[UB]{Division of Applied Mathematics, Brown University, Providence, Rhode Island, USA}
\begin{abstract}
The weighted essentially non-oscillatory  {technique} using a stencil of $2r$ points (WENO-$2r$) is an interpolatory method that consists in obtaining a higher approximation order from the non-linear combination of interpolants of $r+1$ nodes. The result is an interpolant of order $2r$ at the smooth parts and order $r+1$ when an isolated discontinuity falls at any grid interval of the large stencil except at the central one. Recently, a new WENO method based on Aitken-Neville's algorithm has been designed for interpolation of equally spaced data at the mid-points and presents progressive order of
accuracy close to discontinuities. This paper is devoted to constructing
a general progressive WENO method for non-necessarily uniformly spaced data and
several variables interpolating in any point of the central
interval. Also, we provide explicit formulas for linear and non-linear
weights and prove  the order obtained. Finally, some numerical
experiments are presented to check the theoretical results.
\end{abstract}

%
%\begin{resume} ... \end{resume}
%

\end{frontmatter}

%%-----------------------------
%%      your text
%%-----------------------------

\def\S{\mathcal{S}}

\section{Introduction and review}

In the last years, WENO methods have been developed and used in several
applications, mainly to obtain numerical solutions of partial
differential equations \dioni{}{(PDEs)} but also in other fields, such as image processing or
computer-aided design (see e.g. \cite{ABM,LPR99,Liu}). The idea  is to
compute a non-linear combination of interpolations through polynomials of
degree $r$, aiming to obtain the maximum possible order $2r$ when the data is
free of discontinuities, and order $r+1$ at the non-smooth parts. In what follows, we
briefly review the method.
Let us denote as $X$ a uniform partition of the interval
$[a, b]$ in $J$ subintervals:
$$X=\{x_i\}^{J}_{i=0}, \quad x_i=a+i\cdot h, \quad h=\frac{b-a}{J},$$
and consider the point-value discretization of a piecewise smooth function $f$ at the nodes $x_i$,
\begin{equation*}
f_i=f(x_i),\, i=0,\hdots,J, \quad f=\left\{f_i\right\}_{i=0}^{J}.
\end{equation*}
In this setting, the WENO method with $2r$ nodes interpolates  at the mid-point of the interval $(x_{i-1}, x_i)$, denoted by $x_{i-\frac12}$,  using the stencil $\mathcal{S}^{2r}_0=\{x_{i-r},\cdots, x_{i+r-1}\}$. We construct this interpolation by means of the following convex combination:
\begin{equation}\label{PO}
\mathcal{I}^{2r-1}(x_{i-\frac{1}{2}};f)=\sum_{k=0}^{r-1}
\omega_{k}^r p_{k}^r(x_{i-\frac{1}{2}}),
\end{equation}
where $\omega_{k}^r\ge 0, \, k=0,\cdots, r-1$ are non-linear
 {(data-dependent)} positive weights such that
$\sum_{k=0}^{r-1}\omega_{k}^r=1$, and $p_{k}^r$ are the Lagrange interpolants with nodes $\mathcal{S}^r_k=\{x_{i-r+k},\cdots, x_{i+k}\}$.

The values of the weights $\omega_{k}^r$ are designed to obtain an order of
accuracy $2r$ at $x_{i-\frac{1}{2}}$ when the function is smooth in
the large stencil, as follows: There are optimal weights $C_{k}^r\ge 0$, with $k=0,\hdots,r-1$ that satisfy the following equality:
\begin{equation*}
p_{0}^{2r-1}\left(x_{i-\frac{1}{2}}\right)
=\sum_{k=0}^{r-1}{C}_{k}^r p_{k}^r\left(x_{i-\frac{1}{2}}\right).
\end{equation*}
A formula for these values is obtained in \cite{ABM}:
\begin{equation}\label{opt_w}
{C}_{k}^r=\frac{1}{2^{2r-1}}
%\begin{array}{l}
\binom{2r}{2k+1}
%\end{array}
,  \quad k=0,\cdots, r-1.
\end{equation}
As stated in \cite{Liu}, the weights $\omega_{k}^r$  satisfy
\begin{equation}\label{omega}
\omega_{k}^r={C}_{k}^r+O(h^{{\kappa}}),\quad k=0,\cdots, r-1,
\end{equation}
at smooth zones, with $\kappa\ge r-1$, thus assuring that the interpolation in (\ref{PO}) attains order of accuracy $2r$
\begin{equation*}
f(x_{i-\frac{1}{2}})-\mathcal{I}^{2r-1}\left(x_{i-\frac{1}{2}};f\right)=O(h^{2r}).
\end{equation*}
In \cite{JiangShu,Liu},  \eqref{omega} is achieved via the  following expressions:
\begin{equation}\label{pesos}
\omega_{k}^r=\frac{\alpha_{k}^r}{\sum_{j=0}^{r-1}\alpha_{j}^r},\,\,  \textrm{ where } \alpha_{k}^r=\frac{{C}_{k}^r}{(\epsilon+I_k^r)^t}, \quad k=0,\cdots, r-1.
\end{equation}
In the previous expression, the parameter $t$ is an integer that assures maximum order of accuracy close to the discontinuities. The parameter $\epsilon>0$ is introduced to avoid divisions by zero, and is usually forced to take the size of the smoothness indicators at smooth zones. In our numerical tests, we will set it to $\epsilon=h^2$.  The values $I_k^r$ are called {\it smoothness indicators} for $f(x)$ on each sub-stencil of $r$ points.

\dioni{}{Since it was first introduced in \cite{Liu}, several successive improvements have been proposed for WENO algorithms. They have been focused on enhancing their accuracy, efficiency, and robustness in approximating solutions to hyperbolic conservation laws, but also on extending their applicability to other contexts such as approximation and interpolation of data. Initially introduced by Liu, Osher, and Chan in 1994 \cite{Liu}, WENO schemes have gone through several advancements. The WENO schemes proposed by Jiang and Shu in 1996 \cite{JiangShu}, improved the original idea in \cite{Liu}, by proposing new smoothness indicators inspired by the measure of the total variation. These smoothness indicators were more capable of detecting discontinuities and allowed to extended the idea of WENO to higher orders of accuracy. However, the classical WENO scheme might experience a loss of accuracy when encountering critical points in the solution: for instance, a fifth-order WENO scheme may only attain third-order accuracy near smooth extrema \cite{HAP05}. In fact, in \cite{HAP05}, the authors proposed the WENO-M scheme, that not only addressed the accuracy issue but also marked the first significant improvement in the solution quality near shocks and high gradients. However, the introduced mapping method proved to be computationally expensive. This study led to the publication of the article \cite{BCCSD08} where the authors proposed the WENO-Z scheme, that uses a new set of weights, derived from previously unused information within the classical WENO scheme: a higher-order global smoothness indicator constructed through a linear combination of the original smoothness indicators. This scheme achieved superior results with almost the same computational effort as the classical WENO method. After that, many variants of WENO schemes were proposed, starting from these two methodologies (see, e.g. \cite{FHW14, KLY13,  HYY20}). In the context of data approximation, we proposed a first improvement of WENO algorithm in 2018 \cite{ARS19}, where the algorithm attained a progressive order of accuracy close to the discontinuities using a recursive formulation of the WENO weights.
Using WENO interpolation, maximum order of accuracy is obtained in smooth parts of the data but the accuracy is reduced to order $r+1$ when, at least, a discontinuity crosses a stencil. In \cite{ARSY20}, Amat, Ruiz, Shu and Y\'a\~nez present a new WENO method using the Aitken-Neville algorithm to obtain progressive order of approximation. This method is introduced for the point-value discretization in a uniform grid, and to approximate at the mid-points of the intervals. In \cite{ARSY22}, the algorithm is extended to calculate the derivative value of a function knowing its evaluation in a non-uniform grid and for any point of the considered interval.}

In this paper, we extend this method to $n$ dimensions and prove its
theoretical properties. We start with dimension 1 in Section
\ref{wenoprogresivo1d}, following the ideas presented in
\cite{ARSY20,ARSY22}. \dioni{}{We provide a new recursive algorithm to obtain a non-linear
scalar approximation for any point in the central interval in Section \ref{wenoprogresivo1d}.
Then, in Section \ref{sec:multivariatelinearandwenoclassic}, we review the multidimensional Lagrange interpolation
using tensor product, and the WENO version designed by Aràndiga et al. in \cite{arandigamuletrenau}. Subsequently, we introduce, in Section \ref{weno2d}},
our general method for any dimension $n$ starting with $n=2$,
and give an explicit expression for the optimal weights. Afterwards, we
generalize the method to any number of dimension $n$, and prove that the
approximation attains the maximum order of accuracy when the nodes are in a region
where $f$ is smooth, and has an increasing order of accuracy depending on where the
isolated discontinuity is. In
Section \ref{smoothaccur}, we present the smoothness indicators. The theoretical results are confirmed by
the numerical examples, that are presented in Section \ref{numericalexps}. Finally, some conclusions are drawn in the last section.

\section{A new univariate WENO-$2r$ algorithm with progressive order of accuracy close to discontinuities}\label{wenoprogresivo1d}

Recently, Amat, Ruiz, Shu and Y\'a\~nez in \cite{ARSY20} introduced a new centered WENO-$2r$ method, which consists in using all the points free of discontinuities to interpolate a value at the mid-point of an interval. In this section, we generalize this interpolation to non-uniform grids. Let $[a,b]\subset \mathbb{R}$ be an interval, we consider a non-uniform grid $a=x_0<x_1<\hdots<x_J=b,$ with
$h=\max_{l=1,\hdots,J}|x_l-x_{l-1}|$ and $i\in\mathbb{N}$ such that $0\leq i-r \leq i+r-1 \leq J$. We consider a point, $x^*\in (x_{i-1},x_i)$; the largest stencil
$\mathcal{S}^{2r}_0=\{x_{i-r},\cdots, x_{i+r-1}\},$
and construct the polynomial which interpolates at these nodes, that we
denote by $p^{2r-1}_0$. Using the Aitken-Neville formula, we express
this polynomial using the following expression
\begin{equation*}
p^{2r-1}_0(x^*)=\sum_{j_0=0}^1C^{2r-2}_{0,j_0}(x^*)p_{j_0}^{2r-2}(x^*),
\end{equation*}
which  involves the two interpolatory polynomials, $p_{j_0}^{2r-2}$,
$j_0=0,1$ with respective stencils:
 $$\mathcal{S}^{2r-1}_0=\{x_{i-r},\cdots, x_{i+r-2}\},\quad \mathcal{S}^{2r-1}_1=\{x_{i-r+1},\cdots, x_{i+r-1}\},$$
where $C^{2r-2}_{0,0}(x^*)$ and $C^{2r-2}_{0,1}(x^*)$ are the optimal weights. We repeat this process with each polynomial up to degree $r$. A representation of this process can be seen in Figure \ref{figuraesquema} (where we have removed the dependence on the value $x^*$), \cite{ARSY20}.
\begin{figure}[!hbtp]
\begin{center}
  \input{figure1}
\end{center}
\caption{Diagram showing the structure of the optimal weights needed to obtain optimal order of accuracy, \cite{ARSY20}.}
\label{figuraesquema}
\end{figure}
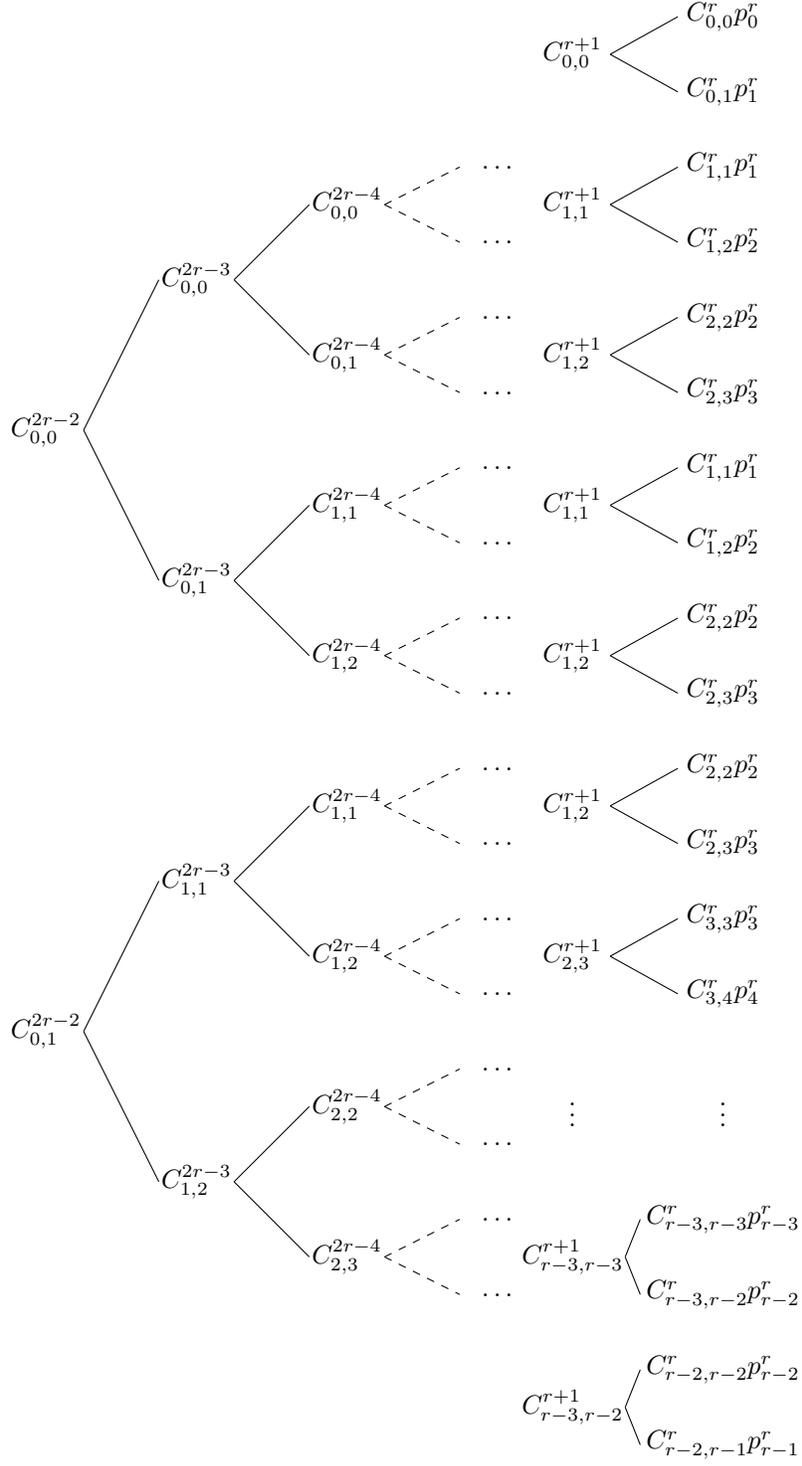
Thus, we obtain the general explicit expression:
 \begin{equation}\label{eq1}
\begin{split}
p^{2r-1}_0(x^*)&=\sum_{j_0=0}^1C^{2r-2}_{0,j_0}(x^*)p_{j_0}^{2r-2}(x^*) \\
&=\sum_{j_0=0}^1C^{2r-2}_{0,j_0}(x^*)\left(\sum_{j_1=j_0}^{j_0+1} C^{2r-3}_{j_0,j_1}(x^*) p_{j_1}^{2r-3}(x^*)\right)\\
&=\sum_{j_0=0}^1C^{2r-2}_{0,j_0}(x^*)\left(\sum_{j_1=j_0}^{j_0+1} C^{2r-3}_{j_0,j_1}(x^*) \left(\sum_{j_2=j_1}^{j_1+1} C^{2r-4}_{j_1,j_2}(x^*)p_{j_2}^{2r-4}(x^*)\right)\right)\\
&=\sum_{j_0=0}^1 C^{2r-2}_{0,j_0}(x^*)\left(\sum_{j_1=j_0}^{j_0+1}C_{j_0,j_1}^{2r-3}(x^*)
\left(\dots\left(\sum_{j_{r-3}=j_{r-4}}^{j_{r-4}+1} C^{r+1}_{j_{r-4},j_{r-3}}(x^*)\left(\sum_{j_{r-2}=j_{r-3}}^{j_{r-3}+1} C^{r}_{j_{r-3},j_{r-2}}(x^*)p^r_{j_{r-2}}(x^*)\right)\right)\dots\right)\right).
\end{split}
\end{equation}

We calculate these values through the following lemma (the version for uniform grids, and when the interpolation is at mid-points is proved in \cite{ARSY20}).

\begin{lemma}\label{lemaauxiliar1general1d}
Let $r\leq l \leq 2r-2$ and $0\leq j\leq (2r-2)-l$, $x^*\in(x_{i-1},x_i)$, then
\begin{equation}
p^{l+1}_j(x^*)=C^{l}_{j,j}(x^*)p_j^{l}(x^*)+C^{l}_{j,j+1}(x^*)p_{j+1}^{l}(x^*),
\end{equation}
where
\begin{equation}\label{pesostodosgeneral1d}
C^l_{j,j}(x^*)=\frac{x^*-x_{i-r+j+l+1}}{x_{i-r+j}-x_{i-r+j+l+1}}, \quad C^l_{j,j+1}(x^*)=1-C^l_{j,j}(x^*)=\frac{x_{i-r+j}-x^*}{x_{i-r+j}-x_{i-r+j+l+1}}.
\end{equation}
\end{lemma}

\begin{proof}
The proof is direct  {by} taking into account that the stencil for constructing $p^{l+1}_j$ is $\{x_{i-r+j},\hdots,x_{i-r+j+l+1}\}$ and for $p^{l}_j$ and $p^{l}_{j+1}$ are
$\{x_{i-r+j},\hdots,x_{i-r+j+l}\}$ and $\{x_{i-r+j+1},\hdots,x_{i-r+j+l+1}\}$ respectively.
\end{proof}

As Corollary, we obtain the values showed in \cite{ARSY20} for mid-point interpolation.

\begin{corollary}\label{lemaauxiliar1}
Let $r\leq l \leq 2r-2$ and $0\leq k\leq (2r-2)-l$, if we denote as $C^l_{k,k}$ and $C^{l}_{k,k+1}$ the values which satisfy:
\begin{equation}
p^{l+1}_k(x_{i-1/2})=C^{l}_{k,k}p_k^{l}(x_{i-1/2})+C^{l}_{k,k+1}p_{k+1}^{l}(x_{i-1/2}),
\end{equation}
then
\begin{equation}\label{pesostodos}
C^l_{k,k}=\frac{2(l-r+k+1)+1}{2(l+1)}, \quad C^l_{k,k+1}=1-C^l_{k,k}=\frac{2(r-k)-1}{2(l+1)}.
\end{equation}
\end{corollary}

From Eq. \eqref{eq1}, we can deduce the following recursive method:
\begin{equation}\label{formulapep}
\begin{cases}
  \tilde{p}^{r}_j(x^*)=p^{r}_{j}(x^*), & j=0,\hdots,r-1,\\[0.3cm]
  \tilde{p}^{l+1}_j(x^*)=\tilde{\omega}^{l}_{j,j}(x^*)\tilde{p}_j^{l}(x^*)+\tilde{\omega}^{l}_{j,j+1}(x^*)\tilde{p}_{j+1}^{l}(x^*),&
    l=r,\dots,2r-2, \, \, j=0,\dots,2r-2-l,\\
\end{cases}
\end{equation}
with the non-linear weights defined as:
\begin{equation}\label{pesosr}
\tilde{\omega}^l_{j,j_1}(x^*)=\frac{\tilde{\alpha}_{j,j_1}^l(x^*)}{\tilde{\alpha}_{j,j}^l(x^*)+\tilde{\alpha}_{j,j+1}^l(x^*)},\quad \tilde{\alpha}_{j,j_1}^l(x^*)=\frac{C_{j,j_1}^l(x^*)}{(\epsilon+ I^l_{j,j_1})^t}, \quad j_1=j,\,\, j+1,
\end{equation}
where $C_{j,j_1}^r(x^*)$, $j_1=j,j+1$ are determined in Eq. \eqref{pesostodosgeneral1d}.
The values $I^l_{j,j_1}$, explained in subsection \ref{smoothnessnuevo}, are smoothness indicators.

With this formulation, the approximation is defined by
\begin{equation}\label{eqpepfinal}
\tilde{\mathcal{I}}^{2r-1}(x^*;f)= \tilde{p}_{0}^{2r-1}(x^*),
\end{equation}
being $\tilde{p}_{0}^{2r-1}(x^*)$ the result of the recursive process, Eq. \eqref{formulapep}.

\subsection{On smoothness indicators and the accuracy of the new univariate progressive WENO method}\label{smoothnessnuevo}
The choice of the smoothness indicators is crucial. To design the new non-linear algorithm, we only calculate the smoothness indicators at level $l=r$, and
we use them in all the levels, as we can see in this subsection. In \cite{ABM} it is proved that if the smoothness indicators satisfy the following conditions:
\begin{enumerate}[label={\bfseries P\arabic*}]
\item\label{P1sm1d} The order of a smoothness indicator that is free of discontinuities is $h^2$, i.e.
$$I^r_{k}=O(h^2) \,\, \text{if}\,\, f \,\, \text{is smooth in } \,\, \mathcal{S}^r_{k}.$$
\item\label{P2sm1d} The distance between two smoothness indicators
  free of discontinuities is $h^{r+1}$, i.e. let
  $k,k'\in\{0,1,\hdots,r-1\}$ be such that
there does not exist any discontinuity in $\mathcal{S}^r_{k}$ and $\mathcal{S}^r_{k'}$, then
$$I^r_{k}-I^r_{k'}=O(h^{r+1}).$$
\item\label{P3sm1d} When a discontinuity crosses the stencil $\mathcal{S}^r_{k}$ then
$$I^r_{k} \nrightarrow 0 \,\, \text{as}\,\, h\to 0.$$
\end{enumerate}
Then, the optimal weights satisfy Eq. \eqref{omega} for determined parameters $t$ and $\epsilon$ (\dioni{}{Proposition} 2 in \cite{ABM}). We exploit this result to construct the smoothness indicators in each step of the Aitken-Neville's algorithm. We will use the following definition (\dioni{}{Definition} 4.1 in \cite{ARSY20}).
\begin{definition}\label{def1}
Let $l=r,\dots, 2r-2$, and $I^r_k$, with $k=0,\hdots,r-1$, be the smoothness indicators with properties \ref{P1sm1d},\ref{P2sm1d} and \ref{P3sm1d}. Then, we define the smoothness indicators at level $l$ as:
\begin{equation}\label{equationindicadores}
\begin{split}
& I_{k,k}^{l}= I_k^r, \quad k=0,\dots,(2r-2)-l,\\
& I_{k,k+1}^{l}= I_{l-(r-1)+k}^r, \quad k=0,\dots,(2r-2)-l.
\end{split}
\end{equation}
\end{definition}
Therefore, in each step, we will remove the part which is ``contaminated'' by a discontinuity. In particular, we can prove the following proposition adapting the proof of \dioni{}{Proposition} 4.4 in \cite{ARSY20} to any point $x^*\in(x_{i-1},x_i)$.
\begin{proposition}\label{prop1}
Let $r\leq l\leq 2r-2$, $0\leq k\leq (2r-2)-l$,
$x^*\in(x_{i-1},x_i)$, and let $I^l_{k,k}$, and $I^l_{k,k+1}$ be
smoothness indicators defined in \dioni{}{Definition} \ref{def1}. The following weights:
\begin{equation}\label{w_teo1}
\begin{aligned}
\tilde{\omega}^l_{k,k}(x^*)=\frac{ \tilde{\alpha}_{k,k}^l(x^*)}{ \tilde{\alpha}_{k,k}^l(x^*)+ \tilde{\alpha}_{k,k+1}^l(x^*)},\quad
\tilde{\omega}^l_{k,k+1}(x^*)=\frac{ \tilde{\alpha}_{k,k+1}^l(x^*)}{ \tilde{\alpha}_{k,k}^l(x^*)+ \tilde{\alpha}_{k,k+1}^l(x^*)},
\end{aligned}
\end{equation}
with,
\begin{equation}\label{alpha_teo1}
\begin{aligned}
\tilde{\alpha}_{k,k}^l(x^*)=\frac{C_{k,k}^l(x^*)}{(\epsilon+ {I}_{k,k}^l)^t},\quad
\tilde{\alpha}_{k,k+1}^l(x^*)=\frac{C_{k,k+1}^l(x^*)}{(\epsilon+ {I}_{k,k+1}^l)^t},
\end{aligned}
\end{equation}
and with $C^l_{k,k}(x^*)+C^l_{k,k+1}(x^*)=1$, will fall under one of the following cases:
\begin{enumerate}
\item If neither $ {I}^l_{k,k}$ nor $  I^l_{k,k+1}$ are affected by a discontinuity, then $ \tilde{\omega}_{k,k}^l(x^*)=C^l_{k,k}(x^*)+O(h^{r-1})$ and $ {\omega}_{k,k+1}^l(x^*)=C^l_{k,k+1}(x^*)+O(h^{r-1})$.
\item If ${I}^l_{k,k+1}$ is affected by a singularity, then $ \tilde{\omega}_{k,k}^l(x^*)=1+O(h^{2t})$ and $ \tilde{\omega}_{k,k+1}^l(x^*)=O(h^{2t})$.
\item If ${I}^l_{k,k}$ is affected by a singularity then $ \tilde{\omega}_{k,k+1}^l(x^*)=1+O(h^{2t})$ and $ \tilde{\omega}_{k,k}^l(x^*)=O(h^{2t})$.
\item If ${I}^l_{k,k}$ and ${I}^l_{k,k+1}$ are affected by a
  singularity then $ \tilde{\omega}_{k,k}^l(x^*)\nrightarrow 0$ and $
  \tilde{\omega}_{k,k+1}^l(x^*)\nrightarrow 0$.
\end{enumerate}
\end{proposition}

%Also, the next lemma is crucial to prove the order or accuracy.
%\begin{lemma}\label{lema_t}
%Let  $1<l_0\leq r-1$, $x^*\in(x_{i-1},x_i)$. If there exists a discontinuity at $[x_{i+l_0-1},x_{i+l_0}]$, then
%\begin{equation*}
%\begin{split}
%&(\tilde{C}_0^r(x^*),\tilde{C}_1^r(x^*),\hdots,\tilde{C}_{r-1}^r(x^*))=\\
%&\sum_{j_0=0}^1 \tilde{\omega}^{r-2+l_0}_{0,j_{0}}(x^*)\left(\sum_{j_{1}=j_{0}}^{j_0+1}\tilde{\omega}_{j_0,j_1}^{r-3+l_0}(x^*)\left(\dots\left(\sum_{j_{l_0-2}=j_{l_0-3}}^{j_{l_0-3}+1} \tilde{\omega}^{r+1}_{j_{l_0-3},j_{l_0-2}}(x^*){\bf  C_{j_{l_0-2}}^{r+1}(x^*)} \right)\dots\right)\right)+O(h^{2t})
%\end{split}
%\end{equation*}
%with ${\bf  C_{k}^{r+1}}(x^*)$, $k=0,\dots,r-1$ defined in Eq. \eqref{nl_op_w_r} and
%\begin{equation}
%\begin{split}
%&\tilde{\omega}^l_{k,k_1}(x^*)=\frac{\tilde{\alpha}_{k,k_1}^l(x^*)}{\tilde{\alpha}_{k,k}^l(x^*)+\tilde{\alpha}_{k,k+1}^l(x^*)},\quad \quad\tilde{\alpha}_{k,k_1}^l=\frac{C_{k,k_1}^l(x^*)}{(\epsilon+I^l_{k,k_1})^t}, \quad k_1=k,\,\,k+1.\\
%\end{split}
%\end{equation}
%where $I^l_{k,k_1}$ are the smoothness indicators satisfying \ref{P1sm1d}, \ref{P2sm1d} and \ref{P3sm1d}, $C_{k,k_1}^l$ defined in Eq. \eqref{pesostodos}.
%
%Also,
%\begin{equation}\label{final}
%(\tilde\omega_0^r(x^*),\hdots,\tilde\omega_{r-1}^r(x^*))=(\hat{C}_0^r(x^*)+O(h^{r-1}),\hdots,\hat{C}_{l_0-1}^r(x^*)+O(h^{r-1}),O(h^{2t}),\hdots,O(h^{2t}))
%\end{equation}
%being
%$$p_0^{l_0+r-1}(x^*)=\sum_{k=0}^{l_0-1}\hat{C}_k^r(x^*) p^r_k(x^*).$$
%\end{lemma}

We have the following main result:
\begin{theorem}\label{teo1}
Let $1< l_0 \leq r-1$, $x^*\in(x_{i-1},x_i)$, and $\tilde{\mathcal{I}}^{2r-1}(x^*;f)$ the result of the recursive process, Eq. \eqref{formulapep},  if $f$ is smooth in $[x_{i-r},x_{i+r-1}]\setminus \Omega$ and $f$ has a discontinuity at $\Omega$ then
\begin{equation}
\tilde{\mathcal{I}}^{2r-1}(x^*;f)-{f}(x^*)=\left\{
                                                  \begin{array}{ll}
                                                    O(h^{2r}), & \hbox{if $\,\,\Omega=\emptyset$;} \\
O(h^{r+l_0}), & \hbox{if   $\,\,\Omega=[x_{i+l_0-1},x_{i+l_0}]$. }\\
                                                  \end{array}
                                                \right.
\end{equation}
\end{theorem}
By symmetry, we only analyze when there exists an isolated discontinuity at an interval $[x_{i-1+l_0},x_{i+l_0}]$, $l_0=1,\hdots, r-1$ (analogously, we obtain the equivalent symmetric results for $[x_{i-r+l_0},x_{i-r+l_0+1}]$, $l_0=0,\hdots,r-2$).

In the next sections, we generalize this method to \dioni{}{multi-dimensions}. We also introduce some possible smoothness indicators which satisfy \ref{P1sm1d},\ref{P2sm1d}, and \ref{P3sm1d} in Section \ref{smoothaccur}.

\section{\dioni{}{Comparison of multivariate linear Lagrange interpolation with the non-linear WENO method in Cartesian grids}}\label{sec:multivariatelinearandwenoclassic}
In this section, we briefly review the Lagrange interpolation problem
in several variables when the data are located in Cartesian grids
(non-necessarily equally-spaced) and we construct a multivariate WENO
method following the ideas presented in \cite{arandigamuletrenau} for
two variables. We present the necessary ingredients to extend the
progressive WENO-$2r$ algorithm introduced in Section
\ref{wenoprogresivo1d} to several variables.

\subsection{Multivariate linear interpolation}\label{linealnd}
Let us start supposing $a_j,b_j\in\mathbb{R}$ with $a_j<b_j, \,j=1,\hdots,n$, and an hypercube denoted as:
$$\prod_{j=1}^n [a_j,b_j]=[a_1,b_1]\times \hdots \times [a_n,b_n].$$
We call the points of a grid for each interval as:
$$a_j=\vv{x}{j}{0}<\vv{x}{j}{1}<\vv{x}{j}{2}<\hdots< \vv{x}{j}{J_j}=b_j, \quad j=1,\hdots,n,$$
and define
$$h_j=\max_{i=1,\hdots,J_j}{|\vv{x}{j}{i}-\vv{x}{j}{i-1}|}, \quad h:=\max_{j=1,\hdots,n} h_j.$$
We suppose an unknown function $f:\prod_{j=1}^n [a_j,b_j] \to \mathbb{R}$, and consider our data as the evaluation of this function in the points of the Cartesian grid, i.e.
$$f_{(l_1,\hdots,l_n)}=f(\vv{x}{1}{l_1},\vv{x}{2}{l_2},\hdots,\vv{x}{n}{l_n}), \quad 0\leq l_j\leq J_j,\quad 1\leq j\leq n.$$
Let the polynomials of $n$ variables be denoted by:
$$\Pi^{\tau_1,\hdots,\tau_n}_n=\{p(x_1,\hdots,x_n)=\sum_{l_1=0}^{\tau_1}
\hdots \sum_{l_n=0}^{\tau_n}
a_{(l_1,\hdots,l_n)}x_1^{l_1}x_2^{l_2}\hdots x_n^{l_n}|\,\,a_{(l_1,\hdots,l_n)}\in\mathbb{R},\,\, 0\leq l_j\leq \tau_j,\,\,j=1,\hdots,n\},$$
when $\tau_1=\hdots=\tau_n=\tau$ we call it $\Pi^{\tau}_n:=\Pi^{\tau_1,\hdots,\tau_n}_n$.
We consider $(i_1,\hdots,i_n)\in\mathbb{N}^n$ to be the reference
index where the approximation is centered at, and $r\in\mathbb{N}$ such that $0\leq i_j-r \leq i_j+r-1\leq J_j$, $j=1,\hdots,n$, and the centered stencil:
\begin{equation}\label{eq1stencil}
\begin{split}
\mathcal{S}_{\mathbf{0}}^{2r}&=\{\vv{x}{1}{i_1-r},\hdots,\vv{x}{1}{i_1+r-1}\}\times\{\vv{x}{2}{i_2-r},\hdots,\vv{x}{2}{i_2+r-1}\}\times \hdots \times \{\vv{x}{n}{i_n-r},\hdots,\vv{x}{n}{i_n+r-1}\}\\
&=:(\mathcal{S}_{0}^{2r})_1\times\hdots\times (\mathcal{S}_{0}^{2r})_n=\prod_{j=1}^n (\mathcal{S}_{0}^{2r})_j,
\end{split}
\end{equation}
then, the problem consists in calculating a polynomial of degree $2r-1$ such that:
$$p(\xx)=f(\xx), \quad \, \forall\, \xx \in \mathcal{S}_{(0,\hdots,0)}^{2r}.$$
To do so, we will use the Lagrange base of polynomials:
$$L^{j}_{p}(x_j)=\prod_{k=i_j-r,k\neq p}^{i_j+r-1}\left(\frac{x_j-\vv{x}{j}{k}}{\vv{x}{j}{p}-\vv{x}{j}{k}}\right), \quad p=i_j-r,\hdots,i_j+r-1, \quad j=1,\hdots,n,$$
then, the unique polynomial is
$$p_{\mathbf{0}}^{2r-1}(x_1,\hdots,x_n)=\sum_{l_1=i_1-r}^{i_1+r-1}\sum_{l_2=i_2-r}^{i_2+r-1}\hdots\sum_{l_n=i_n-r}^{i_n+r-1}f(\vv{x}{1}{l_1},\vv{x}{2}{l_2},\hdots,\vv{x}{n}{l_n})L^{1}_{l_1}(x_1)\hdots L^{n}_{l_n}(x_n).$$
It can be checked (see \cite{arandigamuletrenau}) that if $\xe=(x_1^*,\hdots,x_n^*)\in \prod_{j=1}^n [\vv{x}{j}{i_j-1},\vv{x}{j}{i_j}]$, and $f\in \mathcal{C}^{2nr}(\mathbb{R})$ then the error satisfies:
\begin{equation}\label{errorfuncion}
E(\xe)=f(\xe)-p_{\mathbf{0}}^{2r-1}(\xe)= O(h^{2r}),
\end{equation}
and if $\mathbf{m}=(m_1,\hdots,m_n)\in \mathbb{N}^n$ with $m_j \leq 2r-1$, $j=1,\hdots,n$, we get
\begin{equation}\label{errorderivadas}
E^{\dioni{}{(\mathbf{m})}}(\xe)=E^{(m_1,\hdots,m_n)}(\xe)= O(h^{2r-\max\{m_1,\hdots,m_n\}}),
\end{equation}
where
$$E^{(m_1,\hdots,m_n)}(\xx)=\frac{\partial^{m_1+\hdots+m_n} E}{\partial^{m_1} x_1\hdots\partial^{m_n} x_n}(\xx).$$

\subsection{Multivariate WENO interpolation in Cartesian grids}\label{wenoclasicond}

Using the same notation as in the previous section, the goal is to construct a non-linear interpolant in \dioni{}{multi-dimensions} in the same way as we reviewed in $1d$, i.e., an interpolant with maximum order $2r$ in the smooth parts and with order $r+1$ when there exists a discontinuity which crosses, at least, one small stencil.

In this case, we introduce the method supposing that we want to obtain
an approximation of the  function $f$ evaluated at any point of the
hypercube $\xe \in \prod_{j=1}^n
[\vv{x}{j}{i_j-1},\vv{x}{j}{i_j}]$. The construction is similar to the
$1d$ case:
Firstly, we make the linear combination of interpolants of lower degree, $r$, if $\mathbf{k}=(k_1,\hdots,k_n)$ then
$$p_0^{2r-1}(\xe)=\sum_{\mathbf{k}\in \{0,1,\hdots,r-1\}^n}C^r_{\mathbf{k}}(\xe)p_{\mathbf{k}}^r(\xe),$$
where $C^r_{\mathbf{k}}(\xe)$ are the optimal weights, and it is not difficult to prove that they have the explicit form:
$$C^r_{\mathbf{k}}(\xe)=\prod_{j=1}^n C^r_{k_j}(x^*_j), \quad \mathbf{k}\in \{0,\hdots,r-1\}^n,$$
being $C^r_{k_j}(x^*_j)$, $k_j=0,\hdots,r-1$, $j=1,\hdots,n$ the
optimal weights in $1d$. In the case that we want to approximate at
the mid-point of the hypercube, i.e.  $\xe=\x12e$, then they are defined in Eq. \eqref{opt_w}. The polynomials $p_{\mathbf{k}}^r$ interpolate at the nodes
\begin{equation*}
\begin{split}
\mathcal{S}_{\mathbf{k}}^{r}&=\{\vv{x}{1}{i_1+k_1-r},\hdots,\vv{x}{1}{i_1+k_1}\}\times\{\vv{x}{2}{i_2+k_2-r},\hdots,\vv{x}{2}{i_2+k_2}\}\times \hdots \times \{\vv{x}{n}{i_n+k_n-r},\hdots,\vv{x}{n}{i_n+k_n}\}=\prod_{j=1}^n \mathcal{S}_{k_j}^r.
\end{split}
\end{equation*}
Then, we replace the optimal weights by non-linear ones using the following formula
\begin{equation}\label{omegand}
\omega^r_{\mathbf{k}}(\xe)=\frac{\alpha^r_{\mathbf{k}}(\xe)}{\sum_{\mathbf{l}\in\{0,\hdots,r-1\}^n}\alpha^r_{\mathbf{l}}(\xe)},\,\,\text{with}\,\,
\alpha^r_{\mathbf{k}}(\xe)=\frac{C^r_{\mathbf{k}}(\xe)}{(\epsilon+I^r_{\mathbf{k}})^t},
\end{equation}
being $I^r_{\mathbf{k}}$ the smoothness indicators with the same requirements pointed out in Section \ref{wenoprogresivo1d}, i.e.:
\begin{enumerate}[label={\bfseries P\arabic*}]
\item\label{P1sm} The order of the smoothness indicator free of discontinuities is $h^2$, i.e.
$$I^r_{\mathbf{k}}=O(h^2) \,\, \text{if}\,\, f \,\, \text{is smooth in } \,\, \mathcal{S}^r_{\mathbf{k}}.$$
\item\label{P2sm} The distance between two smoothness indicators free
  of discontinuities is $h^{r+1}$, i.e. let
  $\mathbf{k}=(k_1,\hdots,k_n)$, and $\mathbf{k}'=(k'_1,\hdots,k'_n)$
  be such that
there does not exist any discontinuity in $\mathcal{S}^r_{\mathbf{k}}$ and $\mathcal{S}^r_{\mathbf{k}'}$ then
$$I^r_{\mathbf{k}}-I^r_{\mathbf{k}'}=O(h^{r+1}).$$
\item\label{P3sm} When a discontinuity crosses the stencil $\mathcal{S}^r_{\mathbf{k}}$ then:
$$I^r_{\mathbf{k}} \nrightarrow 0 \,\, \text{as}\,\, h\to 0.$$
\end{enumerate}
An example of smoothness indicators will be introduced in Section
\ref{smoothaccur}. The parameters $\epsilon$ and $t$ are chosen to
obtain maximum order. In our case, following \cite{ABM} and
\cite{arandigamuletrenau} we will take $\epsilon=h^2$ and
$t=\frac{1}{2}(r+1)$. Therefore, we state the  next result, which is
similar to  \dioni{}{Proposition} 2 in \cite{ABM} and Theorem 1 in
\cite{arandigamuletrenau}.

\begin{theorem}
Let $\xe\in \prod_{j=1}^n [\vv{x}{j}{i_j-1},\vv{x}{j}{i_j}]$, the multivariate WENO interpolant
\begin{equation}
\mathcal{I}^{2r-1}(\xe;f)=\sum_{\mathbf{k}\in \{0,1,\hdots,r-1\}^n}\omega^r_{\mathbf{k}}(\xe)p_{\mathbf{k}}^r(\xe)
\end{equation}
with $\omega^r_{\mathbf{k}}$, $\mathbf{k}\in \{0,1,\hdots,r-1\}^n$
defined in Eq. \eqref{omegand}, with smoothness indicator
$I^r_{\mathbf{k}}$, $\mathbf{k}\in \{0,1,\hdots,r-1\}^n$ \dioni{}{fulfilling}
\ref{P1sm}, \ref{P2sm}, \ref{P3sm}; $\epsilon=h^2$ and
$t=\frac{1}{2}(r+1)$ satisfies:
$$f(\xe)-\mathcal{I}^{2r-1}(\xe;f)=
\begin{cases}
O(h^{2r}), & \text{at smooth regions,}\\
O(h^{r+1}), & \text{if, at least, one stencil lies in a smooth region.}
\end{cases}
$$
\end{theorem}

\section{A new progressive bivariate, \dioni{}{$n=2$}, WENO method}\label{weno2d}

The idea of this new method is to reach the maximum possible order of
accuracy when one discontinuity crosses the largest stencil. For this
purpose, we use the Aitken-Neville-based algorithm developed in
\dioni{}{multi-dimensions}, \cite{gascaaitkenclasica}. We start with a
polynomial of degree $2r-1$ and decompose it in \dioni{}{$2^2$} polynomials of
degree $2r-2$ obtaining as weights polynomials of degree 1 that will
be replaced by non-linear weights, depending on the location of the
discontinuity. Afterwards, we continue decomposing the \dioni{}{$2^2$}
polynomials of degree $2r-2$ in polynomials of degree $2r-3$, and so
on. In each step, the non-linear weights determine the stencils free
of discontinuity, which will be used to approximate the value. The
procedure is similar to the method expounded in Section
\ref{wenoprogresivo1d}. \dioni{}{For
ease of reading, we start with $n = 2$ in this section, and then we extend our
results to any dimension $n$ in Section \ref{wenond}.}

\subsection{A first example: A progressive bivariate WENO method with $r=2$}

Let us start by designing a bivariate WENO method with $r=2$. We suppose a non
necessarily uniform  grid in $[a_1,b_1]\times[a_2,b_2]$
defined by $\{(\vv{x}{1}{l_1}, \vv{x}{2}{l_2})\}_{(l_1,l_2)=(0,0)}^{(J_1,J_2)}$, with $a_1=\vv{x}{1}{0}<\vv{x}{1}{1}<\hdots<\vv{x}{1}{J_1}=b_1$, and $a_2=\vv{x}{2}{0}<\vv{x}{2}{1}<\hdots<\vv{x}{2}{J_2}=b_2$, and the data $f_{l_1,l_2}=f(\vv{x}{1}{l_1},\vv{x}{2}{l_2})$, $l_j=0,\hdots,J_j$, with $j=1,2$. Let $(i_1,i_2)\in \mathbb{N}^2$ such that $0\leq i_j-2$ and $i_j+1\leq J_j$, with $j=1,2$. We determine the largest stencil $$\mathcal{S}^{4}_{(0,0)}=\{\vv{x}{1}{i_1-2},\vv{x}{1}{i_1-1},\vv{x}{1}{i_1},\vv{x}{1}{i_1+1}\}\times\{\vv{x}{2}{i_2-2},\vv{x}{2}{i_2-1},\vv{x}{2}{i_2},\vv{x}{2}{i_2+1}\}=(\mathcal{S}_{0}^{4})_1\times(\mathcal{S}_{0}^{4})_2,$$
and we want to interpolate at any point
$\xe=(x^*_1,x^*_2)\in[\vv{x}{1}{i_1-1},\vv{x}{1}{i_1}]\times [\vv{x}{2}{i_2-1},\vv{x}{2}{i_2}]$. We compute the polynomial
$$p^{3}_{(0,0)}(x_1,x_2)=\sum_{l_1,l_2=0}^3 a_{l_1,l_2}x_1^{l_1}x_2^{l_2},$$
such that $p^{3}_{(0,0)}(\vv{x}{1}{l_1},\vv{x}{2}{l_2})=f_{l_1,l_2}$ if $(\vv{x}{1}{l_1},\vv{x}{2}{l_2})\in\mathcal{S}^{4}_{(0,0)}$. Again, we can express it as the sum of $2^2$ polynomials that interpolate in the stencils:
\begin{equation*}
\begin{split}
&\mathcal{S}^{3}_{(0,0)}=\{\vv{x}{1}{i_1-2},\vv{x}{1}{i_1-1},\vv{x}{1}{i_1}\}\times\{\vv{x}{2}{i_2-2},\vv{x}{2}{i_2-1},\vv{x}{2}{i_2}\}=(\mathcal{S}_{0}^{3})_1\times(\mathcal{S}_{0}^{3})_2,\\ &\mathcal{S}^{3}_{(1,0)}=\{\vv{x}{1}{i_1-1},\vv{x}{1}{i_1},\vv{x}{1}{i_1+1}\}\times\{\vv{x}{2}{i_2-2},\vv{x}{2}{i_2-1},\vv{x}{2}{i_2}\}=(\mathcal{S}_{1}^{3})_1\times(\mathcal{S}_{0}^{3})_2,\\
&\mathcal{S}^{3}_{(0,1)}=\{\vv{x}{1}{i_1-2},\vv{x}{1}{i_1-1},\vv{x}{1}{i_1}\}\times\{\vv{x}{2}{i_2-1},\vv{x}{2}{i_2},\vv{x}{2}{i_2+1}\}=(\mathcal{S}_{0}^{3})_1\times(\mathcal{S}_{1}^{3})_2,\\ &\mathcal{S}^{3}_{(1,1)}=\{\vv{x}{1}{i_1-1},\vv{x}{1}{i_1},\vv{x}{1}{i_1+1}\}\times\{\vv{x}{2}{i_2-1},\vv{x}{2}{i_2},\vv{x}{2}{i_2+1}\}=(\mathcal{S}_{1}^{3})_1\times(\mathcal{S}_{1}^{3})_2,
\end{split}
\end{equation*}
called $p^{2}_{(0,0)}, p^{2}_{(1,0)}, p^{2}_{(0,1)}, p^{2}_{(1,1)}$ using, as we mentioned before, the Aitken-Neville-type formula given in \cite{gascaaitkenclasica}:
\begin{equation}
p^{3}_{(0,0)}(\xe)=\sum_{\mathbf{j}_0\in\{0,1\}^2}C_{(0,0),\mathbf{j}_0}^2(\xe)p^{2}_{\mathbf{j}_0}(\xe),
\end{equation}
with
\begin{equation*}
\begin{split}
C^2_{(0,0),(0,0)}(\xe)&=\left(\frac{x^*_1-\vv{x}{1}{i_1+1}}{\vv{x}{1}{i_1-2}-\vv{x}{1}{i_1+1}}\right)\left(\frac{x^*_2-\vv{x}{2}{i_2+1}}{\vv{x}{2}{i_2-2}-\vv{x}{2}{i_2+1}}\right)= C^2_{0,0}(x_1^*)C^2_{0,0}(x_2^*),\\ C^2_{(0,0),(1,0)}(\xe)&=\left(\frac{x^*_1-\vv{x}{1}{i_1-2}}{\vv{x}{1}{i_1+1}-\vv{x}{1}{i_1-2}}\right)\left(\frac{x^*_2-\vv{x}{2}{i_2+1}}{\vv{x}{2}{i_2-2}-\vv{x}{2}{i_2+1}}\right)= C^2_{0,1}(x_1^*)C^2_{0,0}(x_2^*),\\
C^2_{(0,0),(0,1)}(\xe)&=\left(\frac{x^*_1-\vv{x}{1}{i_1+1}}{\vv{x}{1}{i_1-2}-\vv{x}{1}{i_1+1}}\right)\left(\frac{x^*_2-\vv{x}{2}{i_2-2}}{\vv{x}{2}{i_2+1}-\vv{x}{2}{i_2-2}}\right)= C^2_{0,0}(x_1^*)C^2_{0,1}(x_2^*),\\
C^2_{(0,0),(1,1)}(\xe)&=\left(\frac{x^*_1-\vv{x}{1}{i_1-2}}{\vv{x}{1}{i_1+1}-\vv{x}{1}{i_1-2}}\right)\left(\frac{x^*_2-\vv{x}{2}{i_2-2}}{\vv{x}{2}{i_2+1}-\vv{x}{2}{i_2-2}}\right)= C^2_{0,1}(x_1^*)C^2_{0,1}(x_2^*).
\end{split}
\end{equation*}
Finally, we define the non-linear weights as:
\begin{equation*}
\tilde{\omega}_{(0,0),\mathbf{k}}^2(\xe)=\frac{\tilde{\alpha}_{(0,0),\mathbf{k}}^2(\xe)}{\sum_{\mathbf{l}\in\{0,1\}^2}\tilde{\alpha}_{(0,0),\mathbf{l}}^2(\xe)},\,\,  \textrm{ where }\,\, \tilde{\alpha}_{(0,0),\mathbf{k}}^2(\xe)=\frac{C_{(0,0),\mathbf{k}}^2(\xe)}{(\epsilon+I_{(0,0),\mathbf{k}}^2)^t},\quad \mathbf{k}\in\{0,1\}^2,
\end{equation*}
with $I_{(0,0),\mathbf{k}}^2$ defined as:
 $$I_{(0,0),(0,0)}^2=I_{(0,0)}^2, \quad I_{(0,0),(0,1)}^2=I_{(0,1)}^2,\quad I_{(0,0),(1,0)}^2=I_{(1,0)}^2,\quad I_{(0,0),(1,1)}^2=I_{(1,1)}^2,$$
where $I_{\mathbf{k}}^2$, $\mathbf{k}\in\{0,1\}^2$ are smoothness indicators which satisfy the properties \ref{P1sm}, \ref{P2sm} and \ref{P3sm}. The parameters $t$ and $\epsilon$ are defined in Section \ref{wenoclasicond} as $h^2$ and $r+1=3$ respectively. Thus, we define the new interpolant as:
$$\mathcal{I}^3\left(\xe;f\right)=\sum_{\mathbf{j}_0\in\{0,1\}^2}\omega_{(0,0),\mathbf{j}_0}^2(\xe)p^{2}_{\mathbf{j}_0}(\xe).$$
If  we apply it at mid-points, the interpolant presented in \cite{arandigamuletrenau} and the method showed in Section \ref{wenoclasicond} for $n=2$ are similar
if the same smoothness indicators are used. Therefore, when $r=2$ there are not differences between the progressive WENO and the classical one. We show the construction of the new method for $r=3$.

\subsection{The case $r=3$}
In this case, we use all the points of the following largest stencil:
$$\mathcal{S}^{6}_{(0,0)}=\{\vv{x}{1}{i_1-3},\vv{x}{1}{i_1-2},\vv{x}{1}{i_1-1},\vv{x}{1}{i_1},\vv{x}{1}{i_1+1},\vv{x}{1}{i_1+2}\}\times\{\vv{x}{2}{i_2-3},\vv{x}{2}{i_2-2},\vv{x}{2}{i_2-1},\vv{x}{2}{i_2},\vv{x}{2}{i_2+1},\vv{x}{2}{i_2+2}\}.$$
We follow the recursive formula described in \eqref{formulapep}. Thus, we start calculating the evaluation at the point $\xe$ of the polynomials:
$$p^{3}_{\mathbf{j}_1}(\xe), \quad  \mathbf{j}_1=(j^{(1)}_1,j^{(1)}_2) \in \{0,1,2\}^2,$$
being the stencils used for constructing each polynomial $\mathcal{S}^{3}_{(l_1,l_2)}$, with $l_j=0,1,2$, and $j=1,2$, see Figure \ref{wenonormal}.
%Figura weno
\begin{figure}[!hbtp]
\usetikzlibrary{arrows}
\begin{center}
  \input{figure3}
\end{center}
\caption{Stencils used to get $p^{3}_{\mathbf{j}_1}(\xe), \,  \mathbf{j}_1\in \{0,1,2\}^2$. They are used in classical bivariate WENO \cite{arandigamuletrenau}}\label{wenonormal}
\end{figure}
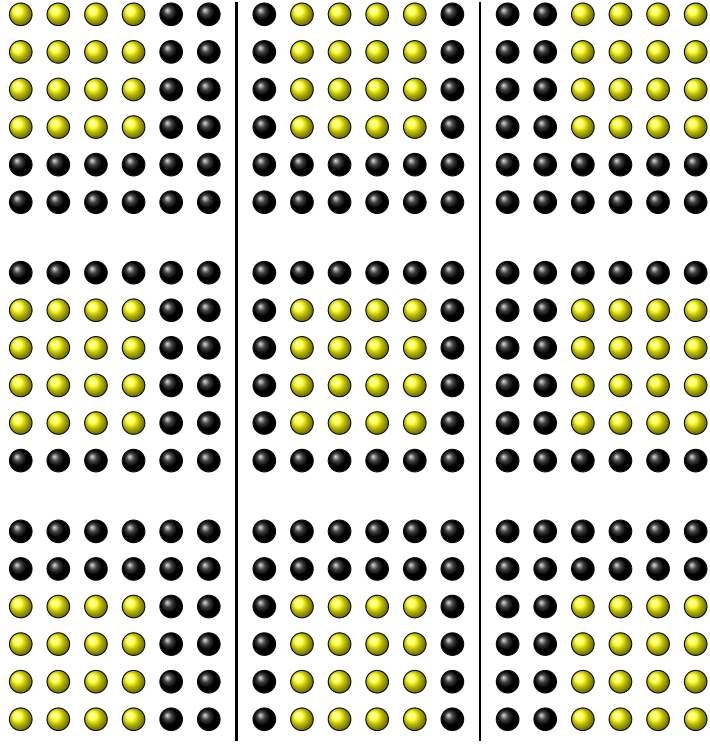
And, by means of the Aitken-Neville formula, we represent $p^4_{\mathbf{j}_0}$, $\mathbf{j}_0=(j^{(0)}_1,j^{(0)}_2) \in\{0,1\}^2$, whose stencils are displayed in Figure \ref{progresivo1}, as the sum of the evaluations of polynomials of degree 3 (Figure \ref{progresive2})
\begin{equation}\label{eqp4primerejemplo}
\begin{split}
&p^{4}_{(0,0)}(\xe)=\sum_{\mathbf{j}_1\in\{0,1\}\times\{0,1\}}C_{(0,0),\mathbf{j}_1}^3(\xe)p^{3}_{\mathbf{j}_1}(\xe),\quad
p^{4}_{(1,0)}(\xe)=\sum_{\mathbf{j}_1\in\{1,2\}\times\{0,1\}}C_{(1,0),\mathbf{j}_1}^3(\xe)p^{3}_{\mathbf{j}_1}(\xe),\\
&p^{4}_{(0,1)}(\xe)=\sum_{\mathbf{j}_1\in\{0,1\}\times\{1,2\}}C_{(0,1),\mathbf{j}_1}^3(\xe)p^{3}_{\mathbf{j}_1}(\xe),\quad p^{4}_{(1,1)}(\xe)=\sum_{\mathbf{j}_1\in\{1,2\}\times\{1,2\}}C_{(1,1),\mathbf{j}_1}^3(\xe)p^{3}_{\mathbf{j}_1}(\xe),\\
\end{split}
\end{equation}
with
$$C^3_{\mathbf{j}_0,\mathbf{j}_1}=C^3_{(j^{(0)}_1,j^{(0)}_2),(j^{(1)}_1,j^{(1)}_2)}=C^3_{j^{(0)}_1,j^{(1)}_1}C^3_{j^{(0)}_2,j^{(1)}_2}, \,\, \mathbf{j}_0\in\{0,1\}^2, \,\,
\mathbf{j}_1\in \mathbf{j}_0+\{0,1\}^2,$$
where $C^3_{k,k_1}$, $k=0,1$, $k_1=k+\{0,1\}$ are defined in Eq. \eqref{pesostodosgeneral1d}. Now, we replace in Eq. \eqref{eqp4primerejemplo} the linear-weights for non-linear ones:
\begin{equation*}
\tilde{\omega}_{\mathbf{j}_0,\mathbf{j}_1}^3(\xe)=\frac{\tilde{\alpha}_{\mathbf{j}_0,\mathbf{j}_1}^3(\xe)}{\sum_{\mathbf{l}\in\mathbf{j}_0+\{0,1\}^2}\tilde{\alpha}_{\mathbf{j}_0,\mathbf{l}}^3(\xe)},\,\,  \textrm{ where }\,\, \tilde{\alpha}_{\mathbf{j}_0,\mathbf{j}_1}^3(\xe)=\frac{C_{\mathbf{j}_0,\mathbf{j}_1}^3(\xe)}{(\epsilon+I_{\mathbf{j}_0,\mathbf{j}_1}^3)^t},\quad \mathbf{j}_0\in\{0,1\}^2, \, \mathbf{j}_1\in\mathbf{j}_0+\{0,1\}^2,
\end{equation*}
and
$$I_{\mathbf{j}_0,\mathbf{j}_1}^3=I_{\mathbf{j}_1}^3,\,\, \mathbf{j}_0\in\{0,1\}^2, \,\,
\mathbf{j}_1\in \mathbf{j}_0+\{0,1\}^2,$$
being $I^3_{\mathbf{j}_1}$ smoothness indicators which satisfy the properties \ref{P1sm}, \ref{P2sm} and \ref{P3sm}.
We obtain the approximation:
\begin{equation}\label{eqp4primerejemplonolineal}
\begin{split}
&\tilde{p}^{4}_{(0,0)}(\xe)=\sum_{\mathbf{j}_1\in\{0,1\}\times\{0,1\}}\tilde{\omega}_{(0,0),\mathbf{j}_1}^3(\xe)p^{3}_{\mathbf{j}_1}(\xe),\quad
\tilde{p}^{4}_{(1,0)}(\xe)=\sum_{\mathbf{j}_1\in\{1,2\}\times\{0,1\}}\tilde{\omega}_{(1,0),\mathbf{j}_1}^3(\xe)p^{3}_{\mathbf{j}_1}(\xe),\\
&\tilde{p}^{4}_{(0,1)}(\xe)=\sum_{\mathbf{j}_1\in\{0,1\}\times\{1,2\}}\tilde{\omega}_{(0,1),\mathbf{j}_1}^3(\xe)p^{3}_{\mathbf{j}_1}(\xe),\quad \tilde{p}^{4}_{(1,1)}(\xe)=\sum_{\mathbf{j}_1\in\{1,2\}\times\{1,2\}}\tilde{\omega}_{(1,1),\mathbf{j}_1}^3(\xe)p^{3}_{\mathbf{j}_1}(\xe).\\
\end{split}
\end{equation}

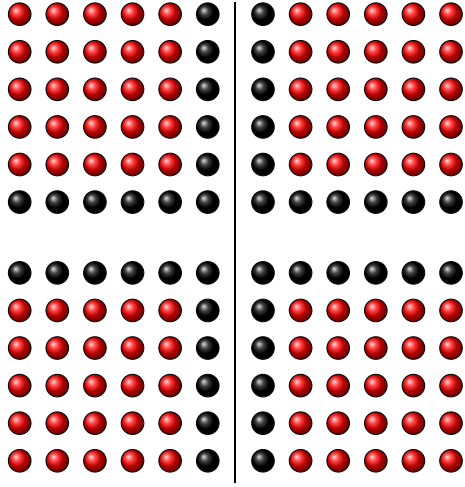
\begin{figure}[H]
\usetikzlibrary{arrows}
\begin{center}
  \input{figure2}
\end{center}
\caption{Stencils used to get $p^4_{\mathbf{j}_0}(\xe)$, $\mathbf{j}_0\in\{0,1\}^2$.}\label{progresivo1}
\end{figure}

The last step in our new algorithm for $r=3$ is to develop the polynomial of degree 5 as follows:
\begin{equation}\label{p5lineal}
p^{5}_{(0,0)}(\xe)=\sum_{\mathbf{j}_0\in\{0,1\}^2}C_{(0,0),\mathbf{j}_0}^4({\xe})p^{4}_{\mathbf{j}_0}(\xe),
\end{equation}
with
\begin{equation*}
\begin{split}
C^4_{(0,0),(0,0)}(\xe)&=\left(\frac{x^*_1-\vv{x}{1}{i_1+2}}{\vv{x}{1}{i_1-3}-\vv{x}{1}{i_1+2}}\right)\left(\frac{x^*_2-\vv{x}{2}{i_2+2}}{\vv{x}{2}{i_2-3}-\vv{x}{2}{i_2+2}}\right)=C^4_{0,0}(x_1^*)C^4_{0,0}(x_2^*),\\ C^4_{(0,0),(1,0)}(\xe)&=\left(\frac{x^*_1-\vv{x}{1}{i_1-3}}{\vv{x}{1}{i_1+2}-\vv{x}{1}{i_1-3}}\right)\left(\frac{x^*_2-\vv{x}{2}{i_2+2}}{\vv{x}{2}{i_2-3}-\vv{x}{2}{i_2+2}}\right)=C^4_{0,1}(x_1^*)C^4_{0,0}(x_2^*),\\
C^4_{(0,0),(0,1)}(\xe)&=\left(\frac{x^*_1-\vv{x}{1}{i_1+2}}{\vv{x}{1}{i_1-3}-\vv{x}{1}{i_1+2}}\right)\left(\frac{x^*_2-\vv{x}{2}{i_2-3}}{\vv{x}{2}{i_2+2}-\vv{x}{2}{i_2-3}}\right)=C^4_{0,0}(x_1^*)C^4_{0,1}(x_2^*),\\
C^4_{(0,0),(1,1)}(\xe)&=\left(\frac{x^*_1-\vv{x}{1}{i_1-3}}{\vv{x}{1}{i_1+2}-\vv{x}{1}{i_1-3}}\right)\left(\frac{x^*_2-\vv{x}{2}{i_2-3}}{\vv{x}{2}{i_2+2}-\vv{x}{2}{i_2-3}}\right)=C^4_{0,1}(x_1^*)C^4_{0,1}(x_2^*),
\end{split}
\end{equation*}
and to define the new approximation changing both the non-linear weights, and the approximation to polynomials of degree 4
\begin{equation}\label{p5lineal}
\tilde{\mathcal{I}}^5(\xe;f):=\tilde{p}^{5}_{(0,0)}(\xe)=\sum_{\mathbf{j}_0\in\{0,1\}^2}\tilde{\omega}_{(0,0),\mathbf{j}_0}^4({\xe})\tilde{p}^{4}_{\mathbf{j}_0}(\xe),
\end{equation}
being $\tilde{p}^{4}_{\mathbf{j}_0}(\xe)$ obtained in Eq. \eqref{eqp4primerejemplonolineal} and
\begin{equation*}
\tilde{\omega}_{(0,0),\mathbf{j}_0}^4(\xe)=\frac{\alpha_{(0,0),\mathbf{j}_0}^4(\xe)}{\sum_{\mathbf{l}\in\{0,1\}^2}\alpha_{(0,0),\mathbf{l}}^4(\xe)},\,\,  \textrm{ where }\,\, \alpha_{(0,0),\mathbf{j}_0}^4(\xe)=\frac{C_{(0,0),\mathbf{j}_0}^4(\xe)}{(\epsilon+I_{(0,0),\mathbf{j}_0}^4)^t}, \quad \mathbf{j}_0\in\{0,1\}^2,
\end{equation*}
with the smoothness indicators defined in Eq. \eqref{equationindicadores} as:
\begin{equation*}
I_{(0,0),(0,0)}^4=I^3_{(0,0)},\quad I_{(0,0),(0,1)}^4=I^3_{(0,2)},\quad I_{(0,0),(1,0)}^4=I^3_{(2,0)},\quad I_{(0,0),(1,1)}^4=I^3_{(2,2)},
\end{equation*}
where $I^3_{\mathbf{j}_0}, \, \mathbf{j}_0\in\{0,1,2\}^2$ are smoothness indicators satisfying the properties \ref{P1sm}, \ref{P2sm}, and \ref{P3sm}.

%%Figura segundo paso weno progresivo
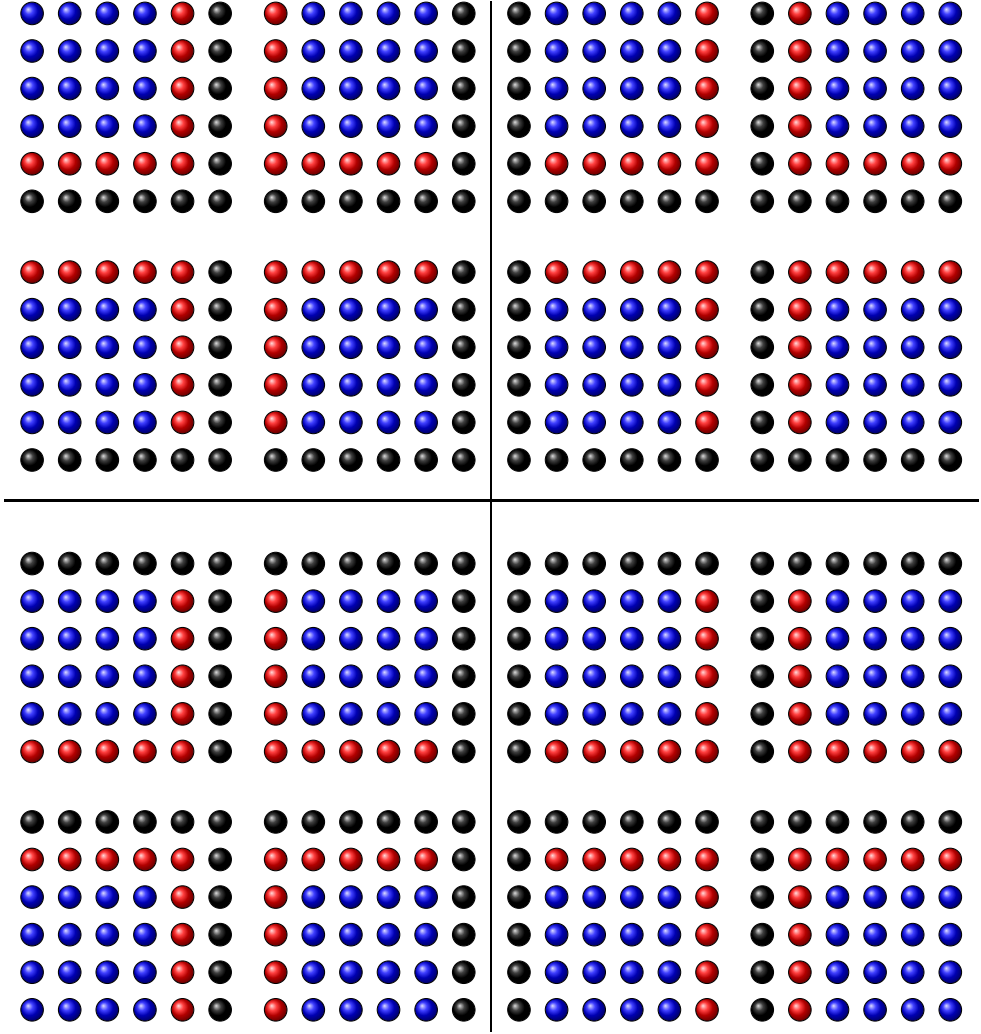
\begin{figure}[!hbtp]
\usetikzlibrary{arrows}
\begin{center}
  \input{figure4}
\end{center}
\caption{Contribution of each stencil of $p^{3}_{\mathbf{j}_1}(\xe)$, $\mathbf{j}_1\in \mathbf{j}_1+\{0,1\}^2$ to the approximation of $p^4_{\mathbf{j}_0}(\xe)$, $\mathbf{j}_0\in\{0,1\}^2$.}\label{progresive2}
\end{figure}

\subsection{The general case}\label{weno2dgeneralcase}
Now, we construct the new bivariate  WENO-$2r$ algorithm using a larger stencil of $(2r)^2$ points
$$\mathcal{S}^{2r}_{(0,0)}=\{\vv{x}{1}{i_1-r},\vv{x}{1}{i_1-r+1},\hdots,\vv{x}{1}{i_1+r-1}\}\times \{\vv{x}{2}{i_2-r},\vv{x}{2}{i_2-r+1},\hdots,\vv{x}{2}{i_2+r-1}\}=(\mathcal{S}^{2r}_0)_1\times (\mathcal{S}^{2r}_0)_2.$$
The goal is to obtain maximum order of approximation if there exists an isolated discontinuity. We start representing the polynomial of degree $2r-1$ as follows:
$$p^{2r-1}_{(0,0)}=\sum_{\mathbf{j}_0\in\{0,1\}^2} C^{2r-2}_{(0,0),\mathbf{j}_0}p^{2r-2}_{\mathbf{j}_0},$$
and repeat the process up to degree $r$:
\begin{equation}\label{eqsuperimp02d}
\begin{split}
p^{2r-1}_{(0,0)}=&\sum_{\mathbf{j}_0\in\{0,1\}^2} C^{2r-2}_{(0,0),\mathbf{j}_0}\left(\sum_{\mathbf{j}_1\in\Omega_0}C_{\mathbf{j}_0,\mathbf{j}_1}^{2r-3}\left(\sum_{\mathbf{j}_2\in \Omega_1}C_{\mathbf{j}_1,\mathbf{j}_2}^{2r-4}\left(\hdots \left(\sum_{\mathbf{j}_{r-3}\in \Omega_{r-4}}C^{r+1}_{\mathbf{j}_{r-4},\mathbf{j}_{r-3}}\left(\sum_{\mathbf{j}_{r-2}\in \Omega_{r-3}}C^{r}_{\mathbf{j}_{r-3},\mathbf{j}_{r-2}}p^r_{\mathbf{j}_{r-2}} \right)\right)\dots\right)\right)\right),
\end{split}
\end{equation}
being $\Omega_l=\mathbf{j}_l+\{0,1\}^2$ with $1\leq l\leq r-3$. Consequently, we determine the values $C^{l}_{\mathbf{j},\mathbf{j}_1}$ with $r\leq l \leq 2r-2$, and $\mathbf{j}_1\in\mathbf{j}+\{0,1\}^2$, proving the following lemma.

\begin{lemma}\label{lemaauxiliar12d}
Let  $r\leq l \leq 2r-2$; $\xe\in [\vv{x}{1}{i_1-1},\vv{x}{1}{i_1}]\times [\vv{x}{2}{i_2-1},\vv{x}{2}{i_2}]$, and $0\leq j_1,j_2\leq (2r-2)-l$, if we denote as $C^l_{\mathbf{j},\mathbf{j}_1}$ with $\mathbf{j}_1\in\mathbf{j}+\{0,1\}^2$ the values which satisfy:
\begin{equation}
p^{l+1}_{\mathbf{j}}(\xe)=\sum_{\mathbf{j}_1\in\mathbf{j}+\{0,1\}^2}C^l_{\mathbf{j},\mathbf{j}_1}(\xe)p_{\mathbf{j}_1}^{l}(\xe),
\end{equation}
then
\begin{equation*}
\begin{split}
&C^l_{\mathbf{j},\mathbf{j}+(0,0)}(\xe)=\left(\frac{x^*_1-\vv{x}{1}{i_1-r+j_1+l+1}}{\vv{x}{1}{i_1-r+j_1}-\vv{x}{1}{i_1-r+j_1+l+1}}\right)\left(\frac{x^*_2-\vv{x}{2}{i_2-r+j_2+l+1}}{\vv{x}{2}{i_2-r+j_2}-\vv{x}{2}{i_2-r+j_2+l+1}}\right)=C^l_{j_1,j_1}(x^*_1)C^l_{j_2,j_2}(x^*_2),\\
& C^l_{\mathbf{j},\mathbf{j}+(1,0)}(\xe)=\left(\frac{\vv{x}{1}{i_1-r+j_1}-x^*_1}{\vv{x}{1}{i_1-r+j_1}-\vv{x}{1}{i_1-r+j_1+l+1}}\right)\left(\frac{x^*_2-\vv{x}{2}{i_2-r+j_2+l+1}}{\vv{x}{2}{i_2-r+j_2}-\vv{x}{2}{i_2-r+j_2+l+1}}\right)=C^l_{j_1,j_1+1}(x^*_1)C^l_{j_2,j_2}(x^*_2),\\
&C^l_{\mathbf{j},\mathbf{j}+(0,1)}(\xe)=\left(\frac{x^*_1-\vv{x}{1}{i_1-r+j_1+l+1}}{\vv{x}{1}{i-r+j_1}-\vv{x}{1}{i_1-r+j_1+l+1}}\right)\left(\frac{\vv{x}{2}{i_2-r+j_2}-x^*_2}{\vv{x}{2}{i_2-r+j_2}-\vv{x}{2}{i_2-r+j_2+l+1}}\right)=C^l_{j_1,j_1}(x^*_1)C^l_{j_2,j_2+1}(x^*_2), \\ &C^l_{\mathbf{j},\mathbf{j}+(1,1)}(\xe)=\left(\frac{\vv{x}{1}{i-r+j_1}-x^*_1}{\vv{x}{1}{i-r+j_1}-\vv{x}{1}{i_1-r+j_1+l+1}}\right)\left(\frac{\vv{x}{2}{i_2-r+j_2}-x^*_2}{\vv{x}{2}{i_2-r+j_2}-\vv{x}{2}{i_2-r+j_2+l+1}}\right)=C^l_{j_1,j_1+1}(x^*_1)C^l_{j_2,j_2+1}(x^*_2).
\end{split}
\end{equation*}
\end{lemma}

\begin{proof}
The proof is direct, considering that the stencils used to obtain each interpolant are
$\{\vv{x}{1}{i_1-r+j_1},\dots,\vv{x}{1}{i_1-r+j_1+l+1}\}\times \{\vv{x}{2}{i_2-r+j_2},\dots,\vv{x}{2}{i_2-r+j_2+l+1}\}$ for $p_{\mathbf{j}}^{l+1}$, and
$$\{\vv{x}{1}{i_1-r+(j_1)_1},\dots,\vv{x}{1}{i_1-r+(j_1)_1+l}\}\times \{\vv{x}{2}{i_2-r+(j_2)_1},\dots,\vv{x}{2}{i_2-r+(j_2)_1+l}\}$$ for $p_{\mathbf{j}_1}^{l}$, with $\mathbf{j}_1 \in \mathbf{j}+\{0,1\}^2$, and the Aitken-Neville formula, see \cite{gascaaitkenclasica}.
\end{proof}

%\begin{proof}
%Let be $r\leq l \leq 2r-2$, $0\leq l_1,l_2\leq (2r-2)-l$ and $(x_1,x_2)\in [\vv{x}{1}{i_1-1},\vv{x}{1}{i_1}]\times [\vv{x}{2}{i_2-1},\vv{x}{2}{i_2}]$, the stencils used to obtain each interpolator are
%$\{\vv{x}{1}{i_1-r+l_1},\dots,\vv{x}{1}{i_1-r+l_1+l+1}\}\times \{\vv{x}{2}{i_2-r+l_2},\dots,\vv{x}{2}{i_2-r+l_2+l+1}\}$ for $p_{(l_1,l_2)}^{l+1}$ and
%$\{\vv{x}{1}{i_1-r+(l_1)_1},\dots,\vv{x}{1}{i_1-r+(l_1)_1+l}\}\times \{\vv{x}{2}{i_2-r+(l_2)_1},\dots,\vv{x}{2}{i_2-r+(l_2)_1+l}\}$ for $p_{((l_1)_1,(l_2)_1)}^{l}$ with $(l_1)_1=l_1,l_1+1$ and $(l_2)_1=l_2,l_2+1$. Then, using Aitken-Neville-type formula \cite{gascaaitkenclasica}, we get:
%\begin{equation*}
%\begin{split}
%p^{l+1}_{(l_1,l_2)}(x_1,x_2)=&\sum_{((l_1)_1,(l_2)_1)\in\{l_1,l_1+1\}\times\{l_2,l_2+1\}}C^l_{(l_1,l_2),((l_1)_1,(l_2)_1)}(x_1,x_2)p_{((l_1)_1,(l_2)_1)}^{l}(x_1,x_2)\\
%=&C^l_{l_1,l_1}(x_1)C^l_{l_2,l_2}(x_2)p^l_{(l_1,l_2)}(x_1,x_2)+C^l_{l_1,l_1+1}(x_1)C^l_{l_2,l_2}(x_2)p^l_{(l_1+1,l_2)}(x_1,x_2)\\
%&+C^l_{l_1,l_1}(x_1)C^l_{l_2,l_2+1}(x_2)p^l_{(l_1,l_2+1)}(x_1,x_2)+C^l_{l_1,l_1+1}(x_1)C^l_{l_2,l_2+1}(x_2)p^l_{(l_1+1,l_2+1)}(x_1,x_2).\\ \end{split}
%\end{equation*}
%\end{proof}

Note that if our data are equally spaced and we approximate at the mid-points, we recover the result showed in Cor. \ref{lemaauxiliar1}.
\begin{corollary}
Let  $r\leq l \leq 2r-2$, and $0\leq l_1,l_2\leq (2r-2)-l$, and $\vv{x}{1}{i_1}=a+k_1h_1, \,h_1=(b_1-a_1)/J_1$, $k_1=0,\hdots,J_1$, and $\vv{x}{2}{k_2}=a_2+k_2h_2, \, h_2=(b_2-a_2)/J_2$, $k_2=0,\hdots,J_2$,
if we denote as $\mathbf{x}_{\frac{1}{2}}=(\vv{x}{1}{i_1-1/2},\vv{x}{2}{i_2-1/2})$, and $C^l_{\mathbf{j},\mathbf{j}_1}(\mathbf{x}_{\frac{1}{2}})$ the values which satisfy:
\begin{equation}
p^{l+1}_{\mathbf{j}}(\mathbf{x}_{\frac{1}{2}})=\sum_{\mathbf{j}_1\in\mathbf{j}+\{0,1\}^2}C^l_{\mathbf{j},\mathbf{j}_1}(\mathbf{x}_{\frac{1}{2}})p_{\mathbf{j}_1}^{l}(\mathbf{x}_{\frac{1}{2}}),
\end{equation}
then, if we denote as $C^l_{\mathbf{j},\mathbf{j}_1}:=C^l_{\mathbf{j},\mathbf{j}_1}(\mathbf{x}_{\frac{1}{2}})$, we get
\begin{equation*}
\begin{split}
&C^l_{\mathbf{j},\mathbf{j}}=C^l_{j_1,j_1}C^l_{j_2,j_2}, \quad C^l_{\mathbf{j},\mathbf{j}+(1,0)}=C^l_{j_1,j_1+1}C^l_{j_2,j_2},\quad C^l_{\mathbf{j}+(0,1)}=C^l_{j_1,j_1}C^l_{j_2,j_2+1}, \quad C^l_{\mathbf{j}+(1,1)}=C^l_{j_1,j_1+1}C^l_{j_2,j_2+1},
\end{split}
\end{equation*}
being $C^l_{j_i,j_i}$, and $C^l_{j_i,j_i+1}$, $i=1,2$ the weights defined in Eq. \eqref{pesostodos}.
\end{corollary}
\begin{proof}
It is clear from Lemma \ref{lemaauxiliar12d}.
\end{proof}

Now, we determine the iterative process as Eq. \eqref{formulapep}
\begin{equation*}
\begin{split}
&  \tilde{p}^{r}_{\mathbf{j}_{r-2}}(\xe)=p^{r}_{\mathbf{j}_{r-2}}(\xe), \quad \mathbf{j}_{r-2}=\{0,\hdots,r-1\}^2,\\
&  \tilde{p}^{l+1}_{\mathbf{j}_{2r-l-3}}(\xe)=\sum_{\mathbf{j}_{2r-l-2}\in\mathbf{j}_{2r-l-3}+\{0,1\}^2}
  \tilde{\omega}^{l}_{\mathbf{j}_{2r-l-3},\mathbf{j}_{2r-l-2}}(x^*)\tilde{p}_{\mathbf{j}_{2r-l-2}}^{l}(x^*), \quad
    l=r,\dots,2r-2, \, \, \mathbf{j}_{2r-l-3}\in\{0,\dots,2r-2-l\}^2,\\
\end{split}
\end{equation*}
for $l=r,\dots,2r-2$, and $\mathbf{k} \in \{0,\hdots,2r-2-l\}^2$, and $\mathbf{k}_1\in\mathbf{k}+\{0,1\}^2$:
\begin{equation}\label{pesosr2d}
\begin{split}
&\tilde{\omega}^l_{\mathbf{k},\mathbf{k}_1}(\xe)=\left(\frac{\tilde{\alpha}_{\mathbf{k},\mathbf{k}_1}^l}{\tilde{\alpha}_{\mathbf{k},\mathbf{k}}^l+\tilde{\alpha}_{\mathbf{k},\mathbf{k}+(1,0)}^l+\tilde{\alpha}_{\mathbf{k},\mathbf{k}+(0,1)}^l+\tilde{\alpha}_{\mathbf{k},\mathbf{k}+(1,1)}^l}\right)(\xe),\\ &\tilde{\alpha}_{\mathbf{k},\mathbf{k}_1}^l(\xe)=\frac{{C}_{\mathbf{k},\mathbf{k}_1}^l(\xe)}{(\epsilon+I^l_{\mathbf{k},\mathbf{k}_1})^t}, \quad \mathbf{k}_1\in \mathbf{k}+\{0,1\}^2,
\end{split}
\end{equation}
where $ I^l_{\mathbf{k},\mathbf{k}_1}$ are the smoothness indicators determined by the following formula:
\begin{equation}\label{indices}
{I}^l_{\mathbf{k},\mathbf{k}_1}=\left\{
                       \begin{array}{ll}
                         I^r_{\mathbf{k}}, & \hbox{if } \mathbf{k}_1=\mathbf{k}, \\
                         I^r_{\mathbf{k}+(l-(r-1),0)}, & \hbox{if } \mathbf{k}_1=\mathbf{k}+(1,0),\\
                         I^r_{\mathbf{k}+(0,l-(r-1))}, & \hbox{if } \mathbf{k}_1=\mathbf{k}+(0,1),\\
                         I^r_{\mathbf{k}+(l-(r-1),l-(r-1))}, & \hbox{if } \mathbf{k}_1=\mathbf{k}+(1,1),
                       \end{array}
                     \right.
\end{equation}
being $I^r_{\mathbf{k}}$, with $0\leq k_1,k_2\leq r-1$, smoothness indicators satisfying the properties \ref{P1sm}, \ref{P2sm}, and \ref{P3sm}.

Finally, the new bivariate progressive WENO approximation is
$$\tilde{\mathcal{I}}^{2r-1}\left(\xe;f\right)=\tilde{p}^{2r-1}_{(0,0)}(\xe).$$

\section{A new progressive multivariate WENO method}\label{wenond}

In this section we generalize the method designed \dioni{}{in Section \ref{wenoprogresivo1d} for $1d$, and in Section \ref{weno2dgeneralcase} for $2d$}. We follow the same steps: First of all, we consider the data as in Section \ref{linealnd}, i.e., if $\prod_{j=1}^n [a_j,b_j]$ is an hypercube and
$$a_j=\vv{x}{j}{0}<\vv{x}{j}{1}<\vv{x}{j}{2}<\hdots< \vv{x}{j}{J_j}=b_j, \quad j=1,\hdots,n,$$
are the points of a non-regular grid, we suppose our data as the evaluation of a unknown function at these points
$$f_{(l_1,\hdots,l_n)}=f(\vv{x}{1}{l_1},\vv{x}{2}{l_2},\hdots,\vv{x}{n}{l_n}), \quad 0\leq l_j\leq J_j,\quad 1\leq j\leq n.$$
Let  $(i_1,\hdots,i_n)\in\mathbb{N}^n$ and $r\in\mathbb{N}$ be such that $0\leq i_j-r \leq i_j+r-1\leq J_j$, $j=1,\hdots,n$, a point $\xe \in \prod_{j=1}^n [\vv{x}{j}{i_j-1},\vv{x}{j}{i_j}]$, and the centered stencil:
\begin{equation}
\mathcal{S}_{\mathbf{0}}^{2r}=\prod_{j=1}^n (\mathcal{S}_{0}^{2r})_j,
\end{equation}
with $\mathbf{0}=(0,\hdots,0)$. We compute the interpolatory polynomial of degree $2r-1$ with nodes $\mathcal{S}_{\mathbf{0}}^{2r}$, $p^{2r-1}_{\mathbf{0}}(\xe)$, and we express it as the combination of $2^n$ polynomials of degree $2r-2$. Thus, we get
\begin{equation}
p^{2r-1}_{\mathbf{0}}(\xe)=\sum_{\mathbf{j}_0\in\{0,1\}^n}C^{2r-2}_{\mathbf{0},\mathbf{j}_0}(\xe)p^{2r-2}_{\mathbf{j}_0}(\xe).
\end{equation}
Now, we repeat this process up to polynomials of degree $r$ (for simplicity remove the dependence of $\xe$), we denote as $\Omega_l=\mathbf{j}_l+\{0,1\}^n$:
\begin{equation}
\begin{split}
p^{2r-1}_{\mathbf{0}}&=\sum_{\mathbf{j}_0\in\{0,1\}^n}C^{2r-2}_{\mathbf{0},\mathbf{j}_0}p^{2r-2}_{\mathbf{j}_0}\\
&=\sum_{\mathbf{j}_0\in\{0,1\}^n}C^{2r-2}_{\mathbf{0},\mathbf{j}_0}\left(\sum_{\mathbf{j}_1\in \Omega_0}C^{2r-3}_{\mathbf{j}_0,\mathbf{j}_1} p^{2r-3}_{\mathbf{j}_1}\right)\\
&=\sum_{\mathbf{j}_0\in\{0,1\}^n}C^{2r-2}_{\mathbf{0},\mathbf{j}_0}\left(\sum_{\mathbf{j}_1\in \Omega_0}C^{2r-3}_{\mathbf{j}_0,\mathbf{j}_1} \left(\sum_{\mathbf{j}_2\in \Omega_1}C^{2r-4}_{\mathbf{j}_1,\mathbf{j}_2}p^{2r-4}_{\mathbf{j}_2}\right)\right)\\
&=\sum_{\mathbf{j}_0\in\{0,1\}^n}C^{2r-2}_{\mathbf{0},\mathbf{j}_0}\left(\sum_{\mathbf{j}_1\in \Omega_0}C^{2r-3}_{\mathbf{j}_0,\mathbf{j}_1} \left(\hdots\left( \sum_{\mathbf{j}_{r-2}\in \Omega_{r-3}}C^{r}_{\mathbf{j}_{r-3},\mathbf{j}_{r-2}}p^r_{\mathbf{j}_{r-2}}\right)\hdots\right)\right),\\
\end{split}
\end{equation}
with
$$C^{l}_{\mathbf{j},\mathbf{j}+\{0,1\}^n}=\prod_{i=1}^n C^{r}_{j_i,j_i+\{0,1\}},$$
being $C^l_{j_i,j_i}$, and $C^l_{j_i,j_i+1}$, $i=1,\hdots,n$ the weights defined in Eq. \eqref{pesostodos}.

We establish the recursive process as Eq. \eqref{formulapep}:
\begin{equation*}
\begin{split}
&  \tilde{p}^{r}_{\mathbf{j}_{r-2}}(\xe)=p^{r}_{\mathbf{j}_{r-2}}(\xe), \quad \mathbf{j}_{r-2}=\{0,\hdots,r-1\}^n,\\
&  \tilde{p}^{l+1}_{\mathbf{j}_{2r-l-3}}(\xe)=\sum_{\mathbf{j}_{2r-l-2}\in\mathbf{j}_{2r-l-3}+\{0,1\}^n}
  \tilde{\omega}^{l}_{\mathbf{j}_{2r-l-3},\mathbf{j}_{2r-l-2}}(x^*)\tilde{p}_{\mathbf{j}_{2r-l-2}}^{l}(x^*), \quad
    l=r,\dots,2r-2, \, \, \mathbf{j}_{2r-l-3}\in\{0,\dots,2r-2-l\}^n,\\
\end{split}
\end{equation*}
being for $l=r,\dots,2r-2$, and $\mathbf{k} \in\{0,\hdots,2r-2-l\}$, $\mathbf{k}_1\in\mathbf{k}+\{0,1\}^n$:
\begin{equation}\label{pesosr2d}
\begin{split}
\tilde{\omega}^l_{\mathbf{k},\mathbf{k}_1}(\xe)=\frac{\tilde{\alpha}_{\mathbf{k},\mathbf{k}_1}^l(\xe)}{\sum_{\mathbf{l}\in\{0,1\}^n}\tilde{\alpha}_{\mathbf{k},\mathbf{k}+\mathbf{l}}^l(\xe)},\quad \tilde{\alpha}_{\mathbf{k},\mathbf{k}_1}^l(\xe)=\frac{{C}_{\mathbf{k},\mathbf{k}_1}^l(\xe)}{(\epsilon+ I^l_{\mathbf{k},\mathbf{k}_1})^t}, \quad \mathbf{k}_1\in \mathbf{k}+\{0,1\}^n,
\end{split}
\end{equation}
where $I^l_{\mathbf{k},\mathbf{k}_1}$ are the smoothness indicators determined by the following formula:
\begin{equation}\label{indicesnd}
{I}^l_{\mathbf{k},\mathbf{k}_1}=I^r_{\mathbf{k}+(l-(r-1))\mathbf{v}},\quad  \text{if } \quad \mathbf{k}_1=\mathbf{k}+\mathbf{v},\quad \text{with} \quad \mathbf{v}\in\{0,1\}^n,
\end{equation}
being $I^r_{\mathbf{k}}$, with $0\leq k_j\leq r-1$, $j=1,\hdots,n$, smoothness indicators satisfying the properties \ref{P1sm}, \ref{P2sm}, and \ref{P3sm}.

Therefore, the approximation will be:
\begin{equation}\label{eqfinalndimensiones}
\tilde{\mathcal{I}}^{2r-1}(\xe)=\tilde{p}_{\mathbf{0}}^{2r-1}(\xe).
\end{equation}
Finally, we calculate the order of accuracy stating the next theorem.
\begin{theorem}\label{teo1_multivariate}
Let  $\mathbf{l}_0\in \{1,\hdots, r-1\}^n$, $l_0=\min_{j=1,\hdots,n} (\mathbf{l}_0)_j$, $\xe\in \prod_{j=1}^n [\vv{x}{j}{i_j-1},\vv{x}{j}{i_j}]$, and $\tilde{\mathcal{I}}^{2r-1}(\xe)$ the approximation defined in Eq. \eqref{eqfinalndimensiones}.  If $f$ is smooth in $\prod_{j=1}^n[\vv{x}{j}{i_j-r},\vv{x}{j}{i_j+r-1}]\setminus \Omega$, and $f$ has a discontinuity at $\Omega$ then
\begin{equation}
\tilde{\mathcal{I}}^{2r-1}(\xe)-{f}(\xe)=\left\{
                                                  \begin{array}{ll}
                                                    O(h^{2r}), & \hbox{if $\,\,\Omega=\emptyset$;} \\
O(h^{r+l_0}), & \hbox{if   $\,\,\Omega=\prod_{j=1}^n[\vv{x}{j}{i_j+(l_0)_j-1},\vv{x}{j}{i_j+(l_0)_j}]$. }\\
                                                  \end{array}
                                                \right.
\end{equation}
\end{theorem}

\section{Smoothness indicators}\label{smoothaccur}
In this section, we present some smoothness indicators which satisfy the above mentioned properties. We generalize the smoothness indicators introduced by Ar\`andiga et al. in \cite{arandigamuletrenau} which are an adaptation of the ones presented in \cite{ABBM}. The idea is to {design} some functionals that fulfill \ref{P1sm}, \ref{P2sm}, and \ref{P3sm}.

Given $\mathbf{x}^*\in\mathbb{R}^n$, $h_1,\dots,h_n>0$,  $\Phi(\mathbf{s})=\mathbf{x}^*+(s_1h_1,\dots,s_nh_n)$, $f$ sufficiently smooth
\begin{align*}
  &\int_{\Phi([0,1]^n)} (f\circ \Phi^{-1})^{(\boldsymbol{l})}(\mathbf{x})^2d\mathbf{x}=
  \int_{\Phi([0,1]^n)} \big(f^{(\boldsymbol{l})}(\Phi^{-1}(\mathbf{x}))h_1^{-l_1}\dots h_n^{-l_n}\big)^2d\mathbf{x}
  \\
  &=h_1^{-2l_1}\dots h_n^{-2l_n}  \int_{\Phi([0,1]^n)} f^{(\boldsymbol{l})}(\Phi^{-1}(\mathbf{x}))d\mathbf{x}
  =h_1^{-2l_1}\dots h_n^{-2l_n}  \int_{[0,1]^n} f^{(\boldsymbol{l})}(\mathbf{s}) h_1\dots h_nd\mathbf{s}
  \\
  &=h_1^{-2l_1+1}\dots h_n^{-2l_n+1}  \int_{[0,1]^n} f^{(\boldsymbol{l})}(\mathbf{s})d\mathbf{s}.
\end{align*}
Hence, we define
\begin{equation}\label{eq1indicator}
I_\Gamma(f)=\sum_{\mathbf{l}\in\mathcal{J}}{h_1^{2l_1-1}\dots h_n^{2l_n-1}}\int_{\Gamma} (f^{\mathbf{l})}(\mathbf{x}))^2d\mathbf{x},
\end{equation}
where $\Gamma=\prod_{j=1}^n [\vv{x}{j}{i_j-1},\vv{x}{j}{i_j}]$, $\mathcal{J}=\{0,1,\hdots,r\}^n\setminus \{\mathbf{0}\}$ and
  smoothness indicators
  \begin{align}\label{eq1indicator1}
I_\mathbf{k}^r=    I_\Gamma(p^r_\mathbf{k}),\quad \mathbf{k} \in \{0,1,\hdots,r-1\}^n,
  \end{align}
  to get an expression which is scale independent.
\begin{theorem}
The smoothness indicators defined in Eq. (\ref{eq1indicator}-\ref{eq1indicator1})
satisfy  properties \ref{P1sm}, \ref{P2sm}, and \ref{P3sm}.
\end{theorem}
\begin{proof}
  Since $|\Gamma|=h_1\hdots h_n$,
it is clear that if $f$ is smooth in $\Gamma$, then
$$I_\Gamma(f)=O(h^2).$$
We suppose that $p$ is a polynomial interpolant at a stencil of $(r+1)^n$ points, then
\begin{equation}
\begin{aligned}
  I_\Gamma(f)-I_\Gamma(p)&=\sum_{\mathbf{l}\in\mathcal{J}}{h_1^{2l_1-1}\dots h_n^{2l_n-1}}\int_{\Gamma} ((f^{\mathbf{l})}(\mathbf{x}))^2-(p^{\mathbf{l})}(\mathbf{x}))^2)d\mathbf{x}
  \\
  &=\sum_{\mathbf{l}\in\mathcal{J}}{h_1^{2l_1-1}\dots h_n^{2l_n-1}}\int_{\Gamma} E^{\mathbf{l})}(\mathbf{x})(2f^{\mathbf{l})}(\mathbf{x})-E^{\mathbf{l})}(\mathbf{x}))d\mathbf{x}\\
&=\sum_{\mathbf{l}\in\mathcal{J}}{h_1^{2l_1}\dots h_n^{2l_n}}O(h^{r+1-||\mathbf{l}||_\infty})\\
&=O(h^{r+1}).
\end{aligned}
\end{equation}
where $E$ is defined in Eq. \eqref{errorfuncion}, and $||\mathbf{l}||_\infty=\max_{j=1,\hdots,n}|l_j|$.
Then, if $\mathbf{k},\mathbf{k}'\in \{0,\hdots,r-1\}^n$, we get:
$$I_\Gamma(p_\mathbf{k}^r)-I_\Gamma(p_{\mathbf{k}'}^r)=I_\Gamma(p_\mathbf{k}^r)-I_\Gamma(f)+I_\Gamma(f)-I_\Gamma(p_{\mathbf{k}'}^r)=O(h^{r+1}).$$
Finally, if a discontinuity crosses the stencil $\mathcal{S}^r_{\mathbf{k}}$ then:
$$I^r_{\mathbf{k}}\nrightarrow 0 \,\, \text{as}\,\, h\to 0,$$
{since some of the quadratic terms of $I^r_{\mathbf{k}}$ will not converge to $0$.}
\end{proof}
\begin{remark}
An evaluation of these smoothness indicators for $2d$ in gridded data is shown in \cite{arandigamuletrenau}. Also, some adaptation to smoothness indicators with better capabilities, and computationally more efficient than those introduced in \cite{WENO_nuevo}, or in \cite{baezaburgermuletzorio} can be performed, but the process requires an study on the order of accuracy, which exceeds the scope of this paper.
 \end{remark}

\section{Numerical experiments}\label{numericalexps}

In this section, we check our theoretical results through some numerical examples. We divide it into four parts, starting with 1$d$ experiments for uniform and non-uniform grids and performing some tests for the multivariable case. In both cases, the order of accuracy is analyzed measuring the error
in a set of points with a determined grid and refining it. For this purpose we define the error  {in a finite set} $\Upsilon\subset \mathbb{R}^n$ as:
\begin{equation}\label{deferror}
E(\Upsilon)=\max_{\mathbf{x}^*\in \Upsilon} |f(\mathbf{x}^*)-\tilde{\mathcal{I}}^{2r-1}(\mathbf{x}^*)|.
\end{equation}
Secondly, in all cases, we study the behavior of the resulting interpolator in the zones close to the discontinuities and conclude that by employing our new non-linear algorithm, if the discontinuity is isolated, the Gibbs phenomenon is avoided.

\subsection{Examples in 1D for uniform grids}\label{sec1duniforme}

We perform the first experiment using the same function studied in \cite{ARSY20}:
\begin{equation}\label{experimento1}
f_1(x)=\left\{\begin{array}{ll}
x^{10}-x^9+x^8-4x^7+x^6+x^5+x^4+x^3+5x^2+3x, &-\frac{\pi}{6}\le x<0,\\
1-(x^{10}-2x^9+3x^8-8x^7-2x^6+x^5-2x^4-3x^3-5x^2+0.5x),&0\le x<1-\frac{\pi}{6},
\end{array}
\right.
\end{equation}
discretizing it with $N_\ell=2^\ell+1$ points in the interval $[-\frac{\pi}{6},1-\frac{\pi}{6}]$. To analyze the order close to the discontinuity we locate the point 0 in each level $\ell$, compute the errors in the adjacent intervals, and estimate the numerical order of accuracy. Therefore, let $j_0^\ell\in\mathbb{N}$ such that
$0\in[x_{j_0^\ell-1},x_{j_0^\ell}]$, then we interpolate the value of the function $f$ at 10000 points equally spaced in the interval $[x_{j_0^\ell-5},x_{j_0^\ell+4}]$,
denoted by $\mathcal{J}^\ell$, determine the error using Eq. \eqref{deferror}, $E^\ell_s=E(\mathcal{J}^\ell\cap [x_{j_0^\ell+s-1},x_{j_0^\ell+s}])$ and approximate the numerical order in each interval $s$ as:
\begin{equation}\label{eqordernum}
o^\ell_s=\log_2\left(\frac{E^{\ell-1}_s}{E^{\ell}_s}\right), \quad s=-4,\hdots,4.
\end{equation}
This  allows us to analyze the progressively increasing order in the neighboring intervals to the {one} containing the discontinuity. For $r=3$, we can see in Table \ref{tabla1} that the order of accuracy increases as we proved in Theorem \ref{teo1_multivariate}.  Also, for $r=4$, Table \ref{tabla3}, we observe the same behavior of the numerical order. However, when we apply the classical WENO algorithm, in both cases $r=3,4$, Tables \ref{tabla2} and \ref{tabla4}, the numerical order is reduced to $r+1$, in all the intervals where one of the small stencils is contaminated by the singularity. Note that when we use our algorithm, we interpolate in any point of the interval, but when using the classical WENO method, we can only interpolate at the mid-point.

\begin{table}[!ht]
\begin{center}
\resizebox{18cm}{!} {
\begin{tabular}{|c|cc|cc|cc|cc|cc|cc|cc|cc|cc|}
\cline{1-19}
$\ell$ &$E_{-4}^\ell$ & $o_{-4}^\ell$ &$E_{-3}^\ell$ & $o_{-3}^\ell$ & $E_{-2}^\ell$ & $o_{-2}^\ell$ & $E_{-1}^\ell$ & $o_{-1}^\ell$ & $E_{0}^\ell$ & $o_{0}^\ell$ & $E_{1}^\ell$ & $o_{1}^\ell$ & $E_{2}^\ell$ & $o_{2}^\ell$  &$E_{3}^\ell$ & $o_{3}^\ell$ &$E_{4}^\ell$ & $o_{4}^\ell$             \\
\hline
5&   1.6751e-08 &&  1.3389e-08 &&  5.7904e-09 &&  7.3537e-07 &&  7.3774e-01 &&  1.7858e-06 &&  3.0257e-08 &&  2.4029e-08 &&  2.6927e-08   &                                       \\
6&   1.4582e-10 &6.84&  1.2149e-10 &6.78&  8.5220e-10 &2.76&  5.1819e-08 &3.82&  5.1486e-01 & 0.51&  1.1284e-07 &3.98&  3.5988e-10 &6.39&  3.4480e-10 & 6.12&  3.3124e-10   & 6.34\\
7&   1.4314e-12 &6.67&  1.2491e-12 &6.60&  3.6705e-11 &4.53&  3.5352e-09 &3.87&  9.8435e-01 &-0.93&  7.1647e-09 &3.97&  2.6416e-11 &3.76&  3.0225e-12 & 6.83&  3.6214e-12   & 6.51\\
8&   1.7396e-14 &6.36&  1.5998e-14 &6.28&  1.2246e-12 &4.90&  2.2682e-10 &3.96&  9.6727e-01 & 0.02&  4.5610e-10 &3.97&  1.0890e-12 &4.60&  3.4750e-14 & {\bf 6.44}&  3.7415e-14   & {\bf 6.59}\\
9&   2.3939e-16 &{\bf 6.18}&  2.2551e-16 &{\bf 6.14}&  3.9328e-14 &{\bf 4.96}&  1.4359e-11 &{\bf 3.98}&  9.3163e-01 & 0.05&  2.8800e-11 &{\bf 3.98}&  3.7748e-14 &{\bf 4.85}&  7.7716e-16 & 5.48&  8.8818e-16   & 5.39\\
\hline
\end{tabular}
}
\caption{Grid refinement analysis for the {\bf new} WENO-6 algorithm for the function in (\ref{experimento1}). }\label{tabla1}%The accuracy is reduced at the central interval of the stencil and around it step by step as spected.
\end{center}
\end{table}

\begin{table}[!ht]
\begin{center}
\resizebox{18cm}{!} {
\begin{tabular}{|c|cc|cc|cc|cc|cc|cc|cc|cc|cc|}
\cline{1-19} $\ell$ &$E_{-4}^\ell$ & $o_{-4}^\ell$ &$E_{-3}^\ell$ & $o_{-3}^\ell$ & $E_{-2}^\ell$ & $o_{-2}^\ell$ & $E_{-1}^\ell$ & $o_{-1}^\ell$ & $E_{0}^\ell$ & $o_{0}^\ell$ & $E_{1}^\ell$ & $o_{1}^\ell$ & $E_{2}^\ell$ & $o_{2}^\ell$  &$E_{3}^\ell$ & $o_{3}^\ell$ &$E_{4}^\ell$ & $o_{4}^\ell$             \\
\hline
5&    1.6749e-08 &&  1.3388e-08 &&  1.5644e-07 &&  6.7525e-07 &&  4.7375e-01 &&  1.6736e-06 &&  3.6670e-07 &&  2.4040e-08 &&  2.6932e-08  &                                      \\
6&    1.4582e-10 &6.84&  1.2148e-10 &6.78&  1.1614e-08 &3.75&  4.8161e-08 &3.80&  4.7959e-01 &-0.01&  1.0573e-07 &3.98&  2.4504e-08 &3.90&  3.4479e-10 &6.12 & 3.3126e-10  & 6.34\\
7&    1.4314e-12 &6.67&  1.2491e-12 &6.60&  7.8165e-10 &3.89&  3.3045e-09 &3.86&  5.0394e-01 &-0.07&  6.7145e-09 &3.97&  1.5610e-09 &3.97&  3.0224e-12 &6.83 & 3.6212e-12  & 6.51\\
8&    1.7396e-14 &6.36&  1.5987e-14 &6.28&  4.9619e-11 &3.97&  2.1244e-10 &3.95&  5.0215e-01 & 0.00&  4.2750e-10 &3.97&  9.9145e-11 &3.97&  3.4639e-14 &6.44 & 3.7192e-14   &6.60 \\
9&    2.3592e-16 &{\bf 6.20}&  2.2204e-16 &{\bf 6.16}&  3.1239e-12 &{\bf 3.98}&  1.3456e-11 &{\bf 3.98}&  5.0123e-01 & 0.00&  2.6997e-11 &{\bf 3.98}&  6.2474e-12 &{\bf 3.98}&  3.3307e-16 &{\bf 6.70} & 5.5511e-16   &\bf{6.06}\\
\hline
\end{tabular}
}
\caption{Grid refinement analysis for the {\bf classical} WENO-6 algorithm for the function in (\ref{experimento1}). }\label{tabla2}%The accuracy is reduced at the central interval of the stencil and around it step by step as spected.
\end{center}
\end{table}

\begin{table}[!ht]
\begin{center}
\resizebox{18cm}{!} {
\begin{tabular}{|c|cc|cc|cc|cc|cc|cc|cc|cc|cc|}
\cline{1-19} $\ell$ &$E_{-4}^\ell$ & $o_{-4}^\ell$ &$E_{-3}^\ell$ & $o_{-3}^\ell$ & $E_{-2}^\ell$ & $o_{-2}^\ell$ & $E_{-1}^\ell$ & $o_{-1}^\ell$ & $E_{0}^\ell$ & $o_{0}^\ell$ & $E_{1}^\ell$ & $o_{1}^\ell$ & $E_{2}^\ell$ & $o_{2}^\ell$  &$E_{3}^\ell$ & $o_{3}^\ell$ &$E_{4}^\ell$ & $o_{4}^\ell$             \\
\hline
5  & 2.2152e-11 &&  1.6644e-09 &&  1.8454e-08 &&  3.5396e-09 &&  7.0227e-01&&   4.1523e-08  && 2.8408e-08  && 2.2035e-09 &&  8.4167e-11   &                                             \\
6  & 9.4535e-14 &{\bf 7.87}&  1.2205e-11 &7.09&  1.6904e-10 &6.77&  2.2973e-09 &0.62&  5.3587e-01& 0.39&   1.2832e-09  &5.01& 3.1781e-10  &6.48& 1.9403e-11 &6.82&  3.3018e-13   &   {\bf  7.99}    \\
7  & 4.9960e-16 &7.56&  9.1455e-14 &7.06&  1.7535e-12 &6.59&  9.5206e-11 &4.59&  9.8422e-01&-0.87&   7.0822e-11  &4.17& 3.9031e-12  &6.34& 1.6120e-13 &{\bf 6.91}&  1.7764e-15   &    7.53    \\
8  & 2.0817e-17 &4.58&  7.0777e-16 &{\bf 7.01}&  2.2673e-14 &6.27&  3.1514e-12 &4.91&  9.6632e-01& 0.02&   2.8244e-12  &4.64& 4.8961e-14  &6.31& 1.7764e-15 &6.50&  5.5511e-16   &    1.67    \\
9  & 1.3878e-17 &0.58&  1.0408e-17 &6.08&  3.2092e-16 &{\bf 6.14}&  1.0091e-13 &{\bf 4.96}&  9.2930e-01& 0.05&   9.6589e-14  &{\bf 4.86}& 8.8818e-16  &{\bf 5.78}& 4.4409e-16 &2.00&  4.4409e-16   &    0.32    \\
\hline
\end{tabular}
}
\caption{Grid refinement analysis for the {\bf new} WENO-8 algorithm for the function in (\ref{experimento1}). }\label{tabla3}%The accuracy is reduced at the central interval of the stencil and around it step by step as spected.
\end{center}
\end{table}

\begin{table}[!ht]
\begin{center}
\resizebox{18cm}{!} {
\begin{tabular}{|c|cc|cc|cc|cc|cc|cc|cc|cc|cc|}
\cline{1-19} $\ell$ &$E_{-4}^\ell$ & $o_{-4}^\ell$ &$E_{-3}^\ell$ & $o_{-3}^\ell$ & $E_{-2}^\ell$ & $o_{-2}^\ell$ & $E_{-1}^\ell$ & $o_{-1}^\ell$ & $E_{0}^\ell$ & $o_{0}^\ell$ & $E_{1}^\ell$ & $o_{1}^\ell$ &
$E_{2}^\ell$ & $o_{2}^\ell$  &$E_{3}^\ell$ & $o_{3}^\ell$ &$E_{4}^\ell$ & $o_{4}^\ell$             \\
\hline
5&    2.2458e-11 &&  1.6502e-09 &&  4.1669e-09 &&  3.2826e-09 &&  4.5560e-01 &&  3.9398e-08 &&  2.2108e-09 &&  4.0566e-10 &&  8.1793e-11  &  \\
6&    9.4508e-14 &{\bf 7.89}&  1.5559e-10 &3.40&  5.6031e-10 &2.89&  2.0622e-09 &0.67&  4.7046e-01 &-0.04&  1.1327e-09 &5.12&  3.6677e-10 &2.59&  1.1083e-10 &1.87 & 3.2652e-13  & {\bf7.96}\\
7&    4.5797e-16 &7.68&  5.9243e-12 &4.71&  2.1924e-11 &4.67&  8.5879e-11 &4.58&  5.0851e-01 &-0.11&  6.3658e-11 &4.15&  1.6916e-11 &4.43&  4.7008e-12 &4.55 & 1.5543e-15  & 7.71\\
8&             0 &-   &  1.9244e-13 &4.94&  7.1713e-13 &4.93&  2.8456e-12 &4.91&  5.0443e-01 & 0.01&  2.5479e-12 &4.64&  6.4937e-13 &4.70&  1.7553e-13 &4.74 &          0   & -\\
9&    3.4694e-18 &-   &  6.1149e-15 &{\bf4.97}&  2.2865e-14 &{\bf4.97}&  9.1148e-14 &{\bf4.96}&  5.0237e-01 & 0.00&  8.7264e-14 &{\bf4.86}&  2.1871e-14 &{\bf4.89}&  5.6621e-15 &{\bf4.95} & 2.2204e-16   &-\\
\hline
\end{tabular}
}
\caption{Grid refinement analysis for the {\bf classical} WENO-8 algorithm for the function in (\ref{experimento1}). }\label{tabla4}%The accuracy is reduced at
%the central interval of the stencil and around it step by step as spected.
\end{center}
\end{table}

\subsection{Examples in 1D for non-uniform grids}
{In} this subsection, we perform two experiments: one with the function defined in \eqref{experimento1}
and, also, with the function studied in \cite{ABM}:
\begin{equation}\label{experimento2}
f_2(x)=\left\{\begin{array}{ll}
5(x-0.25)^3e^{x^2}, &0\le x<2/3,\\
1.5-(x-0.25)^3e^{x^2},&2/3\le x<1.
\end{array}
\right.
\end{equation}
We construct a non-uniform grid $\{x_i\}_{i=0}^{32}$ with $x_i\sim U[0,1]$, being $U[0,1]$ a uniform distribution in $[0,1]$ and interpolate
the values of the function in a uniform grid of 10000 points. The result is showed in Figure \ref{fig1}, we can see that the two interpolants avoid Gibbs-phenomenon.
\dioni{}{We compare our method with classical WENO adapted to non-uniform grids, we change the optimal weights in \eqref{opt_w} using the formulas presented in Lemma \ref{lemaauxiliar1general1d}. In Tables \ref{tabla5b} and \ref{tabla6b} we can observe similar results to those obtained for uniform grids: the order of accuracy in the adjacent cells to the
isolated discontinuity, $o_2^l$ and $o_{-2}^l$, is smaller than the one obtained using the new WENO algorithm}. These numerical results are consistent with the theoretical ones.

\begin{figure}[!ht]
\begin{center}
\begin{tabular}{cc}
\includegraphics[width=8.5cm]{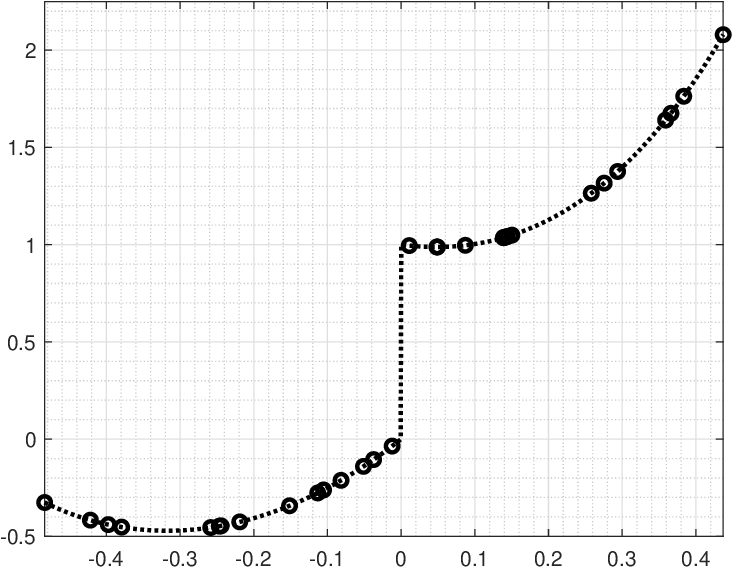}  & \includegraphics[width=8.5cm]{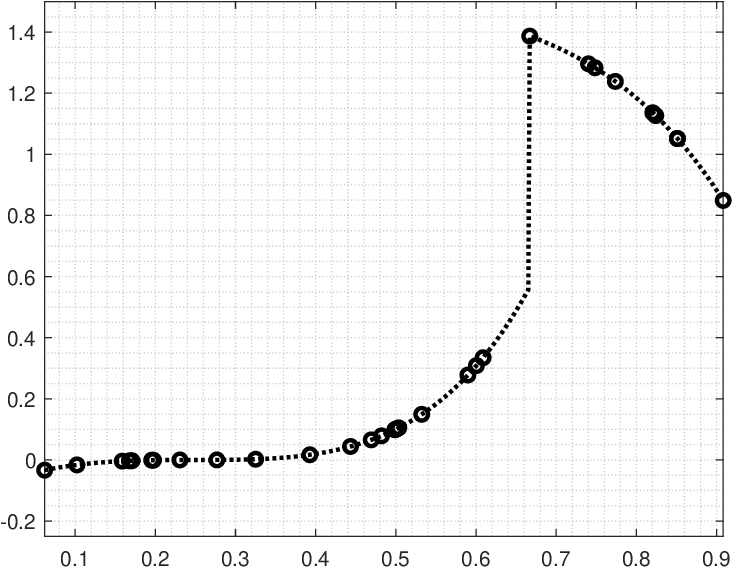} \\
 (a) & (b) \\
\end{tabular}
\end{center}
\caption{Interpolation using the new WENO-6 algorithm. Dashed line: interpolation of the function \eqref{experimento1} (a) and of the function \eqref{experimento2} (b). Circles: Data points.}
    \label{fig1}
 \end{figure}

To analyze the numerical order we take the non-uniform mesh, $\{x^5_i\}_{i=0}^{32}$, we divide it in each level using {the following} formula:
\begin{equation}\label{divisionmalla}
\begin{cases}
x^{\ell+1}_{2j+1}=\frac{1}{2}(x^\ell_j+x^\ell_{j+1}), & j=0,\hdots,2^\ell-1,\\
x^{\ell+1}_{2j}=x_j^\ell,& j=0,\hdots,2^\ell,
\end{cases}
\end{equation}
for $\ell\geq 5$ and compute the numerical order following Eq. \eqref{eqordernum}. Again, we see, in Tables \ref{tabla5} and \ref{tabla6}, that the behaviour of the numerical order is the expected one, and proved in Theorem  \ref{teo1_multivariate}.

\begin{table}[!ht]
\begin{center}
\resizebox{18cm}{!} {
\begin{tabular}{|c|cc|cc|cc|cc|cc|cc|cc|cc|cc|}
\cline{1-19} $\ell$ &$E_{-4}^\ell$ & $o_{-4}^\ell$ &$E_{-3}^\ell$ & $o_{-3}^\ell$ & $E_{-2}^\ell$ & $o_{-2}^\ell$ & $E_{-1}^\ell$ & $o_{-1}^\ell$ & $E_{0}^\ell$ & $o_{0}^\ell$ & $E_{1}^\ell$ & $o_{1}^\ell$ & $E_{2}^\ell$ & $o_{2}^\ell$  &$E_{3}^\ell$ & $o_{3}^\ell$ &$E_{4}^\ell$ & $o_{4}^\ell$             \\
\hline
5  &1.4532e-09 &&   1.9867e-09  && 1.3509e-09 &&  2.1725e-07 &&  5.0179e-01&&  9.6277e-07  &&  5.7583e-12  && 1.8499e-08 &&  2.0715e-08   &                                             \\
6  &3.0627e-12 &8.89&   1.1074e-11  &7.48& 2.8323e-10 &2.25&  1.8743e-08 &3.53&  9.5557e-01&-0.92&  1.4596e-07  &2.72&  1.0063e-10  &-4.12& 9.9920e-16 & 24.14&  8.8818e-16   &  24.47   \\
7  &3.4114e-13 &3.16&   2.7956e-13  &5.30& 9.6570e-12 &4.87&  1.1295e-09 &4.05&  9.0939e-01& 0.07&  4.4933e-09  &5.02&  5.2145e-11  & 0.94& 7.6797e-12 &-12.90&  6.9681e-12   & -12.93   \\
8  &3.1225e-15 &6.77&   2.6784e-15  &6.70& 2.8241e-13 &5.09&  6.9824e-11 &4.01&  8.1464e-01& 0.15&  1.4067e-10  &4.99&  3.2563e-13  & 7.32& 1.1768e-14 &  9.35&  5.5622e-14   &   6.96   \\
9  &4.1633e-17 &{\bf 6.22}&   4.1633e-17  &{\bf 6.00}& 8.9997e-15 &{\bf 4.97}&  4.4054e-12 &{\bf 3.98}&  5.9349e-01& 0.45&  8.8614e-12  &{\bf 3.98}&  9.1038e-15  & {\bf 5.16}& 4.4409e-16 &  {\bf 4.72}&  3.3307e-16   &   {\bf 7.38}   \\
\hline
\end{tabular}
}
\caption{Non-uniform grid refinement analysis for the {\bf new} WENO-6 algorithm for the function (\ref{experimento1})}\label{tabla5}%The accuracy is reduced at the central interval of the stencil and around it step by step as spected.
\end{center}
\end{table}

\begin{table}[!ht]
{    \color{blue}
\begin{center}
\resizebox{18cm}{!} {
\begin{tabular}{|c|cc|cc|cc|cc|cc|cc|cc|cc|cc|}
\cline{1-19} $\ell$ &$E_{-4}^\ell$ & $o_{-4}^\ell$ &$E_{-3}^\ell$ & $o_{-3}^\ell$ & $E_{-2}^\ell$ & $o_{-2}^\ell$ & $E_{-1}^\ell$ & $o_{-1}^\ell$ & $E_{0}^\ell$ & $o_{0}^\ell$ & $E_{1}^\ell$ & $o_{1}^\ell$ & $E_{2}^\ell$ & $o_{2}^\ell$  &$E_{3}^\ell$ & $o_{3}^\ell$ &$E_{4}^\ell$ & $o_{4}^\ell$             \\
\hline
5  &1.4532e-09&    &   1.9867e-09  &    & 2.2394e-09 &     &  2.1596e-07 &    &  5.0556e-01&     &  9.6277e-07  &    &  1.1280e-10  &     & 1.8422e-08 &      &2.0870e-08   &               \\
6  &3.0627e-12&8.89&   1.1074e-11  &7.48& 4.4299e-09 &-0.98&  1.8643e-08 &3.53&  9.5952e-01&-0.92&  1.4596e-07  &2.72&  9.9574e-09  &-6.46& 8.8818e-16 & 24.30&8.8818e-16&  24.48  \\
7  &3.4114e-13&3.16&   2.7956e-13  &5.30& 2.6139e-10 & 4.08&  1.1280e-09 &4.04&  9.1492e-01& 0.06&  4.4933e-09  &5.02&  2.7581e-09  & 1.85& 7.6801e-12 &-13.07&6.9664e-12& -12.93  \\
8  &3.1225e-15&6.77&   2.6784e-15  &6.70& 1.5284e-11 & 4.09&  6.9803e-11 &4.01&  8.1613e-01& 0.16&  1.4067e-10  &4.99&  3.4826e-11  & 6.30& 1.1879e-14 &  9.33&5.5511e-14&   6.97  \\
9  &4.1633e-17&{\bf6.22}&   4.3368e-17  &{\bf5.94}& 9.6033e-13 &{\bf 3.99}&  4.4051e-12 &{\bf3.98}&  5.9392e-01& 0.45&  8.8614e-12  &{\bf3.98}&  1.9277e-12  & {\bf4.17}& 4.4409e-16 &  {\bf4.74}&3.3307e-16&   {\bf7.38} \\
\hline
\end{tabular}
}
\caption{Non-uniform grid refinement analysis for the {\bf classical} WENO-6 algorithm for the function (\ref{experimento1})}\label{tabla5b}%The accuracy is reduced at the central interval of the stencil and around it step by step as spected.
\end{center}  }
\end{table}

 \begin{table}[!ht]
 \begin{center}
 \resizebox{18cm}{!} {
 \begin{tabular}{|c|cc|cc|cc|cc|cc|cc|cc|cc|cc|}
 \cline{1-19} $\ell$ &$E_{-4}^\ell$ & $o_{-4}^\ell$ &$E_{-3}^\ell$ & $o_{-3}^\ell$ & $E_{-2}^\ell$ & $o_{-2}^\ell$ & $E_{-1}^\ell$ & $o_{-1}^\ell$ & $E_{0}^\ell$ & $o_{0}^\ell$ & $E_{1}^\ell$ & $o_{1}^\ell$ & $E_{2}^\ell$ & $o_{2}^\ell$  &$E_{3}^\ell$ & $o_{3}^\ell$ &$E_{4}^\ell$ & $o_{4}^\ell$             \\
 \hline
 5  &1.6775e-08 &&   2.7033e-07  &&  6.2666e-05 &&  2.8007e-04 &&  8.2027e-01&&  5.6347e-05  &&  5.8642e-06  && 1.6321e-08 &&  4.6059e-08   &                                             \\
 6  &3.5056e-12 &12.22&   2.1402e-12  &16.94&  1.9801e-06 &4.98&  3.5278e-06 &6.31&  8.1413e-01&0.01&  1.0834e-05  &2.37&  1.3715e-07  &5.41& 1.3137e-11 &10.27&  4.4049e-12   &  13.35   \\
 7  &2.1300e-13 & 4.04&   1.4868e-10  &-6.11&  1.9147e-08 &6.69&  1.3796e-06 &1.35&  8.0577e-01&0.01&  9.3678e-07  &3.53&  2.7173e-08  &2.33& 8.3370e-10 &-5.98&  1.8362e-10   &  -5.38   \\
 8  &1.4874e-11 &-6.12&   1.5467e-11  & 3.26&  9.9600e-10 &4.26&  9.4274e-08 &3.87&  7.9070e-01&0.02&  5.5039e-08  &4.08&  7.7134e-10  &5.13& 1.6195e-11 & 5.68&  1.7014e-11   &   3.43   \\
 9  &2.5019e-13 & {\bf 5.89}&   2.5602e-13  & {\bf 5.91}&  3.2880e-11 &{\bf 4.92}&  6.1595e-09 &{\bf 3.93}&  7.5953e-01&0.05&  3.3358e-09  &{\bf 4.04}&  2.3013e-11  &{\bf 5.06}& 2.3137e-13 &{\bf  6.12}&  2.3692e-13   &   {\bf 6.16}   \\
 \hline
 \end{tabular}
 }
 \caption{Non-uniform grid refinement analysis for the {\bf new} WENO-6 algorithm for the function \eqref{experimento2}.}\label{tabla6}%The accuracy is reduced at the central interval of the stencil and around it step by step as spected.
 \end{center}
 \end{table}

\begin{table}[!ht]
{\color{blue}
\begin{center}
\resizebox{18cm}{!} {
\begin{tabular}{|c|cc|cc|cc|cc|cc|cc|cc|cc|cc|}
\cline{1-19} $\ell$ &$E_{-4}^\ell$ & $o_{-4}^\ell$ &$E_{-3}^\ell$ & $o_{-3}^\ell$ & $E_{-2}^\ell$ & $o_{-2}^\ell$ & $E_{-1}^\ell$ & $o_{-1}^\ell$ & $E_{0}^\ell$ & $o_{0}^\ell$ & $E_{1}^\ell$ & $o_{1}^\ell$ & $E_{2}^\ell$ & $o_{2}^\ell$  &$E_{3}^\ell$ & $o_{3}^\ell$ &$E_{4}^\ell$ & $o_{4}^\ell$             \\
\hline
5  &1.6383e-08&      &  2.8758e-07  &      &  1.0398e-04 &    &  3.9364e-04 &    &  8.1970e-01&     & 5.6347e-05  &    & 1.8207e-05  &   & 1.6304e-08 &      &4.7433e-08   &               \\
6  &3.5061e-12& 12.19&  2.1401e-12  & 17.03&  5.2570e-06 &4.30&  3.5206e-06 &6.80&  8.1338e-01&0.01&  1.0834e-05 &2.37&  8.1099e-07 &4.48& 1.3140e-11 & 24.30&4.4047e-12& 13.39  \\
7  &2.1300e-13&  4.04&  1.4842e-10  & -6.11&  2.3862e-07 &4.46&  1.3735e-06 &1.35&  8.0534e-01&0.01&  9.3678e-07 &3.53&  1.9806e-07 &2.03& 8.3530e-10 &-13.07&1.8435e-10& -5.38  \\
8  &1.4875e-11& -6.12&  1.5467e-11  &  3.26&  2.0923e-08 &3.51&  9.4127e-08 &3.86&  7.8938e-01&0.02&  5.5039e-08 &4.08&  1.2003e-08 &4.04& 1.6196e-11 &  9.33&1.7015e-11&  3.43  \\
9  &2.5013e-13&  {\bf5.89}&  2.5613e-13  & {\bf 5.91}&  1.3520e-09 &{\bf3.95}&  6.1566e-09 &{\bf3.93}&  7.5647e-01&0.06&  3.3358e-09 &{\bf4.04}&  7.3780e-10 &{\bf4.02}& 2.3115e-13 &  {\bf4.74}&2.3714e-13& {\bf 6.16} \\
\hline
\end{tabular}
}
\caption{Non-uniform grid refinement analysis for the {\bf classical} WENO-6 algorithm for the function (\ref{experimento2})}\label{tabla6b}%The accuracy is reduced at the central interval of the stencil and around it step by step as spected.
\end{center}}
\end{table}

\subsection{Examples in 2d for uniform {grids}}

We start with a smooth function to check numerically that the order is $2r$. Thus, we consider the function
\begin{equation}\label{experimento3}
f_3(x_1,x_2)=\frac{1}{x_1^2+x_2^2+1},
\end{equation}
and interpolate it using a Cartesian grid in the square $[-1,1]^2$:
\begin{equation}\label{gridcartesiana}
X^\ell=\{(x^{\ell,(i)}_1,x^{\ell,(j)}_2)\}_{i,j=0}^{2^\ell}, \quad x^{\ell,(i)}_1=-1+ih_\ell,\,\, x^{\ell,(j)}_2=-1+jh_\ell,\,\, h_\ell=2^{-\ell+1},
\end{equation}
{at} the points
$${B}^\ell=\{(\tau(k_1) x_1+(1-\tau(k_1))y_1,\tau(k_2) x_2+(1-\tau(k_2))y_2):(x_1,x_2),(y_1,y_2)\in {X}^\ell,\, \tau(k)=0.3+0.1k,\,k=0,4\},$$
%
%$B^{\ell+1}=\{(x_1^{\ell+1,(l_1)},x_2^{\ell+1,(l_2)})\}_{l_1,l_2=0}^{2^{\ell+1}}$ with
%\begin{equation*}
%\begin{split}
%x_1^{\ell+1,(l_1)}&=\begin{cases}
%0.3x_1^{\ell,(i)}+0.7x_1^{\ell,(i+1)}, & l_1=2i,\,\,0\le i\le 2^{\ell+1},\\
%0.7x_1^{\ell,(i)}+0.3x_1^{\ell,(i+1)}, & l_1=2i+1,\,\,1\le i\le 2^{\ell+1}-1,\\
%\end{cases}\\
%x_2^{\ell+1,(l_2)}&=\begin{cases}
%0.3x_2^{\ell,(j)}+0.7x_2^{\ell,(j+1)}, & l_2=2j,\,\,0\le j\le 2^{\ell+1},\\
%0.7x_2^{\ell,(j)}+0.3x_2^{\ell,(j+1)}, & l_2=2j+1,\,\,0\le j\le 2^{\ell+1}-1,\\
%\end{cases}
%\end{split}
%\end{equation*}
applying the new WENO algorithm for $r=3$. Note that the set of points $B^\ell$ is conformed by a convex combination of points of the grid. We have chosen
this collection of points, but we could have selected any other. The result is shown in Fig. \ref{fig2}, and we observe that it is similar to the original function $f_3$.
In order to compute the numerical order, we use the same strategy presented in Section \ref{sec1duniforme}, we calculate $E^\ell=E(B^\ell)$ and approximate the order as in Eq. \eqref{eqordernum}. Regarding the table presented in Fig. \ref{fig2}  {it is clear} that the order of accuracy is the expected one.
%\begin{equation}\label{eqordernum2d}
%o^\ell=\log_2\left(\frac{E^{\ell-1}}{E^{\ell}}\right).
%\end{equation}

\begin{figure}[H]
\centering
\begin{minipage}[c]{0.49\linewidth}
\centering
\begin{tabular}{|c|cc|cc|}\hline
&\multicolumn{2}{c|}{$r=3$} & \multicolumn{2}{c|}{$r=4$}\\\hline
$\ell$ & $E^\ell$ & $o^\ell$ & $E^\ell$ & $o^\ell$\\ \hline
4 &7.9879e-07    &  &   1.6882e-08 & \\
5 &   1.3037e-08 & 5.94&   6.4910e-11 &8.02 \\
6 &   2.0584e-10 & 5.98&   2.5102e-13 &{\bf 8.01} \\
7 &   3.2245e-12 & {\bf6.00}&   1.4433e-15 &7.44 \\
 \hline
\end{tabular}
\end{minipage}
\hfill
\begin{minipage}[c]{0.49\linewidth}
\centering
\includegraphics[width=8.5cm]{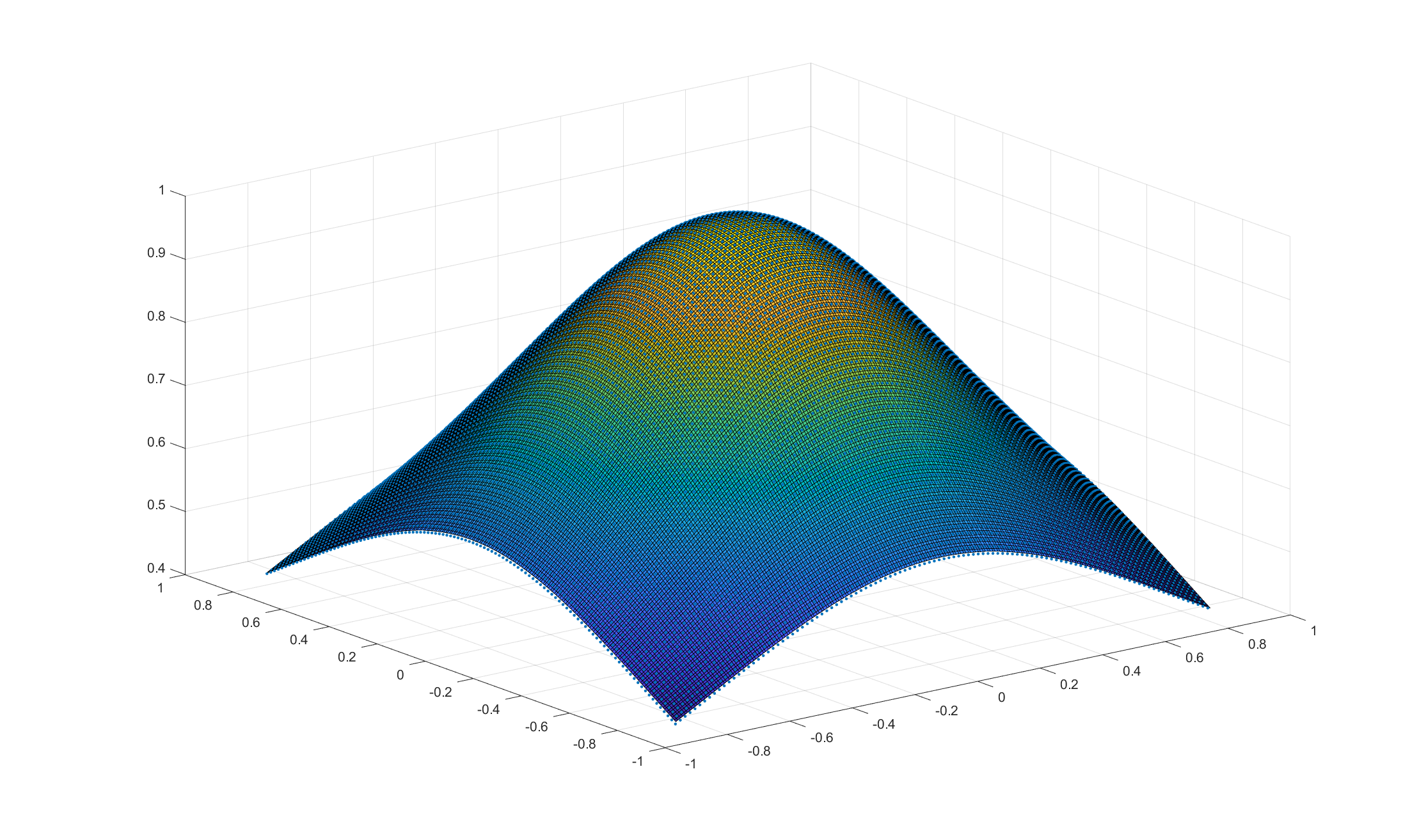}
\end{minipage}
\caption{Left: Grid refinement analysis for the new WENO-6 algorithm for the function \eqref{experimento3}. Right: Interpolation using the new WENO-6 algorithm.}
    \label{fig2}
\end{figure}

In order to analyze the order close to the discontinuities, we perform an experiment with the function:
 \begin{equation}\label{experimento4}
f_3(x_1,x_2)=\begin{cases}
e^{x_1+x_2}\cos(x_1-x_2),& x_1+x_2\leq 0,\\
e^{x_1+x_2}\cos(x_1-x_2)+1,& x_1+x_2> 0.\\
\end{cases}
\end{equation}
We take the mid-point $\mathbf{x}_0=(x^{\ell,(2^{\ell}-1)}_1,x^{\ell,(2^{\ell}-1)}_2)=(0,0)$ of the cartesian grid $X^\ell$; the set
$$Y^\ell=\{(s_1,s_2)h_\ell:-5\leq s_1\leq 6, \,-4\leq s_2\leq 5 \},$$
(see Figure \ref{fig3}.(a) red big points), and calculate an approximation to the function in the points
$$\mathcal{X}^\ell=\{(\tau(k_1) x_1+(1-\tau(k_1))y_1,\tau(k_2) x_2+(1-\tau(k_2))y_2):(x_1,x_2),(y_1,y_2)\in Y^\ell,\, \tau(k)=0.3+0.1k,\,k=0,\hdots,4\},$$
(see Figure \ref{fig3}.(a) gray, green, yellow and blue small points). We chose this set for the convenience of analyzing the order, but any other random sample of points could have been selected. A region between four data points is defined as:
$$\mathcal{X}^\ell_{(s_1,s_2)}=\mathcal{X}^\ell\cap [s_1h_\ell,(s_1+1)h_\ell]\times [s_2h_\ell,(s_2+1)h_\ell],\quad -5\leq s_1\leq 5,\,-4\leq s_2\leq 4,$$
and the errors and the numerical orders in each region as:
$$E^\ell_{(s_1,s_2)}=E(\mathcal{X}^\ell_{(s_1,s_2)}), \quad o^\ell_{(s_1,s_2)}=\log_2\left(\frac{E^{\ell-1}_{(s_1,s_2)}}{E^{\ell}_{(s_1,s_2)}}\right),\quad -5\leq s_1\leq 5,\,-4\leq s_2\leq 4.$$
\dioni{}{Now, we compare with a modification of the classical WENO method adapted to work using tensor products. Note that the WENO-2d method defined in \cite{arandigamuletrenau} is constructed for uniform-grids. In this case, we reformulate the classical WENO to approximate at any point in the considered interval and for non-uniform-grids.}
Theoretically, we have proved that the order of accuracy obtained at the green points in Figure \ref{fig3}.(a) is equal to 4, as there is one unique square stencil that is not contaminated by the discontinuity. For the yellow ones, 5 is the theoretical order of accuracy, and 6 for the rest of the points. \dioni{}{Using the classical WENO algorithm, we only
 distinguish two zones, order 4 (yellow and green zones) and order 6 for the rest of the points.}
\begin{figure}[!ht]
\begin{center}
\begin{tabular}{cc}
\includegraphics[width=8.5cm]{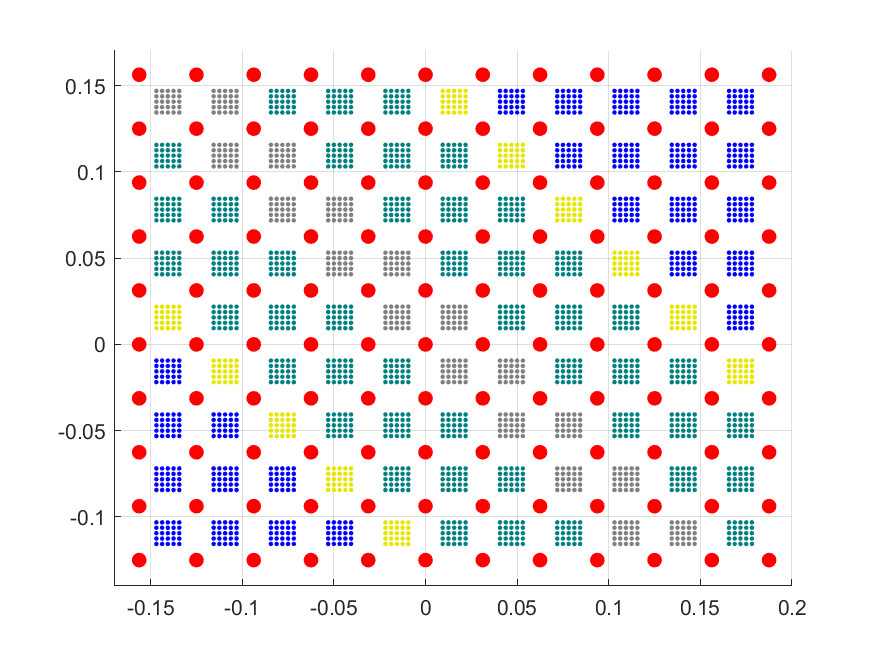}& \includegraphics[width=8.5cm]{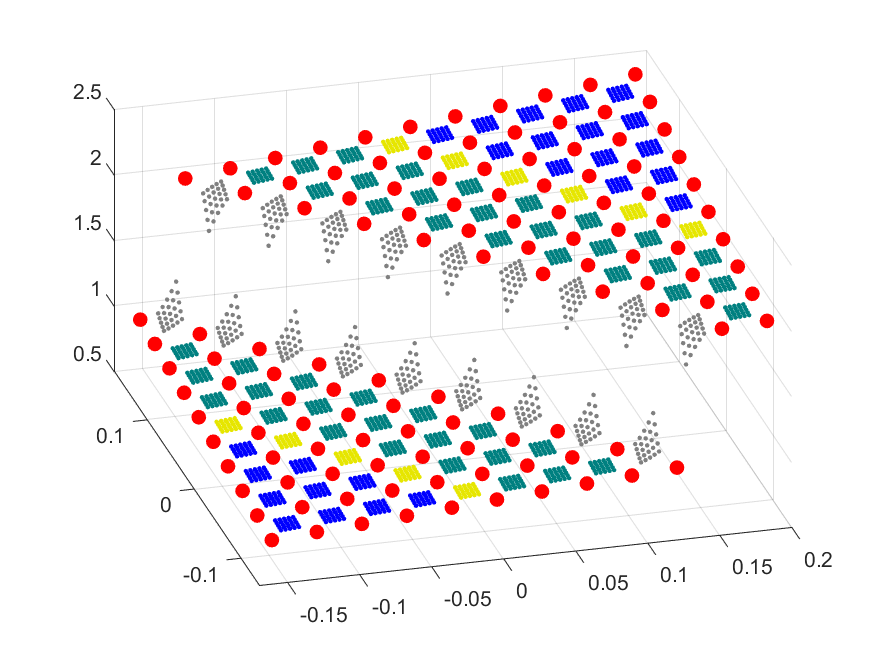} \\
(a) & (b)  \\
\includegraphics[width=8.5cm]{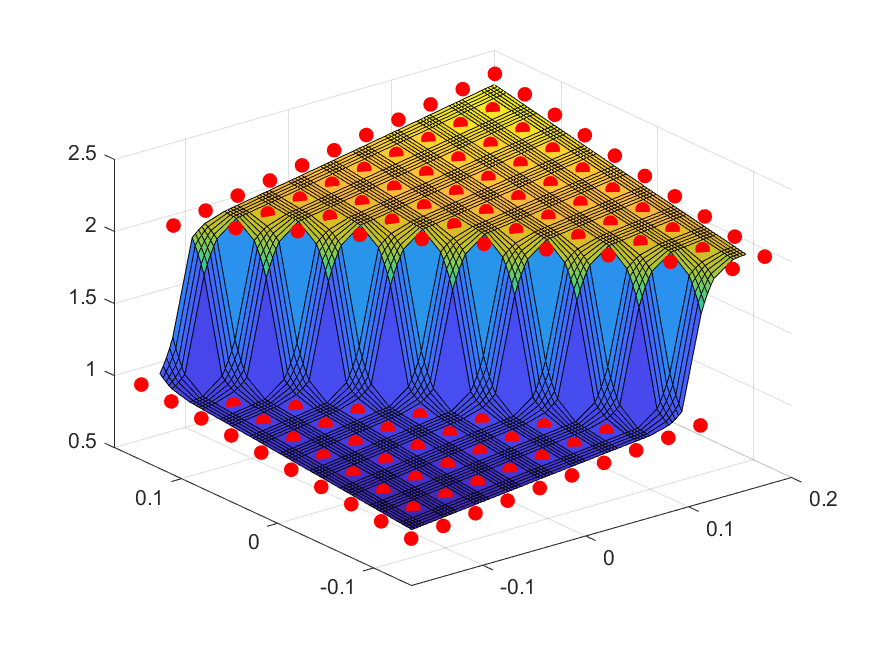}  & \includegraphics[width=8.5cm]{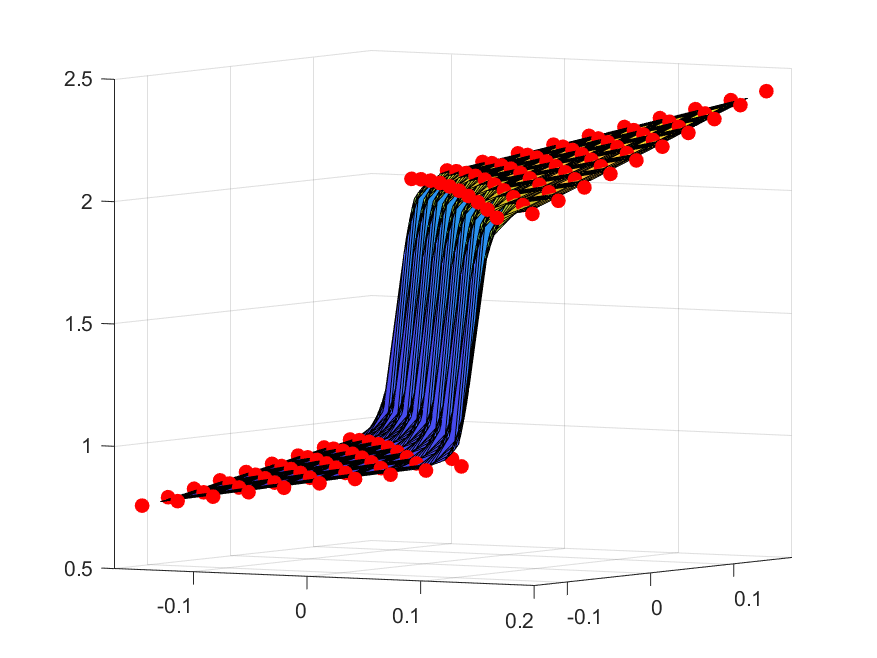}\\
 (c) & (d)
\end{tabular}
\end{center}
\caption{Interpolation using new WENO-6 algorithm to the function \eqref{experimento4}. (a): Nodes (red points) and points of interpolation (color points). (b), (c) and (d): Result of the interpolation in these points.}
    \label{fig3}
 \end{figure}

We can check this fact in Tables  \ref{tabla7} \dioni{}{and \ref{tabla7b}}, where we determine the order, $o_{(s_1,s_2)}^7$ in each region. In this case, in the smooth part, we obtain numerical order 7. The interpolation is shown in Figure \ref{fig3}.(c). We can see again that the resulting intepolator avoids Gibbs phenomenon (Figure \ref{fig3}.(d)).

\begin{table}[!ht]
 \begin{center}
 \begin{tabular}{|r|ccccccccccc|}   \hline
 $o^7_{(s_1,s_2)}$ & $s_1=-5$ & $-4$ & $-3$ & $-2$ & $-1$ & $0$ & $1$ & $2$ & $3$ & $4$ & $5$ \\\hline
  $s_2=-4$      & -0.00   &       0.00    &      4.03    &      4.07   &       3.99   &      {\bf 5.10}   &       7.07   &       7.10    &      7.32    &        6.81          &         7.09 \\
  $-3$       & 3.93   &      -0.00    &      0.00    &      4.05   &       4.08   &       3.99   &      {\bf 5.11}   &       7.23    &      7.09    &        7.00          &         7.08 \\
  $-2$       & 3.90   &       3.94    &     -0.00    &      0.00   &       4.06   &       4.09   &       4.00   &      {\bf 5.11}    &      7.06    &        7.21          &         7.25 \\
  $-1$       & 3.99   &       3.92    &      3.95    &     -0.00   &       0.00   &       4.06   &       4.09   &       4.00    &     {\bf 5.11}    &        7.07          &          7.20 \\
 $0$        & {\bf4.90}   &       4.00    &      3.92    &      3.96   &      -0.00   &       0.00   &       4.06   &       4.09    &      3.99    &       {\bf 5.10}     &        7.09   \\
 $1$        & 6.89   &       {\bf4.91}    &      4.01    &      3.93   &       3.96   &      -0.00   &       0.00   &       4.06    &      4.08    &        3.99          &         {\bf 5.09}  \\
 $2$        & 6.86   &       6.97    &     {\bf 4.91}    &      4.01   &       3.93   &       3.96   &      -0.00   &       0.00    &      4.05    &        4.07          &         3.97 \\
  $3$       & 6.88   &       6.93    &      6.96    &     {\bf 4.91}   &       4.01   &       3.92   &       3.95   &      -0.00    &      0.00    &        4.03          &         4.06    \\
  $4$       & 6.84   &       6.80    &      6.88    &      6.97   &       {\bf4.91}   &       4.00   &       3.92   &       3.94    &     -0.00    &        0.00          &          4.02 \\
\hline
 \end{tabular}
 \caption{Grid refinement analysis for the {\bf new} WENO-6 algorithm for the function \eqref{experimento4} with $\ell=7$}\label{tabla7}%The accuracy is reduced at the central interval of the stencil and around it step by step as spected.
 \end{center}
 \end{table}

\begin{table}[!ht]
{\color{blue} \begin{center}
 \begin{tabular}{|r|ccccccccccc|}   \hline
 $o^7_{(s_1,s_2)}$ & $s_1=-5$ & $-4$ & $-3$ & $-2$ & $-1$ & $0$ & $1$ & $2$ & $3$ & $4$ & $5$ \\\hline
  $s_2=-4$  & -0.00   &  0.00 & 4.03& 4.07& 4.04& {\bf4.08}&  7.15 & 7.10& 7.15&6.98&7.09\\
  $-3$      &  3.92   & -0.00 & 0.00& 4.04& 4.08& 4.05&  {\bf4.08} & 7.06& 6.89&7.00&6.98\\
  $-2$      &  3.90   &  3.94 & 0.00& 0.00& 4.05& 4.08&  4.05 & {\bf4.09}& 6.91&7.21&7.25\\
  $-1$      &  3.94   &  3.91 & 3.95&-0.00& 0.00& 4.05&  4.09 & 4.05& {\bf4.08}&7.15&7.19\\
 $0$        &  {\bf3.91}   &  3.95 & 3.92& 3.95&-0.00& 0.00&  4.05 & 4.08& 4.05&{\bf4.08}&7.14\\
 $1$        &  6.89   &  {\bf3.92} & 3.96& 3.92& 3.96&-0.00&  0.00 & 4.05& 4.08&4.04&{\bf4.07}\\
 $2$        &  6.91   &  6.97 & {\bf3.92}& 3.96& 3.92& 3.95& -0.00 & 0.00& 4.04&4.07&4.02\\
  $3$       &  6.92   &  6.93 & 6.96& {\bf3.92}& 3.96& 3.92&  3.95 &-0.00& 0.00&4.03&4.05\\
  $4$       &  6.90   &  6.94 & 6.93& 6.90& {\bf3.92}& 3.95&  3.91 & 3.94&-0.00&0.00&4.01\\
\hline
 \end{tabular}
 \caption{Grid refinement analysis for the {\bf classical} WENO-6 algorithm for the function \eqref{experimento4} with $\ell=7$}\label{tabla7b}%The accuracy is reduced at the central interval of the stencil and around it step by step as spected.
 \end{center}    }
 \end{table}

\subsection{Examples in 2d for non-uniform grids}

Finally, in this subsection, we will present two examples to corroborate that our new algorithm  is also valid for non-uniform grids. To study the order we design a non-regular grid for $\ell=4$ applying the following formula
\begin{equation}\label{gridcartesiananouniforme}
\tilde{X}^\ell=\{(\tilde{x}^{\ell,(i)}_1,\tilde{x}^{\ell,(j)}_2)\}_{i,j=0}^{2^\ell}, \quad \tilde{x}^{\ell,(i)}_1=-1+ih_\ell+\varepsilon^\ell_i,\,\, \tilde{x}^{\ell,(j)}_2=-1+jh_\ell+\varepsilon^\ell_j,\,\, h_\ell=2^{-\ell+1},\,\, \varepsilon^\ell_i,\varepsilon^\ell_j \sim U\left[-\frac{h_\ell}{2},\frac{h_\ell}{2}\right]
\end{equation}
being $U\left[-\frac{h_\ell}{2},\frac{h_\ell}{2}\right]$ the uniform distribution in the interval $[-h_\ell/2,h_\ell/2]$. And for $\ell\geq 4$ we use the same strategy
as in 1d, Eq. \eqref{divisionmalla}, i.e., we define
 $\tilde{X}^{\ell+1}=\{(\tilde{x}^{\ell+1,(i)}_1,\tilde{x}^{\ell+1,(j)}_2)\}_{i,j=0}^{2^{\ell+1}}$ as
\begin{equation*}
\begin{split}
\tilde{x}_1^{\ell+1,(l_1)}&=\begin{cases}
\frac12(\tilde{x}_1^{\ell,(i)}+\tilde{x}_1^{\ell,(i+1)}), & l_1=2i+1,\,\,1\le i\le 2^{\ell+1}-1,\\
\tilde{x}_1^{\ell,(i)}, & l_1=2i,\,\,0\le i\le 2^{\ell+1},
\end{cases}\\
\tilde{x}_2^{\ell+1,(l_2)}&=\begin{cases}
\frac12(\tilde{x}_2^{\ell,(j)}+\tilde{x}_2^{\ell,(j+1)}),& l_2=2j+1,\,\,0\le j\le 2^{\ell+1}-1,\\
\tilde{x}_2^{\ell,(j)},  & l_2=2j,\,\,0\le j\le 2^{\ell+1},\\
\end{cases}
\end{split}
\end{equation*}
and interpolate at 25 points contained in each region formed by 4 points:
$$\tilde{\mathcal{X}}^\ell=\{(\tau(k_1) x_1+(1-\tau(k_1))y_1,\tau(k_2) x_2+(1-\tau(k_2))y_2):(x_1,x_2),(y_1,y_2)\in \tilde{X}^\ell,\, \tau(k)=0.3+0.1k,\,\,k=0,1,2,3,4\}.$$
Then we calculate the errors and numerical orders in $\Omega_\ell= \mathcal{X}^\ell\cap[-5.5h_\ell,5.5h_\ell]^2,$ as in previous subsections,
with $E^\ell=E(\Omega_\ell)$. %o^\ell=\log_2\left(\frac{E^{\ell-1}}{E^\ell}\right).$$
We can observe  again in Table \ref{tablaerrores5}, that the results obtained in the numerical experiments satisfy the theoretical ones.
\begin{table}[!ht]
\begin{center}
\begin{tabular}{|c|cc|cc|}\hline
&\multicolumn{2}{c|}{$r=3$} & \multicolumn{2}{c|}{$r=4$}\\\hline
$\ell$ & $E^\ell$ & $o^\ell$ & $E^\ell$ & $o^\ell$\\ \hline
%4 &   3.3752e-05    &  &   3.99059e-06 & \\
5 &      1.3938e-06 & &   2.3237e-08 & \\
6 &      2.9147e-08 & 5.58&   1.8794e-10 & 6.95 \\
7 &      5.4251e-10 & 5.75&   1.1793e-12 &7.32 \\
8 &      1.0811e-11  & 5.65  & 5.3291e-15& {\bf 7.79}\\
9 &      1.9518e-13  & {\bf 5.79}  & 5.5511e-16& -\\
 \hline
\end{tabular}
\end{center}
\caption{Non-uniform grid refinement analysis for the {\bf new} WENO-6 algorithm for the function (\ref{experimento3}).}\label{tablaerrores5}
\end{table}
To finish these numerical tests, we perform an interpolation of the function
\begin{equation}\label{experimento5}
f_4(x_1,x_2)=\begin{cases}
\frac{1}{16}(x_1+x_2)\sin(16\pi x_1)\sin(16\pi x_2),& x_1+x_2\leq 0,\\
\frac{1}{16}(x_1+x_2)\sin(16\pi x_1)\sin(16\pi x_2)+0.1,& x_1+x_2> 0,\\
\end{cases}
\end{equation}
at the points:
$$\tilde{B}^7=\{(\tau(k_1) x_1+(1-\tau(k_1))y_1,\tau(k_2) x_2+(1-\tau(k_2))y_2):(x_1,x_2),(y_1,y_2)\in \tilde{X}^7,\, \tau(k)=0.3+0.1k,\,k=0,1,2,3,4\}$$
being $\tilde{X}^7$ a non-uniform grid constructed employing the formula described in Eq. \eqref{gridcartesiananouniforme} (see Figure \ref{fig7}(a)). We show the result in Figure \ref{fig7}(c) and (d) confirming that Gibbs phenomenon does not appear.

\begin{figure}[!ht]
\begin{center}
\begin{tabular}{cc}
\includegraphics[width=8.5cm]{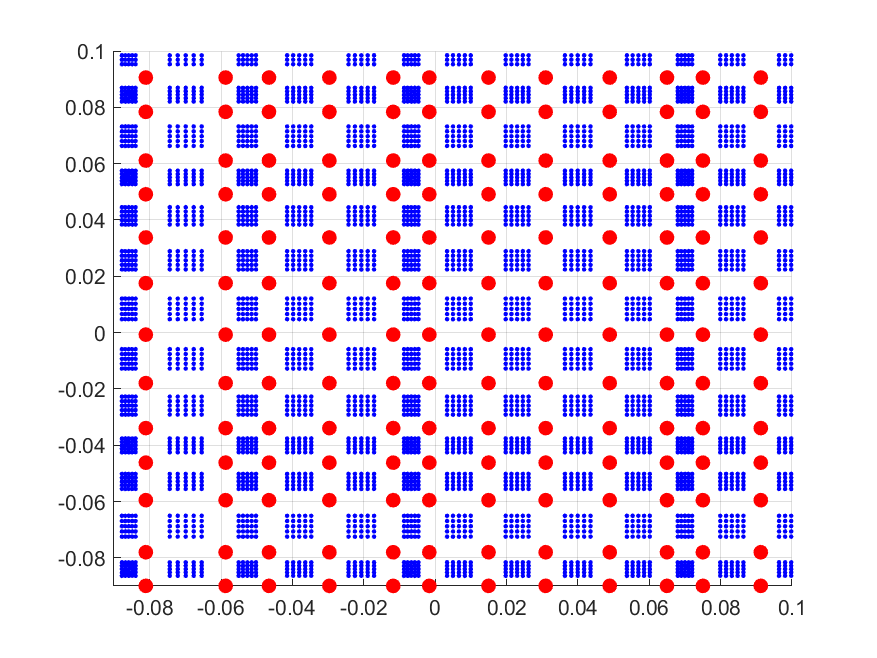} & \includegraphics[width=8.5cm]{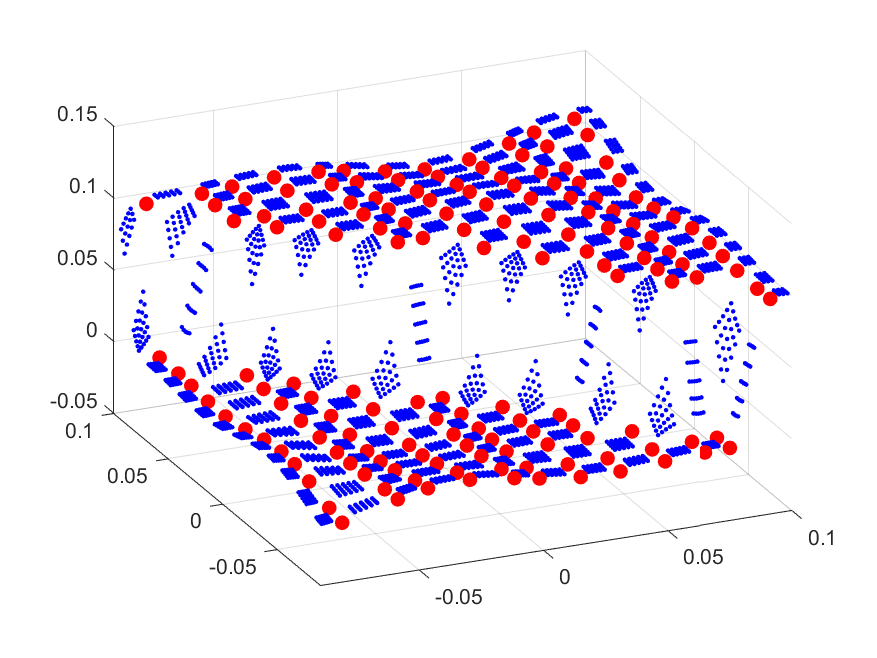}\\
 (a) & (b)\\
 \includegraphics[width=8.5cm]{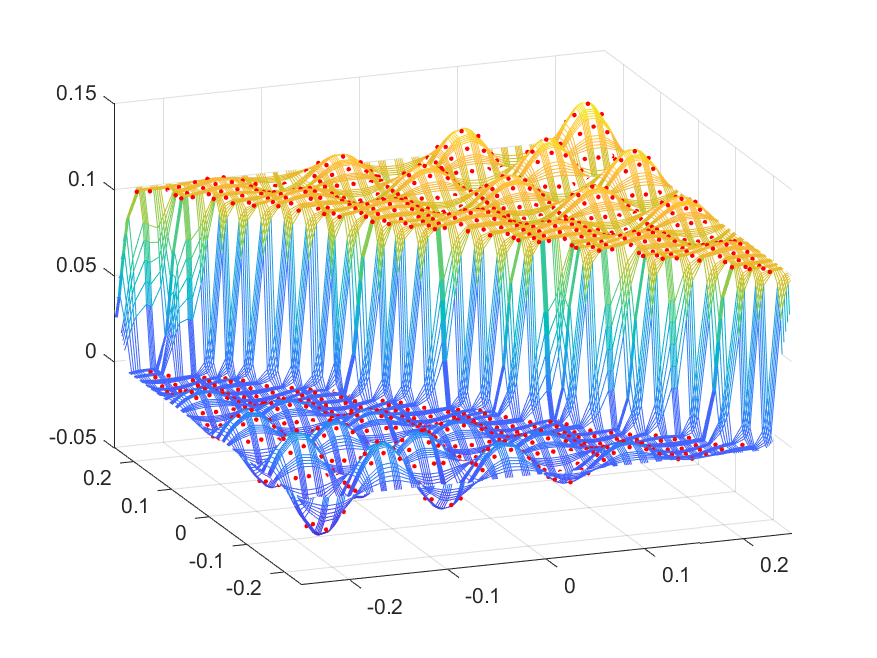} &
\includegraphics[width=8.5cm]{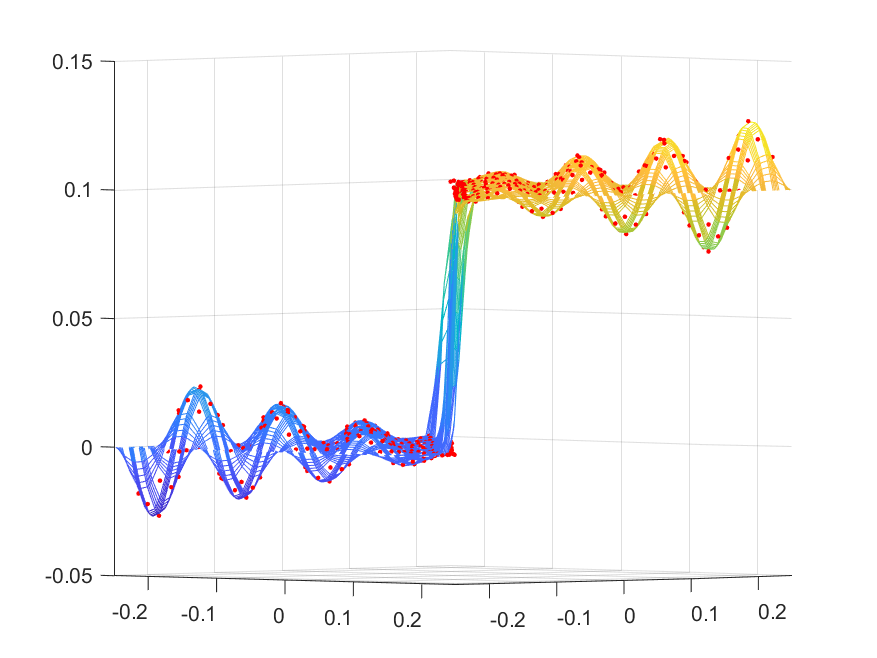}  \\
 (c) & (d)\\
\end{tabular}
\end{center}
\caption{(a) and (b) Non-uniform grid. Red: nodes, blue: points of interpolation. (c) and (d). Interpolation to the function \eqref{experimento5} using the new WENO-6 algorithm}
    \label{fig7}
 \end{figure}

\section{Conclusions and future work}

In this work, we have presented a new WENO method with a progressive order of accuracy close to the discontinuities. This new algorithm
is simpler computationally than the one presented in \cite{ARSY20} and more general. In particular, our main contributions can be summarised in the following: we design a new recursive method to compute the interpolation; we construct non-linear explicit formulas to calculate the weights to interpolate
any point in the central interval; we propose an algorithm for non-uniform grids, and finally, we extend the method to $n$ dimensions. This strategy presents a drawback, the stencils
used to interpolate are square and the number of points is larger than necessary to get the optimal order of accuracy. To improve this fact and
to apply this technique in partial differential equations contexts are the principal lines of our future work.

%\appendix
%
%\section{Kronecker product for general matrices}
%
%In this appendix we review the Kronecker product for general matrices in $\mathbb{R}^{r^n}$ and its properties:
%

\end{document}

%% file: figure1.tex
\begin{tikzpicture}
\draw [-] (-0.5,16) -- (0.5,18);
\draw [-] (-0.5,16) -- (0.5,14);
\draw [-] (-0.5,8) -- (0.5,10);
\draw [-] (-0.5,8) -- (0.5,6);

\draw [-] (1.5,18) -- (2.5,19);
\draw [-] (1.5,18) -- (2.5,17);
\draw [-] (1.5,14) -- (2.5,15);
\draw [-] (1.5,14) -- (2.5,13);
\draw [-] (1.5,10) -- (2.5,11);
\draw [-] (1.5,10) -- (2.5,9);
\draw [-] (1.5,6) -- (2.5,7);
\draw [-] (1.5,6) -- (2.5,5);

\draw [dashed] (3.5,19) -- ( 4.5,19.5);
\draw [dashed] (3.5,19) -- ( 4.5,18.5);
\draw [dashed] (3.5,17) -- ( 4.5,17.5);
\draw [dashed] (3.5,17) -- ( 4.5,16.5);
\draw [dashed] (3.5,15) -- ( 4.5,15.5);
\draw [dashed] (3.5,15) --  ( 4.5,14.5);
\draw [dashed] (3.5,13) --  ( 4.5,13.5);
\draw [dashed] (3.5,13) --  ( 4.5,12.5);
\draw [dashed] (3.5,11) -- ( 4.5,11.5);
\draw [dashed] (3.5,11) -- ( 4.5,10.5);
\draw [dashed] (3.5,9) -- ( 4.5,9.5);
\draw [dashed] (3.5,9) -- ( 4.5,8.5);
\draw [dashed] (3.5,7) -- ( 4.5,7.5);
\draw [dashed] (3.5,7) --  ( 4.5,6.5);
\draw [dashed] (3.5,5) --  ( 4.5,5.5);
\draw [dashed] (3.5,5) --  ( 4.5,4.5);

\draw [-] ( 6.5,21) -- ( 7.4,21.5);
\draw [-] ( 6.5,21) -- ( 7.4,20.5);
\draw [-] ( 6.5,19) -- ( 7.4,19.5);
\draw [-] ( 6.5,19) -- ( 7.4,18.5);
\draw [-] ( 6.5,17) -- ( 7.4,17.5);
\draw [-] ( 6.5,17) -- ( 7.4,16.5);
\draw [-] ( 6.5,15) -- ( 7.4,15.5);
\draw [-] ( 6.5,15) -- ( 7.4,14.5);
\draw [-] ( 6.5,13) -- ( 7.4,13.5);
\draw [-] ( 6.5,13) -- ( 7.4,12.5);
\draw [-] ( 6.5,11) -- ( 7.4,11.5);
\draw [-] ( 6.5,11) -- ( 7.4,10.5);
\draw [-] ( 6.5,9)  -- ( 7.4,9.5) ;
\draw [-] ( 6.5,9)  -- ( 7.4,8.5) ;
\draw [-] ( 6.7,5)  -- ( 6.9,5.5)   ;
\draw [-] ( 6.7,5)  -- ( 6.9,4.5) ;
\draw [-] ( 6.7,3)  -- ( 6.9,3.5);
\draw [-] ( 6.7,3)  -- ( 6.9,2.5);

%%%%%%%%%%%% escala primera
\path node at ( -1,16) {$C^{2r-2}_{0,0}$}
node at ( -1,8) {$C^{2r-2}_{0,1}$}

%%%%%%%%%%%%%% escala segunda
node at ( 1,18) {$C^{2r-3}_{0,0}$}
node at ( 1,14) {$C^{2r-3}_{0,1}$}
node at ( 1,10) {$C^{2r-3}_{1,1}$}
node at ( 1,6) {$C^{2r-3}_{1,2}$}

%%%%%%%%%%%%%% escala tercera
node at ( 3,19) {$C^{2r-4}_{0,0}$}
node at ( 3,17) {$C^{2r-4}_{0,1}$}
node at ( 3,15) {$C^{2r-4}_{1,1}$}
node at ( 3,13) {$C^{2r-4}_{1,2}$}
node at ( 3,11) {$C^{2r-4}_{1,1}$}
node at ( 3,9) {$C^{2r-4}_{1,2}$}
node at ( 3,7) {$C^{2r-4}_{2,2}$}
node at ( 3,5) {$C^{2r-4}_{2,3}$}

%%%%%%%%%%%%%% escala cuarta
node at ( 5,19.5) {$\hdots$}
node at ( 5,18.5) {$\hdots$}
node at ( 5,17.5) {$\hdots$}
node at ( 5,16.5) {$\hdots$}
node at ( 5,15.5) {$\hdots$}
node at ( 5,14.5) {$\hdots$}
node at ( 5,13.5) {$\hdots$}
node at ( 5,12.5) {$\hdots$}
node at ( 5,11.5) {$\hdots$}
node at ( 5,10.5) {$\hdots$}
node at ( 5,9.5) {$\hdots$}
node at ( 5,8.5) {$\hdots$}
node at ( 5,7.5) {$\hdots$}
node at ( 5,6.5) {$\hdots$}
node at ( 5,5.5) {$\hdots$}
node at ( 5,4.5) {$\hdots$}

%%%%%%%%%%%%%% escala penúltima
node at ( 6,21) {$C^{r+1}_{0,0}$}
node at ( 6,19) {$C^{r+1}_{1,1}$}
node at ( 6,17) {$C^{r+1}_{1,2}$}
node at ( 6,15) {$C^{r+1}_{1,1}$}
node at ( 6,13) {$C^{r+1}_{1,2}$}
node at ( 6,11) {$C^{r+1}_{1,2}$}
node at ( 6,9) {$C^{r+1}_{2,3}$}
node at ( 6,7) {$\vdots$}
node at ( 6,5) {$C^{r+1}_{r-3,r-3}$}
node at ( 6,3) {$C^{r+1}_{r-3,r-2}$}

%%%%%%%%%%%%%% escala última
node at ( 8,21.5) {$C^{r}_{0,0}p^r_0$}
node at ( 8,20.5) {$C^{r}_{0,1}p^r_1$}
node at ( 8,19.5) {$C^{r}_{1,1}p^r_1$}
node at ( 8,18.5) {$C^{r}_{1,2}p^r_2$}
node at ( 8,17.5) {$C^{r}_{2,2}p^r_2$}
node at ( 8,16.5) {$C^{r}_{2,3}p^r_3$}
node at ( 8,15.5) {$C^{r}_{1,1}p^r_1$}
node at ( 8,14.5) {$C^{r}_{1,2}p^r_2$}
node at ( 8,13.5) {$C^{r}_{2,2}p^r_2$}
node at ( 8,12.5) {$C^{r}_{2,3}p^r_3$}
node at ( 8,11.5) {$C^{r}_{2,2}p^r_2$}
node at ( 8,10.5) {$C^{r}_{2,3}p^r_3$}
node at ( 8,9.5) {$C^{r}_{3,3}p^r_3$}
node at ( 8,8.5) {$C^{r}_{3,4}p^r_4$}
node at ( 8,7)   {$\vdots$}
node at ( 8,5.5) {$C^{r}_{r-3,r-3}p^r_{r-3}$}
node at ( 8,4.5) {$C^{r}_{r-3,r-2}p^r_{r-2}$}
node at ( 8,3.5) {$C^{r}_{r-2,r-2}p^r_{r-2}$}
node at ( 8,2.5) {$C^{r}_{r-2,r-1}p^r_{r-1}$};
\end{tikzpicture}

%% file: figure3.tex
\begin{tabular}{c|c|c}
\begin{tikzpicture}
\draw[ball color=black] (0  ,0) circle (.15);
\draw[ball color=black] (0.5,0) circle (.15);
\draw[ball color=black] (1  ,0) circle (.15);
\draw[ball color=black] (1.5,0) circle (.15);
\draw[ball color=black]  (2  ,0) circle (.15);
\draw[ball color=black]  (2.5,0) circle (.15);

\draw[ball color=black] (0  ,0.5) circle (.15);
\draw[ball color=black] (0.5,0.5) circle (.15);
\draw[ball color=black] (1  ,0.5) circle (.15);
\draw[ball color=black] (1.5,0.5) circle (.15);
\draw[ball color=black]  (2  ,0.5) circle (.15);
\draw[ball color=black]  (2.5,0.5) circle (.15);

\draw[ball color=yellow](0  ,1) circle (.15);
\draw[ball color=yellow](0.5,1) circle (.15);
\draw[ball color=yellow] (1  ,1) circle (.15);
\draw[ball color=yellow] (1.5,1) circle (.15);
\draw[ball color=black] (2  ,1) circle (.15);
\draw[ball color=black] (2.5,1) circle (.15);

\draw[ball color=yellow](0  ,1.5) circle (.15);
\draw[ball color=yellow](0.5,1.5) circle (.15);
\draw[ball color=yellow] (1  ,1.5) circle (.15);
\draw[ball color=yellow] (1.5,1.5) circle (.15);
\draw[ball color=black] (2  ,1.5) circle (.15);
\draw[ball color=black] (2.5,1.5) circle (.15);

\draw[ball color=yellow] (0  ,2) circle (.15);
\draw[ball color=yellow] (0.5,2) circle (.15);
\draw[ball color=yellow] (1  ,2) circle (.15);
\draw[ball color=yellow] (1.5,2) circle (.15);
\draw[ball color=black] (2  ,2) circle (.15);
\draw[ball color=black] (2.5,2) circle (.15);

\draw[ball color=yellow] (0  ,2.5) circle (.15);
\draw[ball color=yellow] (0.5,2.5) circle (.15);
\draw[ball color=yellow] (1  ,2.5) circle (.15);
\draw[ball color=yellow] (1.5,2.5) circle (.15);
\draw[ball color=black] (2  ,2.5) circle (.15);
\draw[ball color=black] (2.5,2.5) circle (.15);
\end{tikzpicture}
&
\begin{tikzpicture}
\draw[ball color=black] (0  ,0) circle (.15);
\draw[ball color=black] (0.5,0) circle (.15);
\draw[ball color=black] (1  ,0) circle (.15);
\draw[ball color=black] (1.5,0) circle (.15);
\draw[ball color=black]  (2  ,0) circle (.15);
\draw[ball color=black]  (2.5,0) circle (.15);

\draw[ball color=black] (0  ,0.5) circle (.15);
\draw[ball color=black] (0.5,0.5) circle (.15);
\draw[ball color=black] (1  ,0.5) circle (.15);
\draw[ball color=black] (1.5,0.5) circle (.15);
\draw[ball color=black]  (2  ,0.5) circle (.15);
\draw[ball color=black]  (2.5,0.5) circle (.15);

\draw[ball color=black](0  ,1) circle (.15);
\draw[ball color=yellow](0.5,1) circle (.15);
\draw[ball color=yellow] (1  ,1) circle (.15);
\draw[ball color=yellow] (1.5,1) circle (.15);
\draw[ball color=yellow] (2  ,1) circle (.15);
\draw[ball color=black] (2.5,1) circle (.15);

\draw[ball color=black](0  ,1.5) circle (.15);
\draw[ball color=yellow](0.5,1.5) circle (.15);
\draw[ball color=yellow] (1  ,1.5) circle (.15);
\draw[ball color=yellow] (1.5,1.5) circle (.15);
\draw[ball color=yellow] (2  ,1.5) circle (.15);
\draw[ball color=black] (2.5,1.5) circle (.15);

\draw[ball color=black] (0  ,2) circle (.15);
\draw[ball color=yellow] (0.5,2) circle (.15);
\draw[ball color=yellow] (1  ,2) circle (.15);
\draw[ball color=yellow] (1.5,2) circle (.15);
\draw[ball color=yellow] (2  ,2) circle (.15);
\draw[ball color=black] (2.5,2) circle (.15);

\draw[ball color=black] (0  ,2.5) circle (.15);
\draw[ball color=yellow] (0.5,2.5) circle (.15);
\draw[ball color=yellow] (1  ,2.5) circle (.15);
\draw[ball color=yellow] (1.5,2.5) circle (.15);
\draw[ball color=yellow] (2  ,2.5) circle (.15);
\draw[ball color=black] (2.5,2.5) circle (.15);
\end{tikzpicture}
&
\begin{tikzpicture}
\draw[ball color=black] (0  ,0) circle (.15);
\draw[ball color=black] (0.5,0) circle (.15);
\draw[ball color=black] (1  ,0) circle (.15);
\draw[ball color=black] (1.5,0) circle (.15);
\draw[ball color=black]  (2  ,0) circle (.15);
\draw[ball color=black]  (2.5,0) circle (.15);

\draw[ball color=black] (0  ,0.5) circle (.15);
\draw[ball color=black] (0.5,0.5) circle (.15);
\draw[ball color=black] (1  ,0.5) circle (.15);
\draw[ball color=black] (1.5,0.5) circle (.15);
\draw[ball color=black]  (2  ,0.5) circle (.15);
\draw[ball color=black]  (2.5,0.5) circle (.15);

\draw[ball color=black](0  ,1) circle (.15);
\draw[ball color=black](0.5,1) circle (.15);
\draw[ball color=yellow] (1  ,1) circle (.15);
\draw[ball color=yellow] (1.5,1) circle (.15);
\draw[ball color=yellow] (2  ,1) circle (.15);
\draw[ball color=yellow] (2.5,1) circle (.15);

\draw[ball color=black](0  ,1.5) circle (.15);
\draw[ball color=black](0.5,1.5) circle (.15);
\draw[ball color=yellow] (1  ,1.5) circle (.15);
\draw[ball color=yellow] (1.5,1.5) circle (.15);
\draw[ball color=yellow] (2  ,1.5) circle (.15);
\draw[ball color=yellow] (2.5,1.5) circle (.15);

\draw[ball color=black] (0  ,2) circle (.15);
\draw[ball color=black] (0.5,2) circle (.15);
\draw[ball color=yellow] (1  ,2) circle (.15);
\draw[ball color=yellow] (1.5,2) circle (.15);
\draw[ball color=yellow] (2  ,2) circle (.15);
\draw[ball color=yellow] (2.5,2) circle (.15);

\draw[ball color=black] (0  ,2.5) circle (.15);
\draw[ball color=black] (0.5,2.5) circle (.15);
\draw[ball color=yellow] (1  ,2.5) circle (.15);
\draw[ball color=yellow] (1.5,2.5) circle (.15);
\draw[ball color=yellow] (2  ,2.5) circle (.15);
\draw[ball color=yellow] (2.5,2.5) circle (.15);
\end{tikzpicture}
\\[0.5cm]
\begin{tikzpicture}
\draw[ball color=black] (0  ,0) circle (.15);
\draw[ball color=black] (0.5,0) circle (.15);
\draw[ball color=black] (1  ,0) circle (.15);
\draw[ball color=black] (1.5,0) circle (.15);
\draw[ball color=black]  (2  ,0) circle (.15);
\draw[ball color=black]  (2.5,0) circle (.15);

\draw[ball color=yellow] (0  ,0.5) circle (.15);
\draw[ball color=yellow] (0.5,0.5) circle (.15);
\draw[ball color=yellow] (1  ,0.5) circle (.15);
\draw[ball color=yellow] (1.5,0.5) circle (.15);
\draw[ball color=black]  (2  ,0.5) circle (.15);
\draw[ball color=black]  (2.5,0.5) circle (.15);

\draw[ball color=yellow](0  ,1) circle (.15);
\draw[ball color=yellow](0.5,1) circle (.15);
\draw[ball color=yellow] (1  ,1) circle (.15);
\draw[ball color=yellow] (1.5,1) circle (.15);
\draw[ball color=black] (2  ,1) circle (.15);
\draw[ball color=black] (2.5,1) circle (.15);

\draw[ball color=yellow](0  ,1.5) circle (.15);
\draw[ball color=yellow](0.5,1.5) circle (.15);
\draw[ball color=yellow] (1  ,1.5) circle (.15);
\draw[ball color=yellow] (1.5,1.5) circle (.15);
\draw[ball color=black] (2  ,1.5) circle (.15);
\draw[ball color=black] (2.5,1.5) circle (.15);

\draw[ball color=yellow] (0  ,2) circle (.15);
\draw[ball color=yellow] (0.5,2) circle (.15);
\draw[ball color=yellow] (1  ,2) circle (.15);
\draw[ball color=yellow] (1.5,2) circle (.15);
\draw[ball color=black] (2  ,2) circle (.15);
\draw[ball color=black] (2.5,2) circle (.15);

\draw[ball color=black] (0  ,2.5) circle (.15);
\draw[ball color=black] (0.5,2.5) circle (.15);
\draw[ball color=black] (1  ,2.5) circle (.15);
\draw[ball color=black] (1.5,2.5) circle (.15);
\draw[ball color=black] (2  ,2.5) circle (.15);
\draw[ball color=black] (2.5,2.5) circle (.15);
\end{tikzpicture}
&
\begin{tikzpicture}
\draw[ball color=black] (0  ,0) circle (.15);
\draw[ball color=black] (0.5,0) circle (.15);
\draw[ball color=black] (1  ,0) circle (.15);
\draw[ball color=black] (1.5,0) circle (.15);
\draw[ball color=black]  (2  ,0) circle (.15);
\draw[ball color=black]  (2.5,0) circle (.15);

\draw[ball color=black] (0  ,0.5) circle (.15);
\draw[ball color=yellow] (0.5,0.5) circle (.15);
\draw[ball color=yellow] (1  ,0.5) circle (.15);
\draw[ball color=yellow] (1.5,0.5) circle (.15);
\draw[ball color=yellow]  (2  ,0.5) circle (.15);
\draw[ball color=black]  (2.5,0.5) circle (.15);

\draw[ball color=black](0  ,1) circle (.15);
\draw[ball color=yellow](0.5,1) circle (.15);
\draw[ball color=yellow] (1  ,1) circle (.15);
\draw[ball color=yellow] (1.5,1) circle (.15);
\draw[ball color=yellow] (2  ,1) circle (.15);
\draw[ball color=black] (2.5,1) circle (.15);

\draw[ball color=black](0  ,1.5) circle (.15);
\draw[ball color=yellow](0.5,1.5) circle (.15);
\draw[ball color=yellow] (1  ,1.5) circle (.15);
\draw[ball color=yellow] (1.5,1.5) circle (.15);
\draw[ball color=yellow] (2  ,1.5) circle (.15);
\draw[ball color=black] (2.5,1.5) circle (.15);

\draw[ball color=black] (0  ,2) circle (.15);
\draw[ball color=yellow] (0.5,2) circle (.15);
\draw[ball color=yellow] (1  ,2) circle (.15);
\draw[ball color=yellow] (1.5,2) circle (.15);
\draw[ball color=yellow] (2  ,2) circle (.15);
\draw[ball color=black] (2.5,2) circle (.15);

\draw[ball color=black] (0  ,2.5) circle (.15);
\draw[ball color=black] (0.5,2.5) circle (.15);
\draw[ball color=black] (1  ,2.5) circle (.15);
\draw[ball color=black] (1.5,2.5) circle (.15);
\draw[ball color=black] (2  ,2.5) circle (.15);
\draw[ball color=black] (2.5,2.5) circle (.15);
\end{tikzpicture}
&
\begin{tikzpicture}
\draw[ball color=black] (0  ,0) circle (.15);
\draw[ball color=black] (0.5,0) circle (.15);
\draw[ball color=black] (1  ,0) circle (.15);
\draw[ball color=black] (1.5,0) circle (.15);
\draw[ball color=black]  (2  ,0) circle (.15);
\draw[ball color=black]  (2.5,0) circle (.15);

\draw[ball color=black] (0  ,0.5) circle (.15);
\draw[ball color=black] (0.5,0.5) circle (.15);
\draw[ball color=yellow] (1  ,0.5) circle (.15);
\draw[ball color=yellow] (1.5,0.5) circle (.15);
\draw[ball color=yellow]  (2  ,0.5) circle (.15);
\draw[ball color=yellow]  (2.5,0.5) circle (.15);

\draw[ball color=black](0  ,1) circle (.15);
\draw[ball color=black](0.5,1) circle (.15);
\draw[ball color=yellow] (1  ,1) circle (.15);
\draw[ball color=yellow] (1.5,1) circle (.15);
\draw[ball color=yellow] (2  ,1) circle (.15);
\draw[ball color=yellow] (2.5,1) circle (.15);

\draw[ball color=black](0  ,1.5) circle (.15);
\draw[ball color=black](0.5,1.5) circle (.15);
\draw[ball color=yellow] (1  ,1.5) circle (.15);
\draw[ball color=yellow] (1.5,1.5) circle (.15);
\draw[ball color=yellow] (2  ,1.5) circle (.15);
\draw[ball color=yellow] (2.5,1.5) circle (.15);

\draw[ball color=black] (0  ,2) circle (.15);
\draw[ball color=black] (0.5,2) circle (.15);
\draw[ball color=yellow] (1  ,2) circle (.15);
\draw[ball color=yellow] (1.5,2) circle (.15);
\draw[ball color=yellow] (2  ,2) circle (.15);
\draw[ball color=yellow] (2.5,2) circle (.15);

\draw[ball color=black] (0  ,2.5) circle (.15);
\draw[ball color=black] (0.5,2.5) circle (.15);
\draw[ball color=black] (1  ,2.5) circle (.15);
\draw[ball color=black] (1.5,2.5) circle (.15);
\draw[ball color=black] (2  ,2.5) circle (.15);
\draw[ball color=black] (2.5,2.5) circle (.15);
\end{tikzpicture}
\\[0.5cm]
\begin{tikzpicture}
\draw[ball color=yellow] (0  ,0) circle (.15);
\draw[ball color=yellow] (0.5,0) circle (.15);
\draw[ball color=yellow] (1  ,0) circle (.15);
\draw[ball color=yellow] (1.5,0) circle (.15);
\draw[ball color=black]  (2  ,0) circle (.15);
\draw[ball color=black]  (2.5,0) circle (.15);

\draw[ball color=yellow] (0  ,0.5) circle (.15);
\draw[ball color=yellow] (0.5,0.5) circle (.15);
\draw[ball color=yellow] (1  ,0.5) circle (.15);
\draw[ball color=yellow] (1.5,0.5) circle (.15);
\draw[ball color=black]  (2  ,0.5) circle (.15);
\draw[ball color=black]  (2.5,0.5) circle (.15);

\draw[ball color=yellow](0  ,1) circle (.15);
\draw[ball color=yellow](0.5,1) circle (.15);
\draw[ball color=yellow] (1  ,1) circle (.15);
\draw[ball color=yellow] (1.5,1) circle (.15);
\draw[ball color=black] (2  ,1) circle (.15);
\draw[ball color=black] (2.5,1) circle (.15);

\draw[ball color=yellow](0  ,1.5) circle (.15);
\draw[ball color=yellow](0.5,1.5) circle (.15);
\draw[ball color=yellow] (1  ,1.5) circle (.15);
\draw[ball color=yellow] (1.5,1.5) circle (.15);
\draw[ball color=black] (2  ,1.5) circle (.15);
\draw[ball color=black] (2.5,1.5) circle (.15);

\draw[ball color=black] (0  ,2) circle (.15);
\draw[ball color=black] (0.5,2) circle (.15);
\draw[ball color=black] (1  ,2) circle (.15);
\draw[ball color=black] (1.5,2) circle (.15);
\draw[ball color=black] (2  ,2) circle (.15);
\draw[ball color=black] (2.5,2) circle (.15);

\draw[ball color=black] (0  ,2.5) circle (.15);
\draw[ball color=black] (0.5,2.5) circle (.15);
\draw[ball color=black] (1  ,2.5) circle (.15);
\draw[ball color=black] (1.5,2.5) circle (.15);
\draw[ball color=black] (2  ,2.5) circle (.15);
\draw[ball color=black] (2.5,2.5) circle (.15);
\end{tikzpicture}
&
\begin{tikzpicture}
\draw[ball color=black] (0  ,0) circle (.15);
\draw[ball color=yellow] (0.5,0) circle (.15);
\draw[ball color=yellow] (1  ,0) circle (.15);
\draw[ball color=yellow] (1.5,0) circle (.15);
\draw[ball color=yellow]  (2  ,0) circle (.15);
\draw[ball color=black]  (2.5,0) circle (.15);

\draw[ball color=black] (0  ,0.5) circle (.15);
\draw[ball color=yellow] (0.5,0.5) circle (.15);
\draw[ball color=yellow] (1  ,0.5) circle (.15);
\draw[ball color=yellow] (1.5,0.5) circle (.15);
\draw[ball color=yellow]  (2  ,0.5) circle (.15);
\draw[ball color=black]  (2.5,0.5) circle (.15);

\draw[ball color=black](0  ,1) circle (.15);
\draw[ball color=yellow](0.5,1) circle (.15);
\draw[ball color=yellow] (1  ,1) circle (.15);
\draw[ball color=yellow] (1.5,1) circle (.15);
\draw[ball color=yellow] (2  ,1) circle (.15);
\draw[ball color=black] (2.5,1) circle (.15);

\draw[ball color=black](0  ,1.5) circle (.15);
\draw[ball color=yellow](0.5,1.5) circle (.15);
\draw[ball color=yellow] (1  ,1.5) circle (.15);
\draw[ball color=yellow] (1.5,1.5) circle (.15);
\draw[ball color=yellow] (2  ,1.5) circle (.15);
\draw[ball color=black] (2.5,1.5) circle (.15);

\draw[ball color=black] (0  ,2) circle (.15);
\draw[ball color=black] (0.5,2) circle (.15);
\draw[ball color=black] (1  ,2) circle (.15);
\draw[ball color=black] (1.5,2) circle (.15);
\draw[ball color=black] (2  ,2) circle (.15);
\draw[ball color=black] (2.5,2) circle (.15);

\draw[ball color=black] (0  ,2.5) circle (.15);
\draw[ball color=black] (0.5,2.5) circle (.15);
\draw[ball color=black] (1  ,2.5) circle (.15);
\draw[ball color=black] (1.5,2.5) circle (.15);
\draw[ball color=black] (2  ,2.5) circle (.15);
\draw[ball color=black] (2.5,2.5) circle (.15);
\end{tikzpicture}
&
\begin{tikzpicture}
\draw[ball color=black] (0  ,0) circle (.15);
\draw[ball color=black] (0.5,0) circle (.15);
\draw[ball color=yellow] (1  ,0) circle (.15);
\draw[ball color=yellow] (1.5,0) circle (.15);
\draw[ball color=yellow]  (2  ,0) circle (.15);
\draw[ball color=yellow]  (2.5,0) circle (.15);

\draw[ball color=black] (0  ,0.5) circle (.15);
\draw[ball color=black] (0.5,0.5) circle (.15);
\draw[ball color=yellow] (1  ,0.5) circle (.15);
\draw[ball color=yellow] (1.5,0.5) circle (.15);
\draw[ball color=yellow]  (2  ,0.5) circle (.15);
\draw[ball color=yellow]  (2.5,0.5) circle (.15);

\draw[ball color=black](0  ,1) circle (.15);
\draw[ball color=black](0.5,1) circle (.15);
\draw[ball color=yellow] (1  ,1) circle (.15);
\draw[ball color=yellow] (1.5,1) circle (.15);
\draw[ball color=yellow] (2  ,1) circle (.15);
\draw[ball color=yellow] (2.5,1) circle (.15);

\draw[ball color=black](0  ,1.5) circle (.15);
\draw[ball color=black](0.5,1.5) circle (.15);
\draw[ball color=yellow] (1  ,1.5) circle (.15);
\draw[ball color=yellow] (1.5,1.5) circle (.15);
\draw[ball color=yellow] (2  ,1.5) circle (.15);
\draw[ball color=yellow] (2.5,1.5) circle (.15);

\draw[ball color=black] (0  ,2) circle (.15);
\draw[ball color=black] (0.5,2) circle (.15);
\draw[ball color=black] (1  ,2) circle (.15);
\draw[ball color=black] (1.5,2) circle (.15);
\draw[ball color=black] (2  ,2) circle (.15);
\draw[ball color=black] (2.5,2) circle (.15);

\draw[ball color=black] (0  ,2.5) circle (.15);
\draw[ball color=black] (0.5,2.5) circle (.15);
\draw[ball color=black] (1  ,2.5) circle (.15);
\draw[ball color=black] (1.5,2.5) circle (.15);
\draw[ball color=black] (2  ,2.5) circle (.15);
\draw[ball color=black] (2.5,2.5) circle (.15);
\end{tikzpicture}
\end{tabular}

%% file: figure2.tex
\begin{tabular}{c|c}
\begin{tikzpicture}
\draw[ball color=black] (0  ,0) circle (.15);
\draw[ball color=black] (0.5,0) circle (.15);
\draw[ball color=black] (1  ,0) circle (.15);
\draw[ball color=black] (1.5,0) circle (.15);
\draw[ball color=black]  (2  ,0) circle (.15);
\draw[ball color=black]  (2.5,0) circle (.15);

\draw[ball color=red] (0  ,0.5) circle (.15);
\draw[ball color=red] (0.5,0.5) circle (.15);
\draw[ball color=red] (1  ,0.5) circle (.15);
\draw[ball color=red] (1.5,0.5) circle (.15);
\draw[ball color=red]  (2  ,0.5) circle (.15);
\draw[ball color=black]  (2.5,0.5) circle (.15);

\draw[ball color=red](0  ,1) circle (.15);
\draw[ball color=red](0.5,1) circle (.15);
\draw[ball color=red] (1  ,1) circle (.15);
\draw[ball color=red] (1.5,1) circle (.15);
\draw[ball color=red] (2  ,1) circle (.15);
\draw[ball color=black] (2.5,1) circle (.15);

\draw[ball color=red](0  ,1.5) circle (.15);
\draw[ball color=red](0.5,1.5) circle (.15);
\draw[ball color=red] (1  ,1.5) circle (.15);
\draw[ball color=red] (1.5,1.5) circle (.15);
\draw[ball color=red] (2  ,1.5) circle (.15);
\draw[ball color=black] (2.5,1.5) circle (.15);

\draw[ball color=red] (0  ,2) circle (.15);
\draw[ball color=red] (0.5,2) circle (.15);
\draw[ball color=red] (1  ,2) circle (.15);
\draw[ball color=red] (1.5,2) circle (.15);
\draw[ball color=red] (2  ,2) circle (.15);
\draw[ball color=black] (2.5,2) circle (.15);

\draw[ball color=red] (0  ,2.5) circle (.15);
\draw[ball color=red] (0.5,2.5) circle (.15);
\draw[ball color=red] (1  ,2.5) circle (.15);
\draw[ball color=red] (1.5,2.5) circle (.15);
\draw[ball color=red] (2  ,2.5) circle (.15);
\draw[ball color=black] (2.5,2.5) circle (.15);
\end{tikzpicture}
&
\begin{tikzpicture}
\draw[ball color=black] (0  ,0) circle (.15);
\draw[ball color=black] (0.5,0) circle (.15);
\draw[ball color=black] (1  ,0) circle (.15);
\draw[ball color=black] (1.5,0) circle (.15);
\draw[ball color=black]  (2  ,0) circle (.15);
\draw[ball color=black]  (2.5,0) circle (.15);

\draw[ball color=black] (0  ,0.5) circle (.15);
\draw[ball color=red] (0.5,0.5) circle (.15);
\draw[ball color=red] (1  ,0.5) circle (.15);
\draw[ball color=red] (1.5,0.5) circle (.15);
\draw[ball color=red]  (2  ,0.5) circle (.15);
\draw[ball color=red]  (2.5,0.5) circle (.15);

\draw[ball color=black](0  ,1) circle (.15);
\draw[ball color=red](0.5,1) circle (.15);
\draw[ball color=red] (1  ,1) circle (.15);
\draw[ball color=red] (1.5,1) circle (.15);
\draw[ball color=red] (2  ,1) circle (.15);
\draw[ball color=red] (2.5,1) circle (.15);

\draw[ball color=black](0  ,1.5) circle (.15);
\draw[ball color=red](0.5,1.5) circle (.15);
\draw[ball color=red] (1  ,1.5) circle (.15);
\draw[ball color=red] (1.5,1.5) circle (.15);
\draw[ball color=red] (2  ,1.5) circle (.15);
\draw[ball color=red] (2.5,1.5) circle (.15);

\draw[ball color=black] (0  ,2) circle (.15);
\draw[ball color=red] (0.5,2) circle (.15);
\draw[ball color=red] (1  ,2) circle (.15);
\draw[ball color=red] (1.5,2) circle (.15);
\draw[ball color=red] (2  ,2) circle (.15);
\draw[ball color=red] (2.5,2) circle (.15);

\draw[ball color=black] (0  ,2.5) circle (.15);
\draw[ball color=red] (0.5,2.5) circle (.15);
\draw[ball color=red] (1  ,2.5) circle (.15);
\draw[ball color=red] (1.5,2.5) circle (.15);
\draw[ball color=red] (2  ,2.5) circle (.15);
\draw[ball color=red] (2.5,2.5) circle (.15);
\end{tikzpicture}
\\[0.5cm]
\begin{tikzpicture}
\draw[ball color=red] (0  ,0) circle (.15);
\draw[ball color=red] (0.5,0) circle (.15);
\draw[ball color=red] (1  ,0) circle (.15);
\draw[ball color=red] (1.5,0) circle (.15);
\draw[ball color=red]  (2  ,0) circle (.15);
\draw[ball color=black]  (2.5,0) circle (.15);

\draw[ball color=red] (0  ,0.5) circle (.15);
\draw[ball color=red] (0.5,0.5) circle (.15);
\draw[ball color=red] (1  ,0.5) circle (.15);
\draw[ball color=red] (1.5,0.5) circle (.15);
\draw[ball color=red]  (2  ,0.5) circle (.15);
\draw[ball color=black]  (2.5,0.5) circle (.15);

\draw[ball color=red](0  ,1) circle (.15);
\draw[ball color=red](0.5,1) circle (.15);
\draw[ball color=red] (1  ,1) circle (.15);
\draw[ball color=red] (1.5,1) circle (.15);
\draw[ball color=red] (2  ,1) circle (.15);
\draw[ball color=black] (2.5,1) circle (.15);

\draw[ball color=red](0  ,1.5) circle (.15);
\draw[ball color=red](0.5,1.5) circle (.15);
\draw[ball color=red] (1  ,1.5) circle (.15);
\draw[ball color=red] (1.5,1.5) circle (.15);
\draw[ball color=red] (2  ,1.5) circle (.15);
\draw[ball color=black] (2.5,1.5) circle (.15);

\draw[ball color=red] (0  ,2) circle (.15);
\draw[ball color=red] (0.5,2) circle (.15);
\draw[ball color=red] (1  ,2) circle (.15);
\draw[ball color=red] (1.5,2) circle (.15);
\draw[ball color=red] (2  ,2) circle (.15);
\draw[ball color=black] (2.5,2) circle (.15);

\draw[ball color=black] (0  ,2.5) circle (.15);
\draw[ball color=black] (0.5,2.5) circle (.15);
\draw[ball color=black] (1  ,2.5) circle (.15);
\draw[ball color=black] (1.5,2.5) circle (.15);
\draw[ball color=black] (2  ,2.5) circle (.15);
\draw[ball color=black] (2.5,2.5) circle (.15);
\end{tikzpicture}
&
\begin{tikzpicture}
\draw[ball color=black] (0  ,0) circle (.15);
\draw[ball color=red] (0.5,0) circle (.15);
\draw[ball color=red] (1  ,0) circle (.15);
\draw[ball color=red] (1.5,0) circle (.15);
\draw[ball color=red]  (2  ,0) circle (.15);
\draw[ball color=red]  (2.5,0) circle (.15);

\draw[ball color=black] (0  ,0.5) circle (.15);
\draw[ball color=red] (0.5,0.5) circle (.15);
\draw[ball color=red] (1  ,0.5) circle (.15);
\draw[ball color=red] (1.5,0.5) circle (.15);
\draw[ball color=red]  (2  ,0.5) circle (.15);
\draw[ball color=red]  (2.5,0.5) circle (.15);

\draw[ball color=black](0  ,1) circle (.15);
\draw[ball color=red](0.5,1) circle (.15);
\draw[ball color=red] (1  ,1) circle (.15);
\draw[ball color=red] (1.5,1) circle (.15);
\draw[ball color=red] (2  ,1) circle (.15);
\draw[ball color=red] (2.5,1) circle (.15);

\draw[ball color=black](0  ,1.5) circle (.15);
\draw[ball color=red](0.5,1.5) circle (.15);
\draw[ball color=red] (1  ,1.5) circle (.15);
\draw[ball color=red] (1.5,1.5) circle (.15);
\draw[ball color=red] (2  ,1.5) circle (.15);
\draw[ball color=red] (2.5,1.5) circle (.15);

\draw[ball color=black] (0  ,2) circle (.15);
\draw[ball color=red] (0.5,2) circle (.15);
\draw[ball color=red] (1  ,2) circle (.15);
\draw[ball color=red] (1.5,2) circle (.15);
\draw[ball color=red] (2  ,2) circle (.15);
\draw[ball color=red] (2.5,2) circle (.15);

\draw[ball color=black] (0  ,2.5) circle (.15);
\draw[ball color=black] (0.5,2.5) circle (.15);
\draw[ball color=black] (1  ,2.5) circle (.15);
\draw[ball color=black] (1.5,2.5) circle (.15);
\draw[ball color=black] (2  ,2.5) circle (.15);
\draw[ball color=black] (2.5,2.5) circle (.15);
\end{tikzpicture}
\end{tabular}

%% file: figure4.tex
\begin{tabular}{cc|cc}
\begin{tikzpicture}
\draw[ball color=black] (0  ,0) circle (.15);
\draw[ball color=black] (0.5,0) circle (.15);
\draw[ball color=black] (1  ,0) circle (.15);
\draw[ball color=black] (1.5,0) circle (.15);
\draw[ball color=black]  (2  ,0) circle (.15);
\draw[ball color=black]  (2.5,0) circle (.15);

\draw[ball color=red] (0  ,0.5) circle (.15);
\draw[ball color=red] (0.5,0.5) circle (.15);
\draw[ball color=red] (1  ,0.5) circle (.15);
\draw[ball color=red] (1.5,0.5) circle (.15);
\draw[ball color=red]  (2  ,0.5) circle (.15);
\draw[ball color=black]  (2.5,0.5) circle (.15);

\draw[ball color=blue](0  ,1) circle (.15);
\draw[ball color=blue](0.5,1) circle (.15);
\draw[ball color=blue] (1  ,1) circle (.15);
\draw[ball color=blue] (1.5,1) circle (.15);
\draw[ball color=red] (2  ,1) circle (.15);
\draw[ball color=black] (2.5,1) circle (.15);

\draw[ball color=blue](0  ,1.5) circle (.15);
\draw[ball color=blue](0.5,1.5) circle (.15);
\draw[ball color=blue] (1  ,1.5) circle (.15);
\draw[ball color=blue] (1.5,1.5) circle (.15);
\draw[ball color=red] (2  ,1.5) circle (.15);
\draw[ball color=black] (2.5,1.5) circle (.15);

\draw[ball color=blue] (0  ,2) circle (.15);
\draw[ball color=blue] (0.5,2) circle (.15);
\draw[ball color=blue] (1  ,2) circle (.15);
\draw[ball color=blue] (1.5,2) circle (.15);
\draw[ball color=red] (2  ,2) circle (.15);
\draw[ball color=black] (2.5,2) circle (.15);

\draw[ball color=blue] (0  ,2.5) circle (.15);
\draw[ball color=blue] (0.5,2.5) circle (.15);
\draw[ball color=blue] (1  ,2.5) circle (.15);
\draw[ball color=blue] (1.5,2.5) circle (.15);
\draw[ball color=red] (2  ,2.5) circle (.15);
\draw[ball color=black] (2.5,2.5) circle (.15);
\end{tikzpicture}
&
\begin{tikzpicture}
\draw[ball color=black] (0  ,0) circle (.15);
\draw[ball color=black] (0.5,0) circle (.15);
\draw[ball color=black] (1  ,0) circle (.15);
\draw[ball color=black] (1.5,0) circle (.15);
\draw[ball color=black]  (2  ,0) circle (.15);
\draw[ball color=black]  (2.5,0) circle (.15);

\draw[ball color=red] (0  ,0.5) circle (.15);
\draw[ball color=red] (0.5,0.5) circle (.15);
\draw[ball color=red] (1  ,0.5) circle (.15);
\draw[ball color=red] (1.5,0.5) circle (.15);
\draw[ball color=red]  (2  ,0.5) circle (.15);
\draw[ball color=black]  (2.5,0.5) circle (.15);

\draw[ball color=red](0  ,1) circle (.15);
\draw[ball color=blue](0.5,1) circle (.15);
\draw[ball color=blue] (1  ,1) circle (.15);
\draw[ball color=blue] (1.5,1) circle (.15);
\draw[ball color=blue] (2  ,1) circle (.15);
\draw[ball color=black] (2.5,1) circle (.15);

\draw[ball color=red](0  ,1.5) circle (.15);
\draw[ball color=blue](0.5,1.5) circle (.15);
\draw[ball color=blue] (1  ,1.5) circle (.15);
\draw[ball color=blue] (1.5,1.5) circle (.15);
\draw[ball color=blue] (2  ,1.5) circle (.15);
\draw[ball color=black] (2.5,1.5) circle (.15);

\draw[ball color=red] (0  ,2) circle (.15);
\draw[ball color=blue] (0.5,2) circle (.15);
\draw[ball color=blue] (1  ,2) circle (.15);
\draw[ball color=blue] (1.5,2) circle (.15);
\draw[ball color=blue] (2  ,2) circle (.15);
\draw[ball color=black] (2.5,2) circle (.15);

\draw[ball color=red] (0  ,2.5) circle (.15);
\draw[ball color=blue] (0.5,2.5) circle (.15);
\draw[ball color=blue] (1  ,2.5) circle (.15);
\draw[ball color=blue] (1.5,2.5) circle (.15);
\draw[ball color=blue] (2  ,2.5) circle (.15);
\draw[ball color=black] (2.5,2.5) circle (.15);
\end{tikzpicture}
&
\begin{tikzpicture}
\draw[ball color=black] (0  ,0) circle (.15);
\draw[ball color=black] (0.5,0) circle (.15);
\draw[ball color=black] (1  ,0) circle (.15);
\draw[ball color=black] (1.5,0) circle (.15);
\draw[ball color=black]  (2  ,0) circle (.15);
\draw[ball color=black]  (2.5,0) circle (.15);

\draw[ball color=black] (0  ,0.5) circle (.15);
\draw[ball color=red] (0.5,0.5) circle (.15);
\draw[ball color=red] (1  ,0.5) circle (.15);
\draw[ball color=red] (1.5,0.5) circle (.15);
\draw[ball color=red]  (2  ,0.5) circle (.15);
\draw[ball color=red]  (2.5,0.5) circle (.15);

\draw[ball color=black](0  ,1) circle (.15);
\draw[ball color=blue](0.5,1) circle (.15);
\draw[ball color=blue] (1  ,1) circle (.15);
\draw[ball color=blue] (1.5,1) circle (.15);
\draw[ball color=blue] (2  ,1) circle (.15);
\draw[ball color=red] (2.5,1) circle (.15);

\draw[ball color=black](0  ,1.5) circle (.15);
\draw[ball color=blue](0.5,1.5) circle (.15);
\draw[ball color=blue] (1  ,1.5) circle (.15);
\draw[ball color=blue] (1.5,1.5) circle (.15);
\draw[ball color=blue] (2  ,1.5) circle (.15);
\draw[ball color=red] (2.5,1.5) circle (.15);

\draw[ball color=black] (0  ,2) circle (.15);
\draw[ball color=blue] (0.5,2) circle (.15);
\draw[ball color=blue] (1  ,2) circle (.15);
\draw[ball color=blue] (1.5,2) circle (.15);
\draw[ball color=blue] (2  ,2) circle (.15);
\draw[ball color=red] (2.5,2) circle (.15);

\draw[ball color=black] (0  ,2.5) circle (.15);
\draw[ball color=blue] (0.5,2.5) circle (.15);
\draw[ball color=blue] (1  ,2.5) circle (.15);
\draw[ball color=blue] (1.5,2.5) circle (.15);
\draw[ball color=blue] (2  ,2.5) circle (.15);
\draw[ball color=red] (2.5,2.5) circle (.15);
\end{tikzpicture}
&
\begin{tikzpicture}
\draw[ball color=black] (0  ,0) circle (.15);
\draw[ball color=black] (0.5,0) circle (.15);
\draw[ball color=black] (1  ,0) circle (.15);
\draw[ball color=black] (1.5,0) circle (.15);
\draw[ball color=black]  (2  ,0) circle (.15);
\draw[ball color=black]  (2.5,0) circle (.15);

\draw[ball color=black] (0  ,0.5) circle (.15);
\draw[ball color=red] (0.5,0.5) circle (.15);
\draw[ball color=red] (1  ,0.5) circle (.15);
\draw[ball color=red] (1.5,0.5) circle (.15);
\draw[ball color=red]  (2  ,0.5) circle (.15);
\draw[ball color=red]  (2.5,0.5) circle (.15);

\draw[ball color=black](0  ,1) circle (.15);
\draw[ball color=red](0.5,1) circle (.15);
\draw[ball color=blue] (1  ,1) circle (.15);
\draw[ball color=blue] (1.5,1) circle (.15);
\draw[ball color=blue] (2  ,1) circle (.15);
\draw[ball color=blue] (2.5,1) circle (.15);

\draw[ball color=black](0  ,1.5) circle (.15);
\draw[ball color=red](0.5,1.5) circle (.15);
\draw[ball color=blue] (1  ,1.5) circle (.15);
\draw[ball color=blue] (1.5,1.5) circle (.15);
\draw[ball color=blue] (2  ,1.5) circle (.15);
\draw[ball color=blue] (2.5,1.5) circle (.15);

\draw[ball color=black] (0  ,2) circle (.15);
\draw[ball color=red] (0.5,2) circle (.15);
\draw[ball color=blue] (1  ,2) circle (.15);
\draw[ball color=blue] (1.5,2) circle (.15);
\draw[ball color=blue] (2  ,2) circle (.15);
\draw[ball color=blue] (2.5,2) circle (.15);

\draw[ball color=black] (0  ,2.5) circle (.15);
\draw[ball color=red] (0.5,2.5) circle (.15);
\draw[ball color=blue] (1  ,2.5) circle (.15);
\draw[ball color=blue] (1.5,2.5) circle (.15);
\draw[ball color=blue] (2  ,2.5) circle (.15);
\draw[ball color=blue] (2.5,2.5) circle (.15);
\end{tikzpicture}
\\[0.5cm]
\begin{tikzpicture}
\draw[ball color=black] (0  ,0) circle (.15);
\draw[ball color=black] (0.5,0) circle (.15);
\draw[ball color=black] (1  ,0) circle (.15);
\draw[ball color=black] (1.5,0) circle (.15);
\draw[ball color=black]  (2  ,0) circle (.15);
\draw[ball color=black]  (2.5,0) circle (.15);

\draw[ball color=blue] (0  ,0.5) circle (.15);
\draw[ball color=blue] (0.5,0.5) circle (.15);
\draw[ball color=blue] (1  ,0.5) circle (.15);
\draw[ball color=blue] (1.5,0.5) circle (.15);
\draw[ball color=red]  (2  ,0.5) circle (.15);
\draw[ball color=black]  (2.5,0.5) circle (.15);

\draw[ball color=blue](0  ,1) circle (.15);
\draw[ball color=blue](0.5,1) circle (.15);
\draw[ball color=blue] (1  ,1) circle (.15);
\draw[ball color=blue] (1.5,1) circle (.15);
\draw[ball color=red] (2  ,1) circle (.15);
\draw[ball color=black] (2.5,1) circle (.15);

\draw[ball color=blue](0  ,1.5) circle (.15);
\draw[ball color=blue](0.5,1.5) circle (.15);
\draw[ball color=blue] (1  ,1.5) circle (.15);
\draw[ball color=blue] (1.5,1.5) circle (.15);
\draw[ball color=red] (2  ,1.5) circle (.15);
\draw[ball color=black] (2.5,1.5) circle (.15);

\draw[ball color=blue] (0  ,2) circle (.15);
\draw[ball color=blue] (0.5,2) circle (.15);
\draw[ball color=blue] (1  ,2) circle (.15);
\draw[ball color=blue] (1.5,2) circle (.15);
\draw[ball color=red] (2  ,2) circle (.15);
\draw[ball color=black] (2.5,2) circle (.15);

\draw[ball color=red] (0  ,2.5) circle (.15);
\draw[ball color=red] (0.5,2.5) circle (.15);
\draw[ball color=red] (1  ,2.5) circle (.15);
\draw[ball color=red] (1.5,2.5) circle (.15);
\draw[ball color=red] (2  ,2.5) circle (.15);
\draw[ball color=black] (2.5,2.5) circle (.15);
\end{tikzpicture}
&
\begin{tikzpicture}
\draw[ball color=black] (0  ,0) circle (.15);
\draw[ball color=black] (0.5,0) circle (.15);
\draw[ball color=black] (1  ,0) circle (.15);
\draw[ball color=black] (1.5,0) circle (.15);
\draw[ball color=black]  (2  ,0) circle (.15);
\draw[ball color=black]  (2.5,0) circle (.15);

\draw[ball color=red] (0  ,0.5) circle (.15);
\draw[ball color=blue] (0.5,0.5) circle (.15);
\draw[ball color=blue] (1  ,0.5) circle (.15);
\draw[ball color=blue] (1.5,0.5) circle (.15);
\draw[ball color=blue]  (2  ,0.5) circle (.15);
\draw[ball color=black]  (2.5,0.5) circle (.15);

\draw[ball color=red](0  ,1) circle (.15);
\draw[ball color=blue](0.5,1) circle (.15);
\draw[ball color=blue] (1  ,1) circle (.15);
\draw[ball color=blue] (1.5,1) circle (.15);
\draw[ball color=blue] (2  ,1) circle (.15);
\draw[ball color=black] (2.5,1) circle (.15);

\draw[ball color=red](0  ,1.5) circle (.15);
\draw[ball color=blue](0.5,1.5) circle (.15);
\draw[ball color=blue] (1  ,1.5) circle (.15);
\draw[ball color=blue] (1.5,1.5) circle (.15);
\draw[ball color=blue] (2  ,1.5) circle (.15);
\draw[ball color=black] (2.5,1.5) circle (.15);

\draw[ball color=red] (0  ,2) circle (.15);
\draw[ball color=blue] (0.5,2) circle (.15);
\draw[ball color=blue] (1  ,2) circle (.15);
\draw[ball color=blue] (1.5,2) circle (.15);
\draw[ball color=blue] (2  ,2) circle (.15);
\draw[ball color=black] (2.5,2) circle (.15);

\draw[ball color=red] (0  ,2.5) circle (.15);
\draw[ball color=red] (0.5,2.5) circle (.15);
\draw[ball color=red] (1  ,2.5) circle (.15);
\draw[ball color=red] (1.5,2.5) circle (.15);
\draw[ball color=red] (2  ,2.5) circle (.15);
\draw[ball color=black] (2.5,2.5) circle (.15);
\end{tikzpicture}
&
\begin{tikzpicture}
\draw[ball color=black] (0  ,0) circle (.15);
\draw[ball color=black] (0.5,0) circle (.15);
\draw[ball color=black] (1  ,0) circle (.15);
\draw[ball color=black] (1.5,0) circle (.15);
\draw[ball color=black]  (2  ,0) circle (.15);
\draw[ball color=black]  (2.5,0) circle (.15);

\draw[ball color=black] (0  ,0.5) circle (.15);
\draw[ball color=blue] (0.5,0.5) circle (.15);
\draw[ball color=blue] (1  ,0.5) circle (.15);
\draw[ball color=blue] (1.5,0.5) circle (.15);
\draw[ball color=blue]  (2  ,0.5) circle (.15);
\draw[ball color=red]  (2.5,0.5) circle (.15);

\draw[ball color=black](0  ,1) circle (.15);
\draw[ball color=blue](0.5,1) circle (.15);
\draw[ball color=blue] (1  ,1) circle (.15);
\draw[ball color=blue] (1.5,1) circle (.15);
\draw[ball color=blue] (2  ,1) circle (.15);
\draw[ball color=red] (2.5,1) circle (.15);

\draw[ball color=black](0  ,1.5) circle (.15);
\draw[ball color=blue](0.5,1.5) circle (.15);
\draw[ball color=blue] (1  ,1.5) circle (.15);
\draw[ball color=blue] (1.5,1.5) circle (.15);
\draw[ball color=blue] (2  ,1.5) circle (.15);
\draw[ball color=red] (2.5,1.5) circle (.15);

\draw[ball color=black] (0  ,2) circle (.15);
\draw[ball color=blue] (0.5,2) circle (.15);
\draw[ball color=blue] (1  ,2) circle (.15);
\draw[ball color=blue] (1.5,2) circle (.15);
\draw[ball color=blue] (2  ,2) circle (.15);
\draw[ball color=red] (2.5,2) circle (.15);

\draw[ball color=black] (0  ,2.5) circle (.15);
\draw[ball color=red] (0.5,2.5) circle (.15);
\draw[ball color=red] (1  ,2.5) circle (.15);
\draw[ball color=red] (1.5,2.5) circle (.15);
\draw[ball color=red] (2  ,2.5) circle (.15);
\draw[ball color=red] (2.5,2.5) circle (.15);
\end{tikzpicture}
&
\begin{tikzpicture}
\draw[ball color=black] (0  ,0) circle (.15);
\draw[ball color=black] (0.5,0) circle (.15);
\draw[ball color=black] (1  ,0) circle (.15);
\draw[ball color=black] (1.5,0) circle (.15);
\draw[ball color=black]  (2  ,0) circle (.15);
\draw[ball color=black]  (2.5,0) circle (.15);

\draw[ball color=black] (0  ,0.5) circle (.15);
\draw[ball color=red] (0.5,0.5) circle (.15);
\draw[ball color=blue] (1  ,0.5) circle (.15);
\draw[ball color=blue] (1.5,0.5) circle (.15);
\draw[ball color=blue]  (2  ,0.5) circle (.15);
\draw[ball color=blue]  (2.5,0.5) circle (.15);

\draw[ball color=black](0  ,1) circle (.15);
\draw[ball color=red](0.5,1) circle (.15);
\draw[ball color=blue] (1  ,1) circle (.15);
\draw[ball color=blue] (1.5,1) circle (.15);
\draw[ball color=blue] (2  ,1) circle (.15);
\draw[ball color=blue] (2.5,1) circle (.15);

\draw[ball color=black](0  ,1.5) circle (.15);
\draw[ball color=red](0.5,1.5) circle (.15);
\draw[ball color=blue] (1  ,1.5) circle (.15);
\draw[ball color=blue] (1.5,1.5) circle (.15);
\draw[ball color=blue] (2  ,1.5) circle (.15);
\draw[ball color=blue] (2.5,1.5) circle (.15);

\draw[ball color=black] (0  ,2) circle (.15);
\draw[ball color=red] (0.5,2) circle (.15);
\draw[ball color=blue] (1  ,2) circle (.15);
\draw[ball color=blue] (1.5,2) circle (.15);
\draw[ball color=blue] (2  ,2) circle (.15);
\draw[ball color=blue] (2.5,2) circle (.15);

\draw[ball color=black] (0  ,2.5) circle (.15);
\draw[ball color=red] (0.5,2.5) circle (.15);
\draw[ball color=red] (1  ,2.5) circle (.15);
\draw[ball color=red] (1.5,2.5) circle (.15);
\draw[ball color=red] (2  ,2.5) circle (.15);
\draw[ball color=red] (2.5,2.5) circle (.15);
\end{tikzpicture}
\\[0.25cm] \hline & & &  \\[0.25cm]
\begin{tikzpicture}
\draw[ball color=red] (0  ,0) circle (.15);
\draw[ball color=red] (0.5,0) circle (.15);
\draw[ball color=red] (1  ,0) circle (.15);
\draw[ball color=red] (1.5,0) circle (.15);
\draw[ball color=red]  (2  ,0) circle (.15);
\draw[ball color=black]  (2.5,0) circle (.15);

\draw[ball color=blue] (0  ,0.5) circle (.15);
\draw[ball color=blue] (0.5,0.5) circle (.15);
\draw[ball color=blue] (1  ,0.5) circle (.15);
\draw[ball color=blue] (1.5,0.5) circle (.15);
\draw[ball color=red]  (2  ,0.5) circle (.15);
\draw[ball color=black]  (2.5,0.5) circle (.15);

\draw[ball color=blue](0  ,1) circle (.15);
\draw[ball color=blue](0.5,1) circle (.15);
\draw[ball color=blue] (1  ,1) circle (.15);
\draw[ball color=blue] (1.5,1) circle (.15);
\draw[ball color=red] (2  ,1) circle (.15);
\draw[ball color=black] (2.5,1) circle (.15);

\draw[ball color=blue](0  ,1.5) circle (.15);
\draw[ball color=blue](0.5,1.5) circle (.15);
\draw[ball color=blue] (1  ,1.5) circle (.15);
\draw[ball color=blue] (1.5,1.5) circle (.15);
\draw[ball color=red] (2  ,1.5) circle (.15);
\draw[ball color=black] (2.5,1.5) circle (.15);

\draw[ball color=blue] (0  ,2) circle (.15);
\draw[ball color=blue] (0.5,2) circle (.15);
\draw[ball color=blue] (1  ,2) circle (.15);
\draw[ball color=blue] (1.5,2) circle (.15);
\draw[ball color=red] (2  ,2) circle (.15);
\draw[ball color=black] (2.5,2) circle (.15);

\draw[ball color=black] (0  ,2.5) circle (.15);
\draw[ball color=black] (0.5,2.5) circle (.15);
\draw[ball color=black] (1  ,2.5) circle (.15);
\draw[ball color=black] (1.5,2.5) circle (.15);
\draw[ball color=black] (2  ,2.5) circle (.15);
\draw[ball color=black] (2.5,2.5) circle (.15);
\end{tikzpicture}
&
\begin{tikzpicture}
\draw[ball color=red] (0  ,0) circle (.15);
\draw[ball color=red] (0.5,0) circle (.15);
\draw[ball color=red] (1  ,0) circle (.15);
\draw[ball color=red] (1.5,0) circle (.15);
\draw[ball color=red]  (2  ,0) circle (.15);
\draw[ball color=black]  (2.5,0) circle (.15);

\draw[ball color=red] (0  ,0.5) circle (.15);
\draw[ball color=blue] (0.5,0.5) circle (.15);
\draw[ball color=blue] (1  ,0.5) circle (.15);
\draw[ball color=blue] (1.5,0.5) circle (.15);
\draw[ball color=blue]  (2  ,0.5) circle (.15);
\draw[ball color=black]  (2.5,0.5) circle (.15);

\draw[ball color=red](0  ,1) circle (.15);
\draw[ball color=blue](0.5,1) circle (.15);
\draw[ball color=blue] (1  ,1) circle (.15);
\draw[ball color=blue] (1.5,1) circle (.15);
\draw[ball color=blue] (2  ,1) circle (.15);
\draw[ball color=black] (2.5,1) circle (.15);

\draw[ball color=red](0  ,1.5) circle (.15);
\draw[ball color=blue](0.5,1.5) circle (.15);
\draw[ball color=blue] (1  ,1.5) circle (.15);
\draw[ball color=blue] (1.5,1.5) circle (.15);
\draw[ball color=blue] (2  ,1.5) circle (.15);
\draw[ball color=black] (2.5,1.5) circle (.15);

\draw[ball color=red] (0  ,2) circle (.15);
\draw[ball color=blue] (0.5,2) circle (.15);
\draw[ball color=blue] (1  ,2) circle (.15);
\draw[ball color=blue] (1.5,2) circle (.15);
\draw[ball color=blue] (2  ,2) circle (.15);
\draw[ball color=black] (2.5,2) circle (.15);

\draw[ball color=black] (0  ,2.5) circle (.15);
\draw[ball color=black] (0.5,2.5) circle (.15);
\draw[ball color=black] (1  ,2.5) circle (.15);
\draw[ball color=black] (1.5,2.5) circle (.15);
\draw[ball color=black] (2  ,2.5) circle (.15);
\draw[ball color=black] (2.5,2.5) circle (.15);
\end{tikzpicture}
&
\begin{tikzpicture}
\draw[ball color=black] (0  ,0) circle (.15);
\draw[ball color=red] (0.5,0) circle (.15);
\draw[ball color=red] (1  ,0) circle (.15);
\draw[ball color=red] (1.5,0) circle (.15);
\draw[ball color=red]  (2  ,0) circle (.15);
\draw[ball color=red]  (2.5,0) circle (.15);

\draw[ball color=black] (0  ,0.5) circle (.15);
\draw[ball color=blue] (0.5,0.5) circle (.15);
\draw[ball color=blue] (1  ,0.5) circle (.15);
\draw[ball color=blue] (1.5,0.5) circle (.15);
\draw[ball color=blue]  (2  ,0.5) circle (.15);
\draw[ball color=red]  (2.5,0.5) circle (.15);

\draw[ball color=black](0  ,1) circle (.15);
\draw[ball color=blue](0.5,1) circle (.15);
\draw[ball color=blue] (1  ,1) circle (.15);
\draw[ball color=blue] (1.5,1) circle (.15);
\draw[ball color=blue] (2  ,1) circle (.15);
\draw[ball color=red] (2.5,1) circle (.15);

\draw[ball color=black](0  ,1.5) circle (.15);
\draw[ball color=blue](0.5,1.5) circle (.15);
\draw[ball color=blue] (1  ,1.5) circle (.15);
\draw[ball color=blue] (1.5,1.5) circle (.15);
\draw[ball color=blue] (2  ,1.5) circle (.15);
\draw[ball color=red] (2.5,1.5) circle (.15);

\draw[ball color=black] (0  ,2) circle (.15);
\draw[ball color=blue] (0.5,2) circle (.15);
\draw[ball color=blue] (1  ,2) circle (.15);
\draw[ball color=blue] (1.5,2) circle (.15);
\draw[ball color=blue] (2  ,2) circle (.15);
\draw[ball color=red] (2.5,2) circle (.15);

\draw[ball color=black] (0  ,2.5) circle (.15);
\draw[ball color=black] (0.5,2.5) circle (.15);
\draw[ball color=black] (1  ,2.5) circle (.15);
\draw[ball color=black] (1.5,2.5) circle (.15);
\draw[ball color=black] (2  ,2.5) circle (.15);
\draw[ball color=black] (2.5,2.5) circle (.15);
\end{tikzpicture}
&
\begin{tikzpicture}
\draw[ball color=black] (0  ,0) circle (.15);
\draw[ball color=red] (0.5,0) circle (.15);
\draw[ball color=red] (1  ,0) circle (.15);
\draw[ball color=red] (1.5,0) circle (.15);
\draw[ball color=red]  (2  ,0) circle (.15);
\draw[ball color=red]  (2.5,0) circle (.15);

\draw[ball color=black] (0  ,0.5) circle (.15);
\draw[ball color=red] (0.5,0.5) circle (.15);
\draw[ball color=blue] (1  ,0.5) circle (.15);
\draw[ball color=blue] (1.5,0.5) circle (.15);
\draw[ball color=blue]  (2  ,0.5) circle (.15);
\draw[ball color=blue]  (2.5,0.5) circle (.15);

\draw[ball color=black](0  ,1) circle (.15);
\draw[ball color=red](0.5,1) circle (.15);
\draw[ball color=blue] (1  ,1) circle (.15);
\draw[ball color=blue] (1.5,1) circle (.15);
\draw[ball color=blue] (2  ,1) circle (.15);
\draw[ball color=blue] (2.5,1) circle (.15);

\draw[ball color=black](0  ,1.5) circle (.15);
\draw[ball color=red](0.5,1.5) circle (.15);
\draw[ball color=blue] (1  ,1.5) circle (.15);
\draw[ball color=blue] (1.5,1.5) circle (.15);
\draw[ball color=blue] (2  ,1.5) circle (.15);
\draw[ball color=blue] (2.5,1.5) circle (.15);

\draw[ball color=black] (0  ,2) circle (.15);
\draw[ball color=red] (0.5,2) circle (.15);
\draw[ball color=blue] (1  ,2) circle (.15);
\draw[ball color=blue] (1.5,2) circle (.15);
\draw[ball color=blue] (2  ,2) circle (.15);
\draw[ball color=blue] (2.5,2) circle (.15);

\draw[ball color=black] (0  ,2.5) circle (.15);
\draw[ball color=black] (0.5,2.5) circle (.15);
\draw[ball color=black] (1  ,2.5) circle (.15);
\draw[ball color=black] (1.5,2.5) circle (.15);
\draw[ball color=black] (2  ,2.5) circle (.15);
\draw[ball color=black] (2.5,2.5) circle (.15);
\end{tikzpicture}
\\[0.5cm]
\begin{tikzpicture}
\draw[ball color=blue] (0  ,0) circle (.15);
\draw[ball color=blue] (0.5,0) circle (.15);
\draw[ball color=blue] (1  ,0) circle (.15);
\draw[ball color=blue] (1.5,0) circle (.15);
\draw[ball color=red]  (2  ,0) circle (.15);
\draw[ball color=black]  (2.5,0) circle (.15);

\draw[ball color=blue] (0  ,0.5) circle (.15);
\draw[ball color=blue] (0.5,0.5) circle (.15);
\draw[ball color=blue] (1  ,0.5) circle (.15);
\draw[ball color=blue] (1.5,0.5) circle (.15);
\draw[ball color=red]  (2  ,0.5) circle (.15);
\draw[ball color=black]  (2.5,0.5) circle (.15);

\draw[ball color=blue](0  ,1) circle (.15);
\draw[ball color=blue](0.5,1) circle (.15);
\draw[ball color=blue] (1  ,1) circle (.15);
\draw[ball color=blue] (1.5,1) circle (.15);
\draw[ball color=red] (2  ,1) circle (.15);
\draw[ball color=black] (2.5,1) circle (.15);

\draw[ball color=blue](0  ,1.5) circle (.15);
\draw[ball color=blue](0.5,1.5) circle (.15);
\draw[ball color=blue] (1  ,1.5) circle (.15);
\draw[ball color=blue] (1.5,1.5) circle (.15);
\draw[ball color=red] (2  ,1.5) circle (.15);
\draw[ball color=black] (2.5,1.5) circle (.15);

\draw[ball color=red] (0  ,2) circle (.15);
\draw[ball color=red] (0.5,2) circle (.15);
\draw[ball color=red] (1  ,2) circle (.15);
\draw[ball color=red] (1.5,2) circle (.15);
\draw[ball color=red] (2  ,2) circle (.15);
\draw[ball color=black] (2.5,2) circle (.15);

\draw[ball color=black] (0  ,2.5) circle (.15);
\draw[ball color=black] (0.5,2.5) circle (.15);
\draw[ball color=black] (1  ,2.5) circle (.15);
\draw[ball color=black] (1.5,2.5) circle (.15);
\draw[ball color=black] (2  ,2.5) circle (.15);
\draw[ball color=black] (2.5,2.5) circle (.15);
\end{tikzpicture}
&
\begin{tikzpicture}
\draw[ball color=red] (0  ,0) circle (.15);
\draw[ball color=blue] (0.5,0) circle (.15);
\draw[ball color=blue] (1  ,0) circle (.15);
\draw[ball color=blue] (1.5,0) circle (.15);
\draw[ball color=blue]  (2  ,0) circle (.15);
\draw[ball color=black]  (2.5,0) circle (.15);

\draw[ball color=red] (0  ,0.5) circle (.15);
\draw[ball color=blue] (0.5,0.5) circle (.15);
\draw[ball color=blue] (1  ,0.5) circle (.15);
\draw[ball color=blue] (1.5,0.5) circle (.15);
\draw[ball color=blue]  (2  ,0.5) circle (.15);
\draw[ball color=black]  (2.5,0.5) circle (.15);

\draw[ball color=red](0  ,1) circle (.15);
\draw[ball color=blue](0.5,1) circle (.15);
\draw[ball color=blue] (1  ,1) circle (.15);
\draw[ball color=blue] (1.5,1) circle (.15);
\draw[ball color=blue] (2  ,1) circle (.15);
\draw[ball color=black] (2.5,1) circle (.15);

\draw[ball color=red](0  ,1.5) circle (.15);
\draw[ball color=blue](0.5,1.5) circle (.15);
\draw[ball color=blue] (1  ,1.5) circle (.15);
\draw[ball color=blue] (1.5,1.5) circle (.15);
\draw[ball color=blue] (2  ,1.5) circle (.15);
\draw[ball color=black] (2.5,1.5) circle (.15);

\draw[ball color=red] (0  ,2) circle (.15);
\draw[ball color=red] (0.5,2) circle (.15);
\draw[ball color=red] (1  ,2) circle (.15);
\draw[ball color=red] (1.5,2) circle (.15);
\draw[ball color=red] (2  ,2) circle (.15);
\draw[ball color=black] (2.5,2) circle (.15);

\draw[ball color=black] (0  ,2.5) circle (.15);
\draw[ball color=black] (0.5,2.5) circle (.15);
\draw[ball color=black] (1  ,2.5) circle (.15);
\draw[ball color=black] (1.5,2.5) circle (.15);
\draw[ball color=black] (2  ,2.5) circle (.15);
\draw[ball color=black] (2.5,2.5) circle (.15);
\end{tikzpicture}
&
\begin{tikzpicture}
\draw[ball color=black] (0  ,0) circle (.15);
\draw[ball color=blue] (0.5,0) circle (.15);
\draw[ball color=blue] (1  ,0) circle (.15);
\draw[ball color=blue] (1.5,0) circle (.15);
\draw[ball color=blue]  (2  ,0) circle (.15);
\draw[ball color=red]  (2.5,0) circle (.15);

\draw[ball color=black] (0  ,0.5) circle (.15);
\draw[ball color=blue  ] (0.5,0.5) circle (.15);
\draw[ball color=blue ] (1  ,0.5) circle (.15);
\draw[ball color=blue ] (1.5,0.5) circle (.15);
\draw[ball color=blue ]  (2  ,0.5) circle (.15);
\draw[ball color=red ]  (2.5,0.5) circle (.15);

\draw[ball color=black](0  ,1) circle (.15);
\draw[ball color=blue  ](0.5,1) circle (.15);
\draw[ball color=blue ] (1  ,1) circle (.15);
\draw[ball color=blue ] (1.5,1) circle (.15);
\draw[ball color=blue ] (2  ,1) circle (.15);
\draw[ball color=red ] (2.5,1) circle (.15);

\draw[ball color=black](0  ,1.5) circle (.15);
\draw[ball color=blue  ](0.5,1.5) circle (.15);
\draw[ball color=blue ] (1  ,1.5) circle (.15);
\draw[ball color=blue ] (1.5,1.5) circle (.15);
\draw[ball color=blue ] (2  ,1.5) circle (.15);
\draw[ball color=red ] (2.5,1.5) circle (.15);

\draw[ball color=black] (0  ,2) circle (.15);
\draw[ball color=red  ] (0.5,2) circle (.15);
\draw[ball color=red ] (1  ,2) circle (.15);
\draw[ball color=red ] (1.5,2) circle (.15);
\draw[ball color=red ] (2  ,2) circle (.15);
\draw[ball color=red ] (2.5,2) circle (.15);

\draw[ball color=black] (0  ,2.5) circle (.15);
\draw[ball color=black] (0.5,2.5) circle (.15);
\draw[ball color=black] (1  ,2.5) circle (.15);
\draw[ball color=black] (1.5,2.5) circle (.15);
\draw[ball color=black] (2  ,2.5) circle (.15);
\draw[ball color=black] (2.5,2.5) circle (.15);
\end{tikzpicture}
&
\begin{tikzpicture}
\draw[ball color=black] (0  ,0) circle (.15);
\draw[ball color=red] (0.5,0) circle (.15);
\draw[ball color=blue] (1  ,0) circle (.15);
\draw[ball color=blue] (1.5,0) circle (.15);
\draw[ball color=blue]  (2  ,0) circle (.15);
\draw[ball color=blue]  (2.5,0) circle (.15);

\draw[ball color=black] (0  ,0.5) circle (.15);
\draw[ball color=red  ] (0.5,0.5) circle (.15);
\draw[ball color=blue ] (1  ,0.5) circle (.15);
\draw[ball color=blue ] (1.5,0.5) circle (.15);
\draw[ball color=blue ]  (2  ,0.5) circle (.15);
\draw[ball color=blue ]  (2.5,0.5) circle (.15);

\draw[ball color=black](0  ,1) circle (.15);
\draw[ball color=red  ](0.5,1) circle (.15);
\draw[ball color=blue ] (1  ,1) circle (.15);
\draw[ball color=blue ] (1.5,1) circle (.15);
\draw[ball color=blue ] (2  ,1) circle (.15);
\draw[ball color=blue ] (2.5,1) circle (.15);

\draw[ball color=black](0  ,1.5) circle (.15);
\draw[ball color=red  ](0.5,1.5) circle (.15);
\draw[ball color=blue ] (1  ,1.5) circle (.15);
\draw[ball color=blue ] (1.5,1.5) circle (.15);
\draw[ball color=blue ] (2  ,1.5) circle (.15);
\draw[ball color=blue ] (2.5,1.5) circle (.15);

\draw[ball color=black] (0  ,2) circle (.15);
\draw[ball color=red  ] (0.5,2) circle (.15);
\draw[ball color=red ] (1  ,2) circle (.15);
\draw[ball color=red ] (1.5,2) circle (.15);
\draw[ball color=red ] (2  ,2) circle (.15);
\draw[ball color=red ] (2.5,2) circle (.15);

\draw[ball color=black] (0  ,2.5) circle (.15);
\draw[ball color=black] (0.5,2.5) circle (.15);
\draw[ball color=black] (1  ,2.5) circle (.15);
\draw[ball color=black] (1.5,2.5) circle (.15);
\draw[ball color=black] (2  ,2.5) circle (.15);
\draw[ball color=black] (2.5,2.5) circle (.15);
\end{tikzpicture}
\end{tabular}